\documentclass{amsart}

\usepackage[left=33mm, right=33mm, top=25mm, bottom=25mm
]{geometry}

\usepackage[english]{babel}
\usepackage[utf8]{inputenc}
\usepackage{amsmath}
\usepackage{amsfonts}
\usepackage{mathrsfs}
\usepackage{mathtools}
\mathtoolsset{showonlyrefs}
\usepackage{enumerate}

\usepackage{esint}
\usepackage[colorlinks=true, allcolors=red, linktocpage=true]{hyperref}
\usepackage{amssymb}
\usepackage{graphicx}
\usepackage{commath}
\usepackage{cases}
\usepackage{bm}
\usepackage{dsfont}
\usepackage{xcolor}
\usepackage{verbatim}
\usepackage{todonotes}
\usepackage{aligned-overset}

\usepackage{amsthm}

\theoremstyle{plain}

\newtheorem{thm}{Theorem}[section]
\newtheorem{prop}[thm]{Proposition}
\newtheorem{lem}[thm]{Lemma}

\newtheorem{cor}[thm]{Corollary}

\numberwithin{equation}{section}

\theoremstyle{definition}

\newtheorem{defn}[thm]{Definition}

\newtheorem{rem}[thm]{Remark}

\makeatletter
\newcommand*{\rom}[1]{\expandafter\@slowromancap\romannumeral #1@}
\makeatother

\newcommand{\B}{\mathbb{B}}
\newcommand{\C}{\mathbb{C}}

\newcommand{\N}{\mathbb{N}}

\newcommand{\R}{\mathbb{R}}

\newcommand{\Z}{\mathbb{Z}}

\newcommand{\calA}{\mathcal{A}}

\newcommand{\calL}{\mathcal{L}}

\def\dist{\mathop\mathrm{dist}} 
\def\supp{\mathop\mathrm{supp}} 
\def\diam{\mathop\mathrm{diam}} 
\def\Stop{\mathop\mathrm{Stop}} 
 
\def\Top{\mathop\mathrm{Top}}
\def\BWGL{\mathop\mathrm{BWGL}}
\def\ve{\varepsilon}

\newlength{\leftstackrelawd}
\newlength{\leftstackrelbwd}
\def\leftstackrel#1#2{\settowidth{\leftstackrelawd}%
{${{}^{#1}}$}\settowidth{\leftstackrelbwd}{$#2$}%
\addtolength{\leftstackrelawd}{-\leftstackrelbwd}%
\leavevmode\ifthenelse{\lengthtest{\leftstackrelawd>0pt}}%
{\kern-.5\leftstackrelawd}{}\mathrel{\mathop{#2}\limits^{#1}}}

\let\d = \undefined
\DeclareMathOperator{\d}{\textrm{d}}

\title[TSP in a Hilbert Space]{A $d$-dimensional Analyst's Travelling Salesman Theorem for subsets of Hilbert space}

\author{Matthew Hyde }

\address{Matthew Hyde\\
School of Mathematics \\ University of Edinburgh \\ JCMB, Kings Buildings \\
Mayfield Road, Edinburgh,
EH9 3JZ, Scotland.}
\email{m.hyde "at" ed.ac.uk}


\thanks{M. Hyde was supported by The Maxwell Institute Graduate School in Analysis and its
Applications, a Centre for Doctoral Training funded by the UK Engineering and Physical
Sciences Research Council (grant EP/L016508/01), the Scottish Funding Council, Heriot-Watt
University and the University of Edinburgh.}

\begin{document}

\begin{abstract}
We are interested in quantitative rectifiability results for subsets of infinite dimensional Hilbert space $H$. We prove a version of Azzam and Schul's $d$-dimensional Analyst's Travelling Salesman Theorem in this setting by showing for any lower $d$-regular set $E \subseteq H$ that 
\[  \diam(E)^d +  \beta^d(E) \sim \mathscr{H}^d(E) + \text{Error}, \]
where $\beta^d(E)$ give a measure of the curvature of $E$ and the error term is related to the theory of uniform rectifiability (a quantitative version of rectifiability introduced by David and Semmes). To do this, we show how to modify the Reifenberg Parametrization Theorem of David and Toro so that it holds in Hilbert space. As a corollary, we show that a set $E \subseteq H$ is uniformly rectifiable if and only if it satisfies the so-called Bilateral Weak Geometric Lemma, meaning that $E$ is bi-laterally well approximated by planes at most scales and locations.    

\end{abstract}

\maketitle

\tableofcontents

\section{Introduction}
	Let $H$ be a real separable Hilbert space and $1 \leq d < \dim(H)$ (the dimension of $H$, possibly infinite). We say a set $E \subseteq H$ is \textit{$d$-rectifiable} if it can be covered $\mathscr{H}^d$-a.e. by countably many Lipschitz images of $\R^d.$ Here, $\mathscr{H}^d$ stands for the $d$-dimensional Hausdorff measure. Rectifiable sets are a central object of study in geometric measure theory and form a natural setting in many other areas. In the early nineties, driven by trying to understand the $L^2$-boundedness of certain singular integral operators, there was interest in studying more quantitative aspects of rectifiability. Two of the main results to come out of this period were the Analyst's Travelling Salesman Theorem of Peter Jones \cite{jones1990rectifiable}, and the theory of uniform rectifiability, due to David and Semmes \cite{david1991singular,david1993analysis}. 

Peter Jones' Travelling Salesman Theorem (TST) gives a necessary and sufficient condition for when a set $E \subseteq \R^2$ can be contained in a rectifiable curve (curve of finite length). This condition is stated in terms of $\beta$-numbers, which give a measure of flatness at a particular location and scale. For sets $E,B \subseteq H,$ define
\begin{align}\label{e:beta-number}
	\beta_{E,\infty}^d(B) = \frac{2}{r_B}\inf_L \sup\{\text{dist}(y,L) : y \in E \cap B\},
\end{align}
where $L$ ranges over $d$-planes in $H.$ Jones proved the following in the case that $H = \R^2.$ 
\begin{thm}\label{1DTSP}
	Let $E \subseteq \R^2$ and let $\Delta$ denote the collection of dyadic cubes in $\R^2.$ Then there is a connected set $\Gamma \supseteq E$ such that 
	\begin{align}\label{e:TSP1}
		\mathscr{H}^1(\Gamma) \lesssim \diam E + \sum_{\substack{Q \in \Delta \\ Q \cap E \not= \emptyset}} \beta_{E,\infty}^1(3Q)^2\ell(Q).
	\end{align}
	Conversely, if $\Gamma \subseteq \C$ is connected and $\mathscr{H}^1(\Gamma) < \infty$, then
	\begin{align}\label{e:TSP2}
		\diam \Gamma + \sum_{\substack{Q \in \Delta \\ Q \cap \Gamma \not= \emptyset}} \beta_{\Gamma,\infty}^1(3Q)^2\ell(Q) \lesssim \mathscr{H}^1(\Gamma).
	\end{align}
\end{thm}
In other words, if $E$ is flat enough then there exists a rectifiable curve $\Gamma$ so that $E \subseteq \Gamma$ and there is quantitative control on the length of $\Gamma.$ Moreover, the second half of Theorem \ref{1DTSP} shows there is quantitative control on how non-flat a curve can be. This was later extended to subsets of $\R^n$ by Okikiolu \cite{okikiolu1992characterization} (with the implicit constants depending on the ambient dimension $n$), and to subsets of infinite dimensional Hilbert space by Schul \cite{schul2007subsets}. Notice that if $\Gamma$ is a curve, then Theorem \ref{1DTSP} implies 
\begin{align}\label{e:TSP3}
	\mathscr{H}^1(\Gamma) \sim \diam \Gamma + \sum_{\substack{Q \in \Delta \\ Q \cap E \not=\emptyset}} \beta^{1}_{\Gamma,\infty}(3Q)^2\ell(Q). 
\end{align}
Similar results have been considered in a number of non-Hilbert space settings, for example, the Heisenberg group \cite{ferrari2007geometric,li2016upper,li2016traveling}, Carnot groups \cite{chousionis2019traveling,li2019stratified}, and general metric spaces \cite{hahlomaa2005menger, hahlomaa2008menger, schul2007ahlfors, david2020sharp}. See \cite{badger2019holder} for a similar result concerning Hölder curves.  

Theorem \ref{1DTSP} is concerned with quantitative information for 1-dimensional subsets of Euclidean space; for higher dimensional subsets, we are lead to a discussion of uniformly rectifiable sets introduced by David and Semmes in \cite{david1993analysis}. 

\begin{defn}\label{d:UR}
	A set $E \subseteq H$ is called \textit{uniformly $d$-rectifiable} (UR) if it is \textit{Ahlfors $d$-regular}, meaning there is $A > 0$ such that
\[r^d/A \leq \mathscr{H}^d(E \cap B(x,r)) \leq Ar^d \ \text{for all} \ x \in E, \ r \in (0,\text{diam}E) \] 
and has \textit{big pieces of Lipschitz images of} $\R^d$, meaning there are constants $L,c>0$ such that for all $x \in E$ and $r \in (0,\diam E),$ there is an $L$-Lipschitz map $f:\R^d \rightarrow H$ such that 
\[\mathscr{H}^d(E \cap B(x,r) \cap f(B_d(0,r))) \geq cr^d.\]
\end{defn} 
Uniform rectifiability is a stronger quantitative analogue of rectifiability (every UR set is rectifiable but not every rectifiable set is UR); it has found applications to singular integrals \cite{david1993analysis, tolsa2009uniform, nazarov2014uniform} and harmonic measure, see \cite{azzam2020harmonic} and the references therein. When David and Semmes defined uniform rectifiability in \cite{david1993analysis}, they proved many equivalent conditions. One important result (now a black-box in the theory of UR) is a characterization of UR sets by the Bilateral Weak Geometric Lemma. 

To state the definition, we need some notation. For closed sets $E,F \subseteq H$ and a set $B \subseteq H$, recall the local normalized Hausdorff distance
\begin{align}\label{e:LHD}
	d_B(E,F) = \frac{2}{\text{diam}(B)}\max \left\{\sup_{y \in E \cap B}\text{dist}(y,F), \sup_{y \in F \cap B}\text{dist}(y,E) \right\}.
\end{align}
If $B = B(x,r),$ we may write $d_{x,r} = d_{B(x,r)}.$ Let $E \subseteq H$ and $\mathscr{D}$ be the Christ-David cubes for $E$ (these are subsets of $E$ that behave like Euclidean dyadic cubes, see Theorem \ref{cubes} for more details). For a cube $Q \in \mathscr{D}$, let $x_Q$ denote its ``centre", $\ell(Q)$ denote its ``side-length" and $B_Q = B(x_Q,\ell(Q)).$ Again, see Theorem \ref{cubes}. For constants $A\geq 1$ and $\ve>0$, define 
\[ \BWGL(A,\ve) = \{Q \in \mathscr{D} : d_{AB_Q}(E,P) \geq \ve \mbox{ for all $d$-planes } P\} \]
and for a cube $Q \in \mathscr{D},$ define
\[ \BWGL(Q,A,\ve) = \sum_{\substack{R \subseteq Q \\ R \in \BWGL(A,\ve)}} \ell(R)^d. \]
\begin{defn}\label{d:BWGL}
We say $E$ satisfies the \textit{Bilateral Weak Geometric Lemma} (BWGL) if for each $A \geq 1$ and $\ve > 0$ there is a constant $C = C(A,\ve)$ so that for any $Q \in \mathscr{D}$,
\[ \BWGL(Q,A,\ve) \leq C \ell(Q)^d. \]
\end{defn} 
David and Semmes originally stated the BWGL in terms of a \textit{Carleson measure} condition, see \cite[Definition 2.2]{david1993analysis}, however the above formulation is equivalent. See \cite[Pages 54-55]{david1993analysis}, where a similar equivalence is discussed for the \textit{Weak Geometric Lemma} (WGL). In \cite{david1993analysis}, David and Semmes show that an Ahlfors $d$-regular set in $\R^n$ is UR if and only if it satisfies the BWGL. This has been an essential tool in many results concerning UR, for example \cite{tolsa2015uniform, hofmann2016uniform, azzam2020weak}. As a corollary of our main result, we show the same is true for subsets of infinite dimensional Hilbert space. 
\begin{thm}\label{t:BWGL}
	Let $E \subseteq H$ be Ahlfors $d$-regular. Then $E$ is UR if and only if $E$ satisfies the \emph{BWGL}. 
\end{thm}

\begin{rem}
	The proof of Theorem \ref{t:BWGL} for subsets of $\R^n$ can be found in \cite[Chapter II.2]{david1993analysis}. The proof relies on the existence of a  ``unit star" $\Sigma_0$, which is the intersection of the unit ball with a \textit{finite} number of $d$-planes, chosen in such a way that $\Sigma_0$ has large projections onto every $d$-plane $P \subseteq \R^n.$ No such set exists in Hilbert space. These star-like sets are also used in the David and Semmes proof of the fact that that WHIP and WTP imply BWGL. See \cite[Theorem II.3.9]{david1993analysis} and its proof beginning in \cite[Section II.3.5]{david1993analysis}.
\end{rem}

For sets which are not Ahlfors $d$-regular, some partial results akin to the first half of Theorem \ref{1DTSP} are obtained in \cite{pajot1996theoreme, david2012reifenberg}. However, the second half of Theorem \ref{1DTSP} (with the $\beta$-numbers as defined in \eqref{e:beta-number}) is known to be false by an example of Fang \cite{fang1990cauchy}. With a different $\beta$-number, some progress was made in \cite{azzam2018analyst} for lower content $d$-regular sets. A set $E \subseteq H$ is said to be \textit{lower content $d$-regular} in a ball $B$ if there exists a constant $c>0$ such that for each $x \in E \cap B$ and $r \in (0,r_B),$ 
\begin{align*}
	\mathscr{H}_\infty^d(E \cap B(x,r)) \geq cr^d. 
\end{align*}
Here, $\mathscr{H}_\infty^d$ denotes the Hausdorff content, see Section \ref{s:prelims}. Azzam and Schul introduced a new $\beta$-number, based on a $\beta$-number of David and Semmes \cite[Equation 1.9]{david1991singular}, and proved a version of Jones' TST for lower content $d$-regular subsets of $\R^n$. Our main result is a version of that in \cite{azzam2018analyst} for infinite dimensional Hilbert spaces. We first recall the $\beta$-number from \cite{azzam2018analyst}. For $E \subseteq H$ and a ball $B$, define
\begin{align}\label{e:beta-cont}
	\beta_E^{d,p}(B) &= \inf_L \left(\frac{1}{r_B^d} \int_0^1 \mathscr{H}_\infty^d (\{x \in E \cap B : \text{dist}(x,P) > tr_B \})t^{p-1}\, dt \right)^\frac{1}{p},
\end{align}
where the infimum is taken over all $d$-planes $L$ in $H$. Also, define the critical exponent, 
\begin{align}
	p(d) =	\begin{cases}
		\frac{2d}{d-2} & d > 2 \\
		\infty & d \leq 2		
	\end{cases}.
\end{align}
The statement, which separate into two parts, reads as follows.

\begin{thm}\label{Thm1}
	Let $1 \leq d < \dim(H)$, $1 \leq p < p(d)$, $C_0 > 1$ and $A > 10^5$. Let $E \subseteq H$ be a lower content $d$-regular set with regularity constant $c$ and Christ-David cubes $\mathscr{D}.$ There exists $\ve >0$ small enough so that the following holds.
	Let $Q_0 \in \mathscr{D}$ and 
	\[ \beta_{E,C_0,d,p}(Q_0) \coloneqq \ell(Q_0)^d + \sum_{Q \subseteq Q_0} \beta^{d,p}_E(C_0B_Q)^2\ell(Q)^d. \] 
	Then
	\begin{align}\label{e:dTSP1}
		\beta_{E,C_0,d,p}(Q_0) \lesssim_{A,d,c,p,C_0,\epsilon} \mathscr{H}^d(Q_0) + \emph{BWGL}(Q_0,A,\ve).
	\end{align}
\end{thm}

\begin{thm}\label{Thm2}
	Let $1 \leq d < \dim(H)$, $1 \leq p < \infty$, $C_0 > 1$ be sufficiently large (we will choose $C_0 > 2\rho^{-1}$ with $\rho$ as in Theorem \ref{cubes}), $A > 1$ and $\ve >0.$ Let $E \subseteq H$ be a lower content $d$-regular set with regularity constant $c$ and Christ-David cubes $\mathscr{D}$. For $Q_0 \in \mathscr{D}$ we have 
	\begin{align}\label{e:dTSP2}
		\mathscr{H}^d(Q_0) + \BWGL(Q_0,A,\ve) \lesssim_{A,d,c,C_0,\epsilon} \beta_{E,C_0,d,p}(Q_0),
	\end{align}
	with $\beta_{E,C_0,d,p}(Q_0)$ as in Theorem \ref{Thm1}. 
\end{thm}

\begin{rem}
As already mentioned, this was proven in the case $\dim(H) < \infty$ in \cite{azzam2018analyst} (with constant depending on $\dim(H)$), see \cite[Theorem 3.2, Theorem 3.3]{azzam2018analyst}. Our contribution is the case $\dim(H) = \infty$ (we also achieve constants independent of $\dim(H)$ if $\dim(H) < \infty$).  
\end{rem} 

Instead of asking what kind of sets can imitate the role of rectifiable curves in Theorem \ref{1DTSP}, Theorem \ref{Thm1} and Theorem \ref{Thm2} focus on proving estimates like \eqref{e:TSP3} for lower content $d$-regular subsets of $H$. In contrast to \eqref{e:TSP3}, note the appearance of the error term $\BWGL(Q_0,A,\ve).$ This accounts for the fact a general lower content $d$-regular set may have holes, while a curve does not. The error terms disappears for example if $E$ is Reifenberg flat. Recall that a set $E \subseteq H$ is said to be $(\ve,d,\delta_0)$-\textit{Reifenberg flat} if for each $x \in E$ and $0 < r < \delta_0$ there exists a $d$-plane $P_{x,r}$ so that 
\[ d_{x,r}(E,P_{x,r}) \leq \ve. \] 
In fact, the proof of Theorem \ref{Thm1} will utilize a similar result specialized to Reifenberg flat surfaces. Again, the following was proved in \cite{azzam2018analyst} for the case $\dim(H) < \infty$ with constants depending on $n$.  

\begin{thm}\label{Thm3}
	Let $1 \leq d < \dim(H)$, $1 \leq p < p(d)$ and $C_0 > 1$. Suppose $E$ is an $(\ve,d,\delta_0)$-Reifenberg flat surface with Christ-David cubes $\mathscr{D}$. There exists $\ve >0$ small enough so that for any $Q_0 \in \mathscr{D}$ we have 
	\begin{align}\label{e:dTSP3}
		\beta_{E,C_0,d,p}(Q_0) \lesssim_{d,p,C_0,\epsilon,\delta_0} \mathscr{H}^d(Q_0).
	\end{align}
\end{thm}

Before moving on, let us mention some related works. Edelen, Naber and Valtorta \cite{edelen2016quantitative} give very general results about how well the size of a measure $\mu$ is bounded from above by a quantity involving $\beta$-numbers. These $\beta$-numbers are like those introduced in \eqref{e:beta-cont}, except one integrates with respect to $\mu$ instead of $\mathscr{H}^d_\infty.$ In contrast to the above, their result requires the existence of a locally finite measure. Moreover, even if $\mathscr{H}^d|_E$ is locally finite, the $\beta$-number associated to the measure $\mathscr{H}^d|_E$ can still be quite unstable. See \cite[Section 1.3]{azzam2018analyst}. \\

In the Euclidean setting, Theorem \ref{Thm1} and Theorem \ref{Thm2} were extended in \cite{azzam2019quantitative} to allow for more general error terms in place of $\BWGL(Q_0,A,\ve).$ Like $\BWGL(Q_0,A,\ve)$, these error terms are connected to the theory of uniform rectifiability, see \cite{azzam2019quantitative} for more details. Whereas the results in this paper focus on proving estimates like \eqref{e:TSP3}, a combination of two recent papers gives a $d$-dimensional TST that more closely resembles that in Theorem \ref{1DTSP}. Working with the $\beta$-numbers as defined in \eqref{e:beta-cont}, Villa \cite{villa2019higher} shows that if $\beta_{E,C_0,p}(Q_0) < \infty$ for a lower content $d$-regular set $E$ and $Q_0 \in \mathscr{D}$, then $Q_0$ can be contained in a \textit{stable $d$-surfaces} $\Sigma$ with $\mathscr{H}^d(\Sigma) \sim \beta_{E,C_0,p}(Q_0)$. Stable $d$-surfaces were defined by David in \cite{david2004hausdorff}. In \cite{hyde2020analyst}, we introduce a new $\beta$-number and show that for \textit{any} set $E$ and $Q_0 \in \mathscr{D}$, $Q_0$ can be contained in a lower regular set $F$ with $\tilde\beta_{E,C_0,p}(Q_0) \sim \beta_{F,C_0,p}(Q_0),$ where $\tilde\beta_{E,C_0,p}$ is defined as in the statement of Theorem \ref{Thm1} with respect to the $\beta$-number from \cite{hyde2020analyst}. Combing the two, for any $E \subseteq \R^n$ and $Q_0 \in \mathscr{D}$, there exists a stable $d$-surface $\Sigma$ so that $Q_0 \subseteq \Sigma$ and $\mathscr{H}^d(\Sigma) \sim \tilde\beta_{E,C_0,p}(Q_0).$

\subsection{Structure of the paper}

The proof of the forward direction of Theorem \ref{Thm1} will be carried out in Section \ref{s:Thm1}. This will make use of Theorem \ref{Thm3} which is proven in Section \ref{Sec:StopTime}. For each of these, it suffices to consider the sets as subsets of $\R^n$ for some $n \in \N$ and show that the constants are independent of $n$. These reductions are shown in Appendix \ref{a:Euclidean}. The proof then follows a similar strategy employed by Azzam and Schul in \cite{azzam2018analyst}. We will state without proof those results from \cite{azzam2018analyst} whose constants do not depend on $n$; we will provide alternative proofs for those results whose constants do depend on $n$. Let us mention here that there is a mistake in Azzam and Schul's proof of Theorem \ref{e:dTSP1}, specifically, in the proof of \cite[Lemma 10.4]{azzam2018analyst}. We fix the error in this paper, with a bit more work. An important tool will be a Hilbert space version of the Reifenberg parametrization result of David and Toro. A similar result (also for general Banach spaces) has appeared (in fragments) in \cite{edelen2018effective}, however, David and Toro's Euclidean formulation is very general and can be applied as soon as one has a \textit{coherent collection of balls and planes} (CCBP), see Definition \ref{d:CCBP}. In Section \ref{s:DT} and Appendix \ref{a:DT}, we describe how to tie these results together to give a Hilbert space version that is also stated in terms of a CCBP. The proofs of Theorem \ref{Thm2} and Theorem \ref{t:BWGL} will be carried out in Section \ref{s:Thm2} and Section \ref{s:BWGL}, respectively.

\section{Preliminaries}\label{s:prelims}

We gather some notation and results that will be needed throughout.  If there exists $C>0$ such that $a \leq Cb,$ we write $a \lesssim b.$ If the constant $C$ depends on a parameter $t$, we write $a \lesssim_t b.$ We will write $a \sim b$ if $a \lesssim b$ and $b\lesssim a,$ and define $a \sim_t b$ similarly. From here on, we fix $H$ to be an infinite dimensional real and separable Hilbert Space. For a ball $B \subseteq H$, we will let $r_B$ denote its radius. For closed sets $A,B \subseteq H$, let
\begin{align*}
\text{dist}(A,B) = \inf\{|x-y| : x \in A, \ y \in B \}
\end{align*}
and 
\[\text{diam}(A) = \sup\{|x-y| : x,y \in A \}. \]
For $A \subseteq H, \ d \geq 0$ and $0 <\delta \leq \infty$ define
\begin{align}\label{e:HausCont}
\mathscr{H}^d_\delta(A) = \inf \left\{ \sum_i \diam(A_i)^d : A \subseteq \bigcup_i A_i \mbox{ and } \diam A_i \leq \delta \right\}.
\end{align}
The $d$-dimensional Hausdorff \textit{content} of $A$ is defined to be $\mathscr{H}^d_\infty(A)$ and the $d$-dimensional Hausdorff \textit{measure} of $A$ is defined to be
\[ \mathscr{H}^d(A) = \lim_{\delta \rightarrow 0} \mathscr{H}^d_\delta(A).\]

\subsection{Christ-David cubes}

We will need a version of Christ-David cubes tailored to subsets $E$ of $H$. These are a collection of  subsets of $E$ that behave like Euclidean dyadic cubes. Theorem \ref{cubes} below is a combination of the formulations due to Christ \cite{christ1990b}, and Hytönen and Martikainen \cite{hytonen2012non}.   
\begin{lem}\label{cubes}
There exist constants $c_0, \rho \in (0,1)$ so that the following holds. Let $E \subseteq H$ and $X_k$ be a sequence of maximal $\rho^k$-separated nets. Then, for each $k \in \Z$, there is a collection $\mathscr{D}_k$ of subsets of $E$, which we call \textit{cubes}, such that:
\begin{enumerate}
\item If $Q_1,Q_2 \in \mathscr{D} = \bigcup_{k}\mathscr{D}_k$ and $Q_1 \cap Q_2 \not= \emptyset,$ then $Q_1 \subseteq Q_2$ or $Q_2 \subseteq Q_1.$ 
\item For $Q \in \mathscr{D},$ let $k(Q)$ be the unique integer so that $Q \in \mathscr{D}_k$ and set $\ell(Q) = 5\rho^k.$ Then there is $x_Q \in X_k$ such that
\begin{align*}
B(x_Q,c_0\ell(Q)) \subseteq Q \subseteq B(x_Q , \ell(Q)) \eqqcolon B_Q. 
\end{align*}
\item For each $k \in \Z,$ $E \subseteq \bigcup_{Q \in \mathscr{D}_k} B_Q.$
\end{enumerate}
With some additional assumption on the space $H$ or the set $E$, we have some more properties.  
\begin{enumerate}
	\setcounter{enumi}{3}
	\item If $H = \R^n$, then for each $k \in \Z, \ E = \bigcup_{Q \in \mathscr{D}_k} Q.$ 
	\item If $E$ is Ahlfors $d$-regular, then there exists $C \geq 1$ so that $\mathscr{H}^d( \{ x \in Q : \dist(x,E \setminus Q) \leq \eta \rho^k \} ) \lesssim \eta^\frac{1}{C} \ell(Q)^d$ for all $Q \in \mathscr{D}$ and $t > 0.$
\end{enumerate}

\end{lem}

\begin{rem}
	Properties (1)-(3) only rely on the metric space structure of $H$ and hold in both formulations; (3) is not explicitly stated in either but is an easy consequence of the definitions. Property (4) is unique to the cubes constructed in \cite{hytonen2012non}, which are those used in \cite{azzam2018analyst}. Finally, (5) is a property of the cubes constructed in \cite{christ1990b}, we will make use of this when proving Theorem \ref{t:BWGL}.
	\end{rem}
For $Q \in \mathscr{D}$ and $N \in \N$, let $Q^{(N)}$ denote the unique cube $Q'$ so that $Q \subseteq Q'$ and $\ell(Q) = \rho^N \ell(Q').$ We call $Q^{(1)}$ the \textit{parent} of $Q$. We let $\text{Child}(Q)$ denote the set of $R \in \mathscr{D}$ so that $R^{(1)} = Q.$ The cubes in $\text{Child}(Q)$ are the \textit{children} of $Q$. We introduce a distance function to a collection of cubes $\mathscr{C} \subseteq \mathscr{D}.$ For $x \in H$ set
\[ d_\mathscr{C}(x) = \inf\{\ell(R) + \text{dist}(x,R) : R \in \mathscr{C} \}, \] 
and for $Q \in \mathscr{D},$ set
\[ d_\mathscr{C}(Q) = \inf\{ d_\mathscr{C}(x) : x \in Q \} = \inf \{ \ell(R) + \dist(Q,R) : R \in \mathscr{C} \}. \]
The following lemma is standard and can be found in, for example, \cite{azzam2018analyst}.
\begin{lem}\label{LipCubes}
Let $\mathscr{C} \subseteq \mathscr{D}$ and $Q,Q^\prime \in \mathscr{D}.$ Then
\begin{align}\label{Tr}
d_\mathscr{C}(Q) \leq 2\ell(Q) + \emph{dist}(Q,Q^\prime) + 2\ell(Q^\prime) + d_\mathscr{C}(Q^\prime).
\end{align}
\end{lem}

At various points in the proof we will need to split $\mathscr{D}$ into stopping-time regions, defined below.

\begin{defn}\label{StoppingTime}
	A collection of cubes $S \subseteq \mathscr{D}$ is called a \textit{stopping-time region} if the following hold.
	\begin{enumerate}
		\item There is a \textit{top cube} $Q(S) \in S$ such that $Q(S)$ contains all cubes in $S$. 
		\item If $Q \in S$ and $Q \subseteq R \subseteq Q(S),$ then $R \in S$.
		\item If $Q \in S$, then all siblings of $Q$ are also in $S$. 
	\end{enumerate}
	The \textit{minimal cubes} of $S$, denoted by $\min(S)$, are those cubes in $S$ whose children are not in $S$ and let $z(S)$ denote the points contained in $Q(S)$ for which we never stopped, i.e. 
	\begin{align*}
		z(S) = Q(S) \setminus \bigcup \{Q : Q \in \min(S) \}.
	\end{align*}
\end{defn}

\subsection{Choquet integrals and $\beta$-numbers} For a set $E \subseteq H$ and a function $f : H \rightarrow [0,\infty)$, we define the Choquet integral of $f$ over $E$ by the formula
\begin{align}
\int_E f^p \, d\mathscr{H}_\infty^d \coloneqq \int_0^\infty \mathscr{H}_\infty^d( \{x \in E : f(x) > t \}) t^{p-1} \, dt.
\end{align} 
We state some properties of Choquet integration. For the first lemma, see \cite[Proposition 2.1]{wang2011some} for more details. Note, in \cite{wang2011some} there are additional upper and lower continuity assumption, but these are not required for the following lemma.

\begin{lem}\label{Choquet}
	Let $E \subseteq H,$ $f,g: H \rightarrow [0,\infty)$ such that $f \leq g$, and $\alpha >0$. Then 
	\begin{enumerate}
		\item $\int_E f \, d\mathscr{H}_\infty^d \leq  \int_E g \, d\mathscr{H}_\infty^d;$
		\item $ \int_E (f + \alpha) \, d\mathscr{H}_\infty^d = \int_E f  \, d\mathscr{H}_\infty^d + \alpha \mathscr{H}_\infty^d(E);$
		\item $\int_E \alpha f  \, d\mathscr{H}_\infty^d =\alpha \int_E f \, d\mathscr{H}_\infty^d.$
	\end{enumerate}
\end{lem}

\begin{lem}[{\cite[Lemma 2.1]{azzam2018analyst}}] \label{l:subsum}
Let $0 < p< \infty.$ Let $f_i:H \rightarrow [0,\infty)$ be a countable collection of Borel functions in $H.$ If the sets $\emph{supp} \, f_i = \{f_i > 0\}$ have bounded overlap, meaning there exists a $C < \infty$ such that
\[ \sum \mathds{1}_{\emph{supp} \, f_i} \leq C,\]
then
\begin{align}
\int_H \left( \sum f_i \right)^p \, d\mathscr{H}^d_\infty \leq C^p \sum \int_H f_i^p \, d\mathscr{H}^d_\infty.
\end{align}
\end{lem}

We need to know that Choquet integration satisfies a Jensen inequality. The proof of this fact in \cite{azzam2018analyst} is unique to $\R^n.$ We adapt the usual proof of Jensen's inequality (see for example Rudin \cite{rudin2006real}) to work in $H$.  

\begin{lem}\label{JensenPhi}
Suppose $E \subseteq H$ with $0 < \mathscr{H}_\infty^d(E) <\infty$, $\phi : [0,\infty) \rightarrow [0,\infty)$ is convex and $f : H \rightarrow [0,\infty)$ is bounded. Then,
\begin{align}\label{e:Jensen}
\phi \left( \frac{1}{\mathscr{H}^{d}_{\infty}(E) }\int_E f \, d\mathscr{H}^{d}_{\infty} \right) \leq \frac{1}{\mathscr{H}^{d}_{\infty}(E)}\int_E \phi \circ f \, d\mathscr{H}^{d}_{\infty}.
\end{align}
\end{lem} 

\begin{proof}
Since $f$ is bounded and $0 < \mathscr{H}^{d}_{\infty}(E) <\infty$, we can set 
\[t = \frac{1}{\mathscr{H}^{d}_{\infty}(E) }\int_E f \, d\mathscr{H}^{d}_{\infty} < \infty.\] 
Since $\phi$ is convex, if $0 \leq s <t <u <\infty$, then 
\begin{align}\label{Jensen1}
\frac{\phi(t)-\phi(s)}{t-s} \leq \frac{\phi(u)-\phi(t)}{u-t}.
\end{align}
Let $\gamma$ be the supremum of the right-hand side of \eqref{Jensen1} taken over all $u \in [0,t).$ It is clear then that 
\begin{align}
\phi(t) \leq \phi(s) + \gamma \cdot (t -s)
\end{align}
for all $s \in [0,\infty).$ By rearranging the above inequality, for any $x \in H$ we have 
\[ \gamma f(x) + \phi(t) \leq \phi(f(x)) + \gamma t . \]
Integrating both sides with respect to $x$ over $E$, and using Lemma \ref{Choquet}, we have 
\begin{align}
	\gamma \int_E f \, d\mathscr{H}^d_\infty + \mathscr{H}^d_\infty(E)\phi(t) \leq \int_E \phi \circ f \, d\mathscr{H}^{d}_{\infty} + \gamma \mathscr{H}^d_\infty(E)t,
\end{align}
hence, 
\begin{align}
\mathscr{H}^{d}_{\infty}(E) \phi(t) \leq \int_E \phi \circ f \, d\mathscr{H}^{d}_{\infty} + \gamma \cdot \left(\mathscr{H}^{d}_{\infty}(E) t - \int_E f \, d\mathscr{H}^{d}_{\infty} \right).
\end{align}
By definition of $t$, the second term on the right-hand side of the above inequality is zero, from which the lemma follows. 

\end{proof}

Let us recall the $\beta$-number of Azzam and Schul.

\begin{defn}\label{d:beta}
For a set $E \subseteq H$, a ball $B$ and a $d$-dimensional plane $P$, define 
\begin{align*}
\beta_E^{d,p}(B,P) &= \left( \frac{1}{r_B^d} \int_{E \cap B} \left(\frac{\text{dist}(x,P)}{r_B} \right)^p \, \mathscr{H}_\infty^d(x)\right)^\frac{1}{p} \\
&=\left(\frac{1}{r_B^d} \int_0^1 \mathscr{H}_\infty^d (\{x \in E \cap B : \text{dist}(x,P) > tr_B \})t^{p-1}\, dt \right)^\frac{1}{p}
\end{align*}
and set 
\begin{align*}
\beta_E^{d,p}(B) = \inf \{\beta_E^{d,p}(B,P) : P \ \text{is a $d$-dimensional plane in} \ X \}.
\end{align*}
\end{defn}

We shall need the following properties of $\beta^{d,p}_E$.   

\begin{lem}[{\cite[Lemma 2.12]{azzam2018analyst}}] \label{lemma:betap_betainfty}
Assume $E \subseteq H$ and there is $B$ centred on $E$ so that for all $B' \subseteq B$ centred on $E$ we have $\mathscr{H}^d_\infty(E \cap B') \geq cr_{B'}^d$. Then 
\begin{align}\label{eq:betap_betainfty}
\beta_{E,\infty}^d\left( \frac{1}{2}B\right) \lesssim_{c,d} \beta_E^{d,1}(B)^\frac{1}{d+1}.
\end{align}
\end{lem}

\begin{lem}\label{c:beta}
Let $1 \leq p < \infty$ and $E \subseteq H$. Then, for any ball $B$ centred on $E$ with $\mathscr{H}^d_\infty(E \cap B) > 0,$ 
\begin{align}
\beta_E^{d,1}(B) \lesssim \beta_E^{d,p}(B)
\end{align}
\end{lem}

\begin{proof}
The lemma is trivial for $p =1,$ so we can suppose $p>1.$ We can assume without loss of generality that $r_B = 1.$ Let $\ve > 0$ and suppose $L$ is a $d$-plane so that 
\[\beta^{d,p}_E(B,L) \leq \beta^{d,p}_E(B) + \ve.\]
Define a function $\phi: [0,\infty) \rightarrow [0,\infty),$ by setting $\phi(x) = x^p.$ For this range of $p$, $\phi$ is convex. Using this with $0 < \mathscr{H}^d_\infty(E \cap B) < \infty,$ Lemma \ref{JensenPhi} implies
\begin{align}
\beta^{d,1}_E(B) \leq \beta_E^{d,1}(B,L) &= \int_{E \cap B} \text{dist}(x,L) \, d\mathscr{H}_\infty^d \\
&\lesssim \mathscr{H}_\infty^d(E \cap B)^{1-\frac{1}{p}}\left(\int_{E \cap B} \text{dist}(x,L)^p \, d\mathscr{H}_\infty^d \right)^\frac{1}{p} \\
&\leq \left(\int_{E \cap B} \text{dist}(x,L)^p \, d\mathscr{H}_\infty^d \right)^\frac{1}{p} \leq \beta_E^{d,p}(B) + \ve. 
\end{align}
This finishes the proof since $\ve > 0$ was arbitrary. 
\end{proof}

\begin{lem} [{\cite[Lemma 2.14]{azzam2018analyst}}]\label{lemma:monotonicity}
Let $1 \leq p < \infty$ and $E \subseteq H.$ Then, for all balls $B^\prime \subseteq B$ centred on $E,$
\begin{align}
\beta^{d,p}_E(B^\prime) \leq \left(\frac{r_B}{r_{B^\prime}} \right)^{1+\frac{d}{p}} \beta_E^{d,p}(B). 
\end{align}
\end{lem}

\begin{rem}
	In \cite{azzam2018analyst}, the exponent of $r_B/r_{B'}$ is stated as $d+p.$ Examining their proof, one sees that the above exponent is correct. 
\end{rem}

The following lemma allows us to control the $\beta$-number of a set $E_1$ by the $\beta$-number of a nearby set $E_2$ with an error term depending on the average distance between the two sets. It is a Hilbert space version of \cite[Lemma 2.21]{azzam2018analyst}. 

\begin{lem}\label{betaest}
Let $1 \leq p < \infty.$ There exists $\ve>0$ such that the following holds. Suppose $E_1,E_2 \subseteq H$, $B^1$ is a ball centred on $E_1$ and $B^2$ is a ball of same radius but centred on $E_2$ such that $B^1 \subseteq 2B^2.$ Suppose for $i=1,2$ and all balls $B \subseteq 2B^i$ centred on $E_i$ we have $\mathscr{H}_\infty^d( E_i \cap B) \geq c{r_B}^d$  for some $c > 0.$ Then
\begin{align*}
\beta_{E_1}^{d,p}(B^1,P) &\lesssim_{c,d,p} \beta_{E_2}^{d,p}(2B^2,P) \\
&\hspace{4em}+ \left(\frac{1}{r^d_{B^1}}\int_{E_1 \cap 2B^1}\left( \frac{\dist(y,E_2)}{r_{B_1}}\right)^p \, d\mathscr{H}_\infty^d(y) \right)^\frac{1}{p}.
\end{align*} 
\end{lem}

Azzam and Schul's proof of the above result relies heavily on the use of the Besicovitch Covering Theorem and is unsuitable in our setting. Our proof of Lemma \ref{betaest} follows a similar approach to the proof of \cite[Lemma 2.29]{hyde2020analyst}. An important ingredient is the following.

\begin{lem}[{\cite[Theorem 3.10]{edelen2018effective}}]\label{ENV}
Let $V$ be an affine $d$-dimensional plane in a Banach space $X$, and $\{B(x_i,r_i)\}_{i \in I}$ be a family of pairwise disjoint balls with $r_i \leq R, \ B(x_i,r_i) \subseteq B(x,R),$ for some $x \in X$ and $\emph{dist}(x_i,V) < r_i/2.$ Then, there is a constant $\kappa = \kappa(d)$ such that
\begin{equation}\label{e:ENV}\sum_{i \in I} r_i^d \leq \kappa R^d.
\end{equation}
\end{lem}

\begin{rem}
Lemma \ref{ENV} says that a collection of disjoint balls in $X$ lying close to a $d$-plane satisfy a $d$-dimensional packing condition. This will not only be useful for us here, but will play an important role throughout the paper, as it provides us with a way to implement volume arguments in infinite dimensional Hilbert space.
\end{rem}

\begin{proof}[Proof of Lemma \ref{betaest}]
By scaling and translating, we can assume that $B^1 = \B,$ the unit ball in $H$ centred at the origin. For $t>0$, set
\begin{align*}
E_t = \{x \in E \cap \B : \text{dist}(x,L) > t \}.
\end{align*}
It suffices to show 
\begin{align}\label{e:leq}
\mathscr{H}_\infty^d(E_t) &\lesssim_{c,d,p} \mathscr{H}_{\infty}^{d}(\{x \in F \cap 2\B : \text{dist}(x,L) > t/2 \}) \\ 
&\hspace{4em} + \mathscr{H}_\infty^d(\{x \in E \cap 2\B : \text{dist}(x,F) > t/32 \}),
\end{align}
since then
\begin{align*}
\beta_E^{d,p}(\B,L)^p &= \int_0^1 \mathscr{H}_\infty^d(E_t) t^{p-1} \, dt  \\
&\lesssim_{c,d,p} \int_0^1 \mathscr{H}_{\infty}^{d} (\{x \in F \cap 2\B : \text{dist}(x,L) > t/2 \}) t^{p-1} \, dt \\
&\hspace{4em} +  \int_0^1 \mathscr{H}_\infty^d (\{x \in E \cap 2\B : \text{dist}(x,F) > t/32 \}) t^{p-1} \, dt \\
&\lesssim \int_0^1 \mathscr{H}_{\infty}^{d} (\{x \in F \cap 2\B : \text{dist}(x,L) > 2t \}) t^{p-1} \, dt \\
&\hspace{4em} +  \int_0^1 \mathscr{H}_\infty^d (\{x \in E \cap 2\B : \text{dist}(x,F) > 2t\}) t^{p-1} \, dt \\
&\sim\beta_F^{d,p}(2\B,L)^p +  \int_{E \cap 2\B} \dist(x,F)^p \, d\mathscr{H}_\infty^d(x),
\end{align*}
So, let us prove \eqref{e:leq}. We first need to construct a suitable cover for $E_t.$ For $x \in E_t$, let 
$$\delta(x) = \max\{ \text{dist}(x,L), 16\text{dist}(x,F) \}$$
and set $X_t$ to be a maximally separated net in $E_t$ such that, for $x,y \in X_t,$ we have 
\begin{align}\label{Separation}
|x-y| \geq 4 \max\{\delta(x),\delta(y)\}.
\end{align}
Enumerate $X_t = \{x_i\}_{i \in I}.$ For $x_i \in X_t$ denote $B_i = B(x_i,\delta(x_i)).$ Notice, these balls are non-degenerate since $x_i \in X_t \subseteq E_t.$ By maximality we know the balls $\{4B_i\}_{i \in I}$ cover $E_t$ so,
\begin{align}\label{E_t}
\mathscr{H}_\infty^d(E_t) \lesssim \sum_{i \in I} r_{4B_i}^d. 
\end{align} 
We partition $I = I_1 \cup I_2$, where 
\begin{align*}
I_1 = \{ i \in I : \delta(x_i) = \text{dist}(x_i,L) \} 
\end{align*}
and
\begin{align*}
I_2 = \{ i \in I : \delta(x_i) = 16\text{dist}(x_i,F) \}.
\end{align*}
If $\text{dist}(x_i, L) = 16\text{dist}(x_i,F)$, we put $i$ in $I_1$ or $I_2$ arbitrarily. By \eqref{E_t}, it follows that 
\[\mathscr{H}^d_\infty(E_t) \lesssim \sum_{i \in I_1} r_{4B_i}^d + \sum_{i \in I_2} r_{4B_i}^d.\]
We will show that
\begin{align}\label{I_1}
\sum_{i \in I_1} r_{4B_i}^d \lesssim_{c,d,p} \mathscr{H}^{d}_{\infty}( \{x \in F \cap 2\B : \dist(x,L) > t/2\})
\end{align}
and
\begin{align}\label{I_2}
\sum_{i \in I_2} r_{4B_i}^d \lesssim_{c,d,p} \mathscr{H}^d_\infty ( \{x \in E \cap 2\B : \dist(x,F) > t/32\})
\end{align}
from which \eqref{e:leq} follows. The rest of the proof is dedicated to proving \eqref{I_1} and \eqref{I_2}. Let us begin with \eqref{I_1}. Let 
\[F_t = \{x \in F \cap 2\B : \text{dist}(x,L) > t/2\}\]
and $\mathscr{B}$ be a collection of balls centred on $F,$ which cover $F_t$, such that 
\begin{align}\label{e:good}
\mathscr{H}_{\infty}^{d}(F_t) \sim \sum_{B \in \mathscr{B}} r_B^d.
\end{align}
For $i \in I_1$, since $\text{dist}(x_i,L) >t$ (by virtue of that fact that $x_i \in E_t$), and 
\[\tfrac{1}{2}B_i = B(x_i,\text{dist}(x_i,L)/2) \supseteq B(x_i , 8\text{dist}(x_i,F)),\] 
we have
\begin{align}\label{e:use}
F \cap \tfrac{1}{2}B_i \not= \emptyset \quad \text{and} \quad F \cap \tfrac{1}{2}B_i \subseteq F_t.
\end{align}
Additionally, since the collection of balls $\mathscr{B}$ form a cover of $F_t$, there exists at least one $B \in \mathscr{B}$ such that $B \cap \tfrac{1}{2}B_i \not=\emptyset.$ With this observation, we can partition $I_1$ as follows. Let
\begin{align*}
I_{1,1} &= \{i \in I_1 : \text{there exists} \ B \in \mathscr{B} \ \text{such that} \ \tfrac{1}{2}B_i \cap B \not= \emptyset \ \text{and} \ r_B \geq r_{B_i} \}, \\
I_{1,2} &= \{i \in I_1 : r_B < r_{B_i} \mbox{ for all } B \in \mathscr{B} \mbox{ such that }   \tfrac{1}{2}B_i \cap B \not= \emptyset \}. 
\end{align*}
We first control the sum over $I_{1,1}.$ Let $B \in \mathscr{B}$. If there is a ball $B_i$ satisfying $\tfrac{1}{2}B_i \cap B \not=\emptyset$ and $r_B \geq r_{B_i}$ (which by definition implies $i \in I_{1,1}$), then $2B_i \subseteq 4B.$ By \eqref{Separation} we know the $\{2B_i\}$ are disjoint and satisfy $\text{dist}(x_i,L) \leq r_{2B_i}/2,$ so by Lemma \ref{ENV}, we have
\begin{align*}
\sum_{\substack{i \in I_{1,1} \\ \frac{1}{2}B_i \cap B \not=\emptyset \\ r_{B_i} \leq r_B}} r_{4B_i}^d \lesssim \sum_{\substack{i \in I_{1,1} \\ 2B_i \subseteq 4B}} r_{2B_i}^d \lesssim r_B^d.
\end{align*}
Thus,
\begin{align}\label{I_L^1}
\sum_{i \in I_{1,1}} r_{4B_i}^d \leq \sum_{B \in \mathscr{B}} \sum_{\substack{i \in I_{1,1} \\ \frac{1}{2}B_i \cap B \not=\emptyset \\ r_{B_i} \leq r_B}} r_{4B_i}^d \lesssim \sum_{B \in \mathscr{B}}r_B^d. 
\end{align}
We now turn our attention to $I_{1,2}.$ For $i \in I_{1,2}$, let $x_i^\prime$ be the point in $F$ closest to $x_i$ and set $B_i^\prime = B(x_i^\prime, \text{dist}(x_i,L)/4).$ Note that $B_i^\prime \subseteq \tfrac{1}{2}B_i,$ since for $y \in B_i^\prime$ we have
\begin{align*}
|y - x_i| \leq \frac{1}{4}\text{dist}(x_i,L) + |x_i'-x_i| \overset{(i \in I_1)}{\leq} \left(\frac{1}{4} + \frac{1}{16} \right)\text{dist}(x_i,L) \leq \frac{1}{2}\text{dist}(x_i,L).
\end{align*}
Since $F \cap \tfrac{1}{2}B_i \subseteq F_t$ by \eqref{e:use}, and $\mathscr{B}$ forms a cover for $F_t,$ the balls 
\[\{B \in \mathscr{B} : F \cap B \cap \frac{1}{2}B_i\not=\emptyset\}\]
form a cover for $F \cap \tfrac{1}{2}B_i.$ Furthermore, if $B \cap \tfrac{1}{2}B_i \not=\emptyset$ then $B \cap \tfrac{1}{2}B_j = \emptyset$ for all $i\not=j$, that is 
\begin{align}\label{e:1,2}
\# \{i \in I_{1,2} : B_i \cap B \not=\emptyset \} \leq 1,
\end{align}
since otherwise $|x_i - x_j| < 4\max\{\delta(x_i),\delta(x_j)\},$ contradicting \eqref{Separation}. Since $F$ is lower content $d$-regular, we have
\[\mathscr{H}_{\infty}^{d}(F \cap B_i^\prime) \gtrsim r_{B_i^\prime}^d \gtrsim r_{B_i}^d,\]
which implies
\begin{equation}
\begin{aligned}
\sum_{i \in I_{1,2}} r_{4B_i}^d &\lesssim \sum_{i \in I_{1,2}} \mathscr{H}_{\infty}^{d}(F \cap B_i^\prime) \leq  \sum_{i \in I_{1,2}} \mathscr{H}_{\infty}^{d}(F \cap \tfrac{1}{2}B_i) \\
&\leq \sum_{i \in I_{1,2}} \sum_{\substack{B \in\mathscr{B} \\ F \cap B \cap \frac{1}{2}B_i  \not=\emptyset}} r_B^d  \leq \sum_{B \in\mathscr{B}}  \sum_{\substack{i \in I_{1,2} \\ F \cap B \cap \frac{1}{2}B_i  \not=\emptyset}} r_B^d \overset{\eqref{e:1,2}}{\leq}  \sum_{B \in\mathscr{B}} r_B^d. 
\end{aligned}\label{I_L^2}
\end{equation}
Combining \eqref{I_L^1} and \eqref{I_L^2}, we conclude
\begin{align}\label{F_t}
\sum_{i \in I_1} r_{4B_i}^d \lesssim \sum_{B \in \mathscr{B}} r_B^d \overset{\eqref{e:good}}{\lesssim} \mathscr{H}_{\infty}^{d}(F_t)
\end{align}
which is \eqref{I_1}.

We turn our attention to proving \eqref{I_2}, the proof of which follows much the same as that for \eqref{I_1}. Let 
$$E_t^\prime =\{x \in E \cap 2\B : \text{dist}(x,F) > t/32 \}$$
and let $\mathscr{B}^\prime$ be a collection of balls covering $E_t'$ such that each $B \in \mathscr{B}'$ is centred on $E_t'$, has $r_B \leq 1$ and 
\begin{align}\label{e:good2}
\mathscr{H}_\infty^d(E_t^\prime) \sim \sum_{B \in \mathscr{B}^\prime} r_B^d.
\end{align}
As before, we partition $I_2$. For $i \in I_2,$ we have $\text{dist}(x_i,L) >t$ (since $x_i \in E_t$) and
\begin{align}
\text{dist}(x_i,F) \geq \text{dist}(x_i,L)/16 \geq t/16.
\end{align}
Hence,
\begin{align}\label{use1}
E \cap \tfrac{1}{32}B_i = E \cap B(x_i,\text{dist}(x_i,F)/2) \subseteq E_t^\prime.
\end{align}
Thus, for each $B_i,$ since $\mathscr{B}'$ forms a cover for $E_t',$ there exists $B \in \mathscr{B}^\prime$ such that $B \cap \tfrac{1}{32}B_i \not=\emptyset.$ We partition $I_2$ by letting
\begin{align*}
I_{2,1} &= \{i \in I_2 : \text{there exists} \ B \in \mathscr{B}^\prime \ \text{such that} \ \tfrac{1}{32}B_i \cap B \not= \emptyset \ \text{and} \ r_B \geq r_{B_i} \}, \\
I_{2,2} &= \{i \in I_2 : r_B < r_{B_i} \mbox{ for all } B \in \mathscr{B}' \mbox{ such that } \tfrac{1}{32}B_i \cap B \not= \emptyset \}. 
\end{align*}
If $B \in \mathscr{B}^\prime$ and $B \cap \tfrac{1}{32}B_i \not=\emptyset$ with $r_B \geq r_{B_i}$ then $2B_i \subseteq 4B.$ Furthermore, by \eqref{Separation}, we know the $\{2B_i\}$ are disjoint and satisfy $\text{dist}(x_i,L) < \text{dist}(x_i,F)/16 = r_{2B_i}/2,$ so by Lemma \ref{ENV}, we have
\begin{align*}
\sum_{\substack{i \in I_{2,1} \\ \frac{1}{32}B_i \cap B \not=\emptyset \\ r_{B_i} \leq r_B}} r_{4B_i}^d \lesssim \sum_{\substack{i \in I_{2,1} \\ 2B_i \subseteq 4B}} r_{2B_i}^d \lesssim r_B^d.
\end{align*}
Thus,
\begin{align}\label{I_F^1}
\sum_{i \in I_{2,1}} r_{4B_i}^d \leq \sum_{B \in \mathscr{B}^\prime} \sum_{\substack{i \in I_{2,1} \\ \frac{1}{32}B_i \cap B \not=\emptyset \\ r_{B_i} \leq r_B}} r_{2B_i}^d \leq \sum_{B \in \mathscr{B}^\prime}r_B^d. 
\end{align}
We now deal with $I_{2,2}.$ Since by \eqref{use1}, $E \cap \tfrac{1}{32}B_i \subseteq E_t^\prime$, the balls 
\[\{B \in \mathscr{B}^\prime : E \cap B \cap \frac{1}{32}B_i  \not=\emptyset\}\]
form a cover for $E \cap \tfrac{1}{32}B_i.$ As before, if $B \cap \tfrac{1}{32}B_i \not=\emptyset$ then $B \cap \tfrac{1}{32}B_j = \emptyset$ for all $i\not=j$ by \eqref{Separation}. By lower regularity of $E$, we know $\mathscr{H}_\infty^d(E \cap \tfrac{1}{32}B_i) \gtrsim r_{B_i}^d,$ from which we conclude
\begin{align}\label{I_F^2}
\begin{split}
\sum_{i \in I_{2,2}} r_{4B_i}^d &\lesssim \sum_{i \in I_{2,2}} \mathscr{H}_\infty^d(E \cap \tfrac{1}{32} B_i) \leq   \sum_{i \in I_{2,2}} \sum_{\substack{B \in\mathscr{B}^\prime \\ E \cap B \cap \frac{1}{32}B_i  \not=\emptyset}} r_B^d \\
&= \sum_{B \in\mathscr{B}^\prime}  \sum_{\substack{i \in I_{2,2} \\ E \cap B \cap \frac{1}{32}B_i  \not=\emptyset}} r_B^d \lesssim \sum_{B \in \mathscr{B}^\prime} r_B^d. 
\end{split}
\end{align}
The proof of \eqref{I_2} (and hence the proof of the lemma) is completed since
\[ \sum_{i\in I_2} r_{4B_i}^d =  \sum_{i\in I_{2,1}} r_{4B_i}^d + \sum_{i\in I_{2,2}} r_{4B_i}^d \stackrel{ \substack{\eqref{I_F^1} \\ \eqref{I_F^2}}}{\lesssim} \sum_{B \in \mathscr{B}'} r_B^d \stackrel{\eqref{e:good2}}{\lesssim} \mathscr{H}^d_\infty(E_t'). \]

\end{proof}

Finally, we can use $\beta$-numbers to control the angles between best approximating planes. First, some notation. For two planes $P,P^\prime$ containing the origin, we define
\begin{align}
	\angle(P,P^\prime) = d_{B(0,1)}(P,P^\prime).
\end{align}
If $P,P^\prime$ are general affine planes with $x \in P$ and $y \in P^\prime$, we define 
\begin{align}
	\angle(P,P^\prime) = \angle(P-x,P^\prime - y). 
\end{align}

\begin{lem}[{\cite[Lemma 2.18]{azzam2018analyst}}]\label{AngControl}
Let $M >1$, $\alpha >0$, $E$ a Borel set, and $\mathscr{D}$ be the cubes for $E$ from Lemma \ref{cubes} and $Q_0 \in \mathscr{D}.$ Let $P_Q$ satisfy $\beta_E^{d,1}(MB_Q) = \beta_E^{d,1}(MB_Q,P_Q)$. Let $Q,R \subseteq Q_0$, and suppose each $T \subseteq Q_0$ which contains either $Q$ or $R$ satisfies $\beta_E^{d,1}(MB_T) < \ve$. Then for $\Lambda >0$, if $\emph{dist}(Q,R) \leq \Lambda \max\{\ell(Q),\ell(R) \} \leq \Lambda^2\min\{\ell(Q),\ell(R)\},$ then
\begin{align*}
\angle(P_Q,P_R) \lesssim _{M,\Lambda}\ve.
\end{align*}
\end{lem}

\section{Reifenberg parametrisations in Hilbert space}\label{s:DT}

In this section we discuss the Reifenberg parametrization theorem of David and Toro \cite{david2012reifenberg}. This was one of the main tool in \cite{azzam2018analyst}, and we need a version to hold in Hilbert space. It turns out most of the arguments of David and Toro make sense in Hilbert space. We will describe the necessary changes and the remaining details will be outlined in the Appendix.  

\subsection{CCBP} The main result of \cite{david2012reifenberg} is stated in terms of a \textit{coherent collection of balls and planes} (CCBP), defined below.

\begin{defn}\label{d:CCBP}
Let $P_0$ be a $d$-plane in $H$. Let $\{ x_{j,k} \}, \ k \in \N, \ j \in J_k,$ be a collection of points in $H$ such that 
\begin{align}\label{e:csep}
|x_{i,k} - x_{j,k}| \geq 10^{-k} \eqqcolon r_k \quad \text{for all} \ i,j  \in J_k.
\end{align}
For each $k \in \N, \ j \in J_k$, let $B_{j,k} = B(x_{j,k},r_k)$ and $P_{j,k}$ be a $d$-plane such that $x_{j,k} \in P_{j,k}.$ The triple $\mathscr{P} = ( P_0 , \{B_{j,k}\},\{P_{j,k}\}), \ k \in \N, \ j \in J_k$, is said to be a \textit{coherent collection of balls and planes} (CCBP) if it satisfies the following conditions:
\begin{enumerate}
\item For each $k \in \N$ and $j \in J_k$,
\begin{align}\label{e:cinV^2}
x_{j,k} \in V^2_{k-1}, 
\end{align}
where 
\begin{align}\label{e:V_k}
 V_k^\lambda = \bigcup_{j \in J_k} \lambda B_{j,k}.
\end{align}
\item For $j \in J_0$, 
\begin{align}\label{e:cdist0}
\dist(x_{j,0},P_0) \leq \ve.
\end{align}
\item For $k \in \N$ and $i,j \in J_k$ such that $|x_{i,k} - x_{j,k}| \leq 100 r_k,$ 
\begin{align}\label{e:c100rk}
d_{x_{j,k},100r_k}(P_{i,k},P_{j,k}) \leq \ve.
\end{align}
\item For $j \in J_0$ and $x \in P_0$ such that $\dist(x,P_0) \leq 2,$
\begin{align}\label{e:c2}
d_{x_{j,0},100}(P_{j,0},P_0) \leq \ve.
\end{align}
\item For $k \in \N$, $i \in J_k$ and $j \in J_{k+1}$ such that $|x_{i,k} - x_{j,k+1}| \leq 2r_k, $
\begin{align}\label{e:c2r_k}
d_{x_{i,k},20r_k}(P_{i,k},P_{j,k+1}) \leq \ve.
\end{align}
\end{enumerate}
\end{defn}

\begin{rem}\label{r:ve_k}
	In practice, to verify properties (2)-(5), we will show $\ve_k(x_{j,k}) < \ve$ for all $k \geq0$ and $j \in J_k$, where 
	\[\varepsilon_k(x) = \sup\{d_{x_{i,l},10^4r_l}(P_{j,k},P_{i,l}) : j \in J_k, \ \abs{l-k} \leq 2, \ i \in J_l, x \in 100B_{j,k} \cap 100B_{i,l} \}.\] 
	\end{rem}

\subsection{Partition of unity}

The main alteration to the construction in \cite{david2012reifenberg} is in defining a partition of unity, adapted to the collection of balls in the CCBP. Let us outline the David-Toro construction, see also \cite[Chapter 3]{david2012reifenberg}. Let $k \geq 0.$ David and Toro begin by defining a collection of smooth bump functions $\theta_{j,k}$ supported on the balls $10B_{j,k}$, $j \in J_k$. They then introduce a collection of maximally separated points $x_{l,k}$ in $\R^n \setminus V_k^9$, a collection of balls $B_{l,k}$ centred on $x_{l,k}$, and a collection of function $\theta_{l,k}$ supported on $10B_{l,k}$. These are indexed by $l \in L_k.$ The supports of the function $\theta_{j,k}, \ j \in J_k \cup L_k$ have bounded overlap with constant depending on $n$, which allows them to construct a partition of unity for the whole of $\R^n.$ Since we are working in infinite dimensional Hilbert space, we have no hope of controlling the overlap of balls centred on an arbitrary collection of separated points. We can however, control the overlap of the balls $10B_{j,k}, \ j \in J_k$, using Lemma \ref{ENV}, since nearby points must lie close to a $d$-plane. We will use this fact to construct a partition of unity by the following lemma, which is a version of \cite[Lemma 3.1]{edelen2018effective}.  

\begin{lem}\label{l:tpu}
	Let $H$ be a Hilbert space. Let $\Gamma \geq 1$, $r>0$, and let $\{10B(x_i,r)\}_{i \in I}$ be a collection of balls with bounded overlap, so that 
	\[ \sum_{i \in I} \mathds{1}_{10B(x_i,r)}(x) \leq \Gamma \]
	for all $x \in H$. Then there are smooth Lipschitz functions $\phi_i : H \rightarrow [0,1]$ such that
	\begin{enumerate}
		\item $\supp \phi_i \subseteq 10B(x_i,r)$. \\
		\item $\sum_{i \in I} \phi_i = 1$ on $\bigcup_{i \in I} 8B(x_i,r).$  \\
		\item $|D^m\phi_i(y)| \lesssim_{\Gamma,m} r^{-m}$ for all $y \in H$ and $m \geq 0.$
	\end{enumerate}
	Here, $D^m \phi_i(y)$ denotes the $m^\text{th}$ Frech\'{e}t derivative of the function $\phi_i$ at the point $y.$ 
\end{lem}

\begin{proof}
	Let $\phi : \R \rightarrow \R$ be a standard $C^\infty$ bump function with compact support in $B(0,10)$ so that $\phi(y) = 1$ for all $y \in B(0,8)$, and $0 \leq \phi(y) \leq 1$ for all $y \in \R$. Notice then that $|\phi^{(m)}(y)| \lesssim_m 1$ for all $y \in \R,$ where $\phi^{(m)}$ denotes the $m^\text{th}$ derivative of $\phi.$ For each $i \in I$, define 
	\[ \tilde{\phi}_i(y) = \phi\left( \frac{|y-x_i|}{r} \right). \]
	Since $H$ is a Hilbert space, it has a smooth norm (see for example see \cite{fry2002smooth}). Using this, along with the estimates for $\phi^{(m)}$, we have 
	\begin{align}\label{e:derivphi}
		|D^m \tilde{\phi}_i(y)| \lesssim_m r^{-m}
	\end{align}
	for all $y \in H$ and $m \geq0 .$ Also, $\tilde{\phi}_i$ has compact support in $10B(x_i,r)$ and is identically equal to 1 on $8B(x_i,r)$ and so 
	\begin{align}\label{e:bigphi}
		 \Phi(y) \coloneqq \sum_{i \in I} \tilde{\phi}_i(y) \geq 1, \mbox{ for all } y \in \bigcup_{i \in I} 8B(x_i,r).
	\end{align}
	Since the balls $10B(x_i,r)$ have bounded overlap, we have $\Phi(y) \leq \Gamma.$ Thus, we can define
	\[ \phi_i(y) = \tilde{\phi}_i(y)/\Phi(y). \] 
	In this way, \textit{(1)} follows from the same statement for $\tilde{\phi}_i$, and \textit{(2)} follows from \eqref{e:bigphi} and the definition of $\phi_i$. Finally, \textit{(3)} follows from \eqref{e:derivphi}, \eqref{e:bigphi}, and the fact that the balls $10B(x_i,r)$ have bounded overlap.  
\end{proof}

\begin{lem}\label{l:spu}
Let $(P_0,\{B_{j,k}\},\{P_{j,k}\}), \ k \in \N, \ j \in J_k$, be a CCBP. For each $k \in \N$, the balls $\{10B_{j,k}\}$ have bounded overlap (with constant depending only on $d$), and there exists a collection of smooth Lipschitz function $\{\theta_{j,k}\}_{j \in J_k}$ such that the following holds. 
\begin{enumerate}
\item For each $j \in J_k$,
\begin{align}
\supp \theta_{j,k} \subseteq 10B_{j,k}.
\end{align}
\item For $y \in V_k^8$ we have
\begin{align}
\sum_{j \in J_k} \theta_{j,k}(y) = 1.
\end{align}
\item For each $j \in J_k$ and $y \in H$ we have
\begin{align}
|D^m\theta_{j,k}(y)| \lesssim_d r_{k}^{-m}.
\end{align}
\end{enumerate}
\end{lem}

\begin{proof}
Fix $k \in \N.$ By Lemma \ref{l:tpu}, it suffices to show that the collection of balls $\{10B_{j,k}, \ j \in J_k,$ have bounded overlap (with constant only dependent on $d$). Let $x \in H$ and let $J_k^\prime =\{j_0,j_1,\dots,j_N\} \subseteq J_k$ be such that 
\[ x \in \bigcap_{i=0}^N B(x_{j_i,k},10r_k).\]
Each $j_i \in J_k^\prime$ satisfies $|x_{j_i,k} - x_{j_0,k}| \leq 20r_k \leq 100r_k$, so by Definition \ref{d:CCBP} (3) (for $\ve$ small enough),
\[ \text{dist}(x_{j_i,k},P_{i_0,k}) \leq 100\ve r_k \leq r_k/8. \]
By \eqref{e:csep}, the balls $B(x_{j_i,k},r_k/4), \ j_i \in J_k^\prime$ are disjoint, and so, satisfy the conditions of Lemma \ref{ENV} with $P = P_{j_0,k}$ and $B = B(x,30r_k)$. This gives
\[ Nr_k \sim \sum_{j_i \in J_k^\prime} r_{\frac{1}{4}B_{j_i,k}}^d \lesssim_d r_k \] 
and we conclude $N \lesssim_d 1$. 
\end{proof}

In \cite{david2012reifenberg}, the functions $\theta_{l,k}, \ l \in L_k$ are grouped into a single function $\psi_k$, after which there is no mention of the defining functions $\theta_{l,k}.$ We shall introduce a similar function, also denoted $\psi_k$, which shares the same properties. This may seem artificial, but we hope that it makes comparing the two papers easier. For $k \geq 0$, define a function $\psi_k : H \rightarrow \R$ by setting
\[ \psi_k(y) = 1 - \sum_{j \in J_k} \theta_{j,k}(y). \] 
In this way 
\[ \psi_k(y) + \sum_{j \in J_k} \theta_{j,k}(y) =1 \quad \text{for all} \ y \in H. \] 

Additionally, we have the following properties of $\psi_k.$ Compare these to the corresponding items; \cite[Equation 3.12, Equation 3.15]{david2012reifenberg}. 

\begin{lem}
Let $k \geq 0.$ The function $\psi_k : H \rightarrow \R$ is $C^\infty$ and satisfies 
\[ \psi_k(y) = 0 \quad \text{for all} \ x \in V_k^8 \] 
and 
\[ |D^m \psi_k(y)| \lesssim_m r_k^{-m} \] 
for all $y \in H$ and $m \geq 0.$ 
\end{lem}

\begin{proof}
Since the balls $10B_{j,k},$ $j \in J_k,$ have bounded overlap and each of the functions $\theta_{j,k}$ is $C^\infty,$ it follows that $\psi_k$ are $C^\infty$. The estimates on the derivative follows from bounded overlap with the estimates on the function $\theta_{j,k}.$ 
\end{proof}

\subsection{The Reifenberg map}
For $k \geq 0,$ we define a function $\sigma_k : H \rightarrow H$ by setting 
\[ \sigma_k(y) = y + \sum_{j \in J_k} \theta_{j,k}(y) [ \pi_{j,k}(y) - y ] = \psi_k(y)y + \sum_{j \in J_k} \theta_{j,k}(y) \pi_{j,k}(y), \]
and a function $f_k:P_0 \rightarrow H$ by setting
\[ f_0(y) = y \quad \text{and} \quad f_{k+1}= \sigma_k \circ f_k. \]  
Let $y \in H.$ Since $\supp \theta_{j,k} \subseteq 10B_{j,k},$ if $j \in J_k$ is such that $\theta_{j,k}(y) \not=0,$ then $y \in 10B_{j,k}.$ This implies $|y - \pi_{j,k}(y)| \leq 10r_k$, which, along with the bounded overlap of $\{10B_{j,k}\}$, implies 
\[ |\sigma_k(y) - y| \leq 10r_k. \] 
It follows that
\begin{align}\label{e:uniform}
 ||f_{k+1} - f_{k}||_\infty \leq 10r_k,
 \end{align}
so the functions $f_k$ converge uniformly to a continuous function $f$ on $P_0.$ We now state the Reifenberg parametrization theorem. We include each of the properties used in \cite{azzam2018analyst}. 

\begin{thm}\label{DT}
Let $(P_0,\{B_{j,k}\},\{P_{j,k}\})$ be a CCBP in $H$. Recall the definition,  
\[\varepsilon_k(x) = \sup\{d_{x_{i,l},10^4r_l}(P_{j,k},P_{i,l}) : j \in J_k, \ \abs{l-k} \leq 2, \ i \in J_k, x \in 100B_{j,k} \cap 100B_{i,l} \}.\]
There is $\ve_0>0$ such that if $\ve \in (0,\varepsilon_0)$ and $\ve_k(x_{j,k}) < \ve$ for all $k\geq0$ and $j \in J_k$, then the map $f:P_0 \rightarrow H$ (from above) is bijective on its image and the following hold.
\begin{enumerate}
\item We have 
\begin{equation}
E_\infty = \bigcap_{K=1}^\infty \overline{\bigcup_{k=K}^\infty \{x_{j,k}\}_{j \in J_k}} \subseteq \Sigma = f(P_0).
\end{equation}
\item $f(x) =x$ when $\emph{dist}(z,P_0) >2.$
\item There is $0 < \tau < 1$ so that if $x,y \in P_0$, and $|x-y| \leq 1$, then
\begin{equation}
\frac{1}{4}|x-y|^{1+\tau} \leq |f(x)-f(y)| \leq 10|x-y|^{1-\tau}. 
\end{equation}
\item $|f(x)-x| \lesssim \ve$ for $x \in H.$
\item For $x \in P_0,$ 
\[f(x) = \lim_k \sigma_k \circ \dots \circ \sigma_0.\]
\item For $k \geq 0,$
\begin{equation}
\sigma_k(y) = y \ \text{for} \ y \in H \setminus V^{10}_k.
\end{equation}
\item Let $\Sigma_0 = P_0$ and 
\begin{equation}
\Sigma_k = \sigma_k(\Sigma_{k-1}).
\end{equation}
There is a Lipschitz function $A_{j,k}:P_{j,k} \cap 49B_{j,k} \rightarrow P_{j,k}^\perp$ such that 
\begin{align}
	|A_{j,k}(x_{j,k})| &\lesssim \ve r_k, \\
	|DA_{j,k}(y)| &\lesssim \ve  \quad \text{for} \ y \in P_{j,k} \cap 49B_{j,k},
\end{align}
and if $\Gamma_{j,k}$ is its graph over $P_{j,k}$, then 
\begin{equation}
\Sigma_k \cap D(x_{j,k},P_{j,k},49r_k) = \Gamma_k \cap D(x_{j,k},P_{j,k},49r_k)
\end{equation}
where
\begin{equation}
D(x_{j,k},P_{j,k},49r_k) = \{z+w : z \in P \cap B(x,r), \ w \in P^\perp \cap B(0,r)\}.
\end{equation}
\item For $k \geq 0$ and $y \in \Sigma_k$, there is an affine $d$-plane $P$ through $y$ and a $C\ve$-Lipschitz and $C^2$-function $A:P \rightarrow P^\perp$ so that if $\Gamma$ is the graph of $A$ over $P$, then 
\begin{equation}
\Sigma_k \cap B(y,19r_k) = \Gamma \cap B(y,19r_k). 
\end{equation}
\item The surface $\Sigma = f(P_0)$ is $C\ve$-Reifenberg flat in the sense that for all $z \in \Sigma,$ and $t \in (0,1),$ there is $P=P(z,t)$ so that $d_{B(z,t)}(\Sigma,P) \leq \ve.$ 
\item For all $y \in \Sigma_k$,
\begin{align}\label{Sigma_ksigma_k}
|\sigma_k(y) - y| \lesssim \ve r_k
\end{align}
and moreover,
\begin{align}
\emph{dist}(y,\Sigma) \lesssim \ve r_k,\  \text{for} \ y \in \Sigma_k.
\end{align}
\item For $k \geq 0$ and $y \in \Sigma_k \cap V_k^8,$ choose $i \in J_k$ such that $y \in 10B_{i,k}.$ Then
\begin{align}
|\sigma_k(y) - \pi_{i,k}(y)| \lesssim \ve_k(y) r_k. 
\end{align}
\item For $x,y \in \Sigma_k$, 
\[ \angle(T\Sigma_k(x),T\Sigma_k(y)) \lesssim \ve \frac{|x-y|}{r_k}, \]
where $T\Sigma_k(x)$ denotes the tangent space at $x \in \Sigma_k.$ 
\item For $x \in \Sigma$ and $r>0$, 
\[\mathscr{H}^d_\infty(\Sigma \cap B(x,r)) \gtrsim r^d. \]
\end{enumerate}
\end{thm} 

It is possible to show that $f$ is bi-Lipschitz, when there is better control on the angles between the planes $P_{j,k}.$ This is made precise below. 

\begin{thm}[{\cite[Theorem 2.5]{david2012reifenberg}}]\label{LipDT}
With the same notation and assumptions as Theorem \ref{DT}, assume additionally that there exists $M < \infty$ such that
\[ \sum_{k \geq0} \varepsilon_k^\prime(f_k(z))^2 \leq M \ \text{for} \ z \in P_0\]
with
\begin{align*}
\varepsilon_k^\prime(x) = \sup\{\d_{x_{i,l},10^4r_l}(P_{j,k},P_{i,l}) : j \in J_k, \ \abs{l-k} \leq 2, \ i \in J_l, x \in 10B_{j,k} \cap 10B_{i,l} \}.
\end{align*}
Then $f \colon P_0 \rightarrow \Sigma$ is $C(M,d)$-bi-Lipschitz. 
\end{thm}

With the modifications outlined above, most of the arguments from \cite{david2012reifenberg} give the required result. For the reader's convenience, we include the key details for the above results in the Appendix.

\section{Estimating curvature for Reifenberg flat sets}\label{Sec:StopTime}

In this section we prove Theorem \ref{Thm3}. As mentioned in the introduction, it suffices that prove Theorem \ref{Thm3} for subsets of $\R^n,$ as long as the constants are independent of $n$. The reason we prove Theorem \ref{Thm3} for subsets of Euclidean space, rather than working directly in Hilbert space, is so that we can state directly from \cite{azzam2018analyst} any result we may need. While most of the preliminary results from \cite{azzam2018analyst} remain true in Hilbert space, we wanted to save the reader from having to check this.  

\begin{lem}\label{l:main-red}
	Suppose for any $n \geq 2$, $1 \leq d < n$, $1 \leq p < p(d)$ and $C_0 > 1$, that there exists $\ve > 0$ so that \eqref{e:dTSP3} holds with constants $d,p$, and $C_0,$ for any $(\ve,d,\delta_0)$-Reifenberg flat surface $\Sigma \subseteq \R^n$, with constants independent of $n$. Then, there exists $C \geq 1$ so that \eqref{e:dTSP3} holds with constant $d,p,$ and $C_0$ for any $(C\ve,d,\delta_0)$-Reifenberg flat surface $\Sigma \subseteq H.$ 
\end{lem}

Let $n \geq 2,$ $1 \leq d < n$, $1 \leq p < p(d)$ and $C_0 > 1$. Let $\ve > 0$ be a constant to be chosen sufficiently small and $\Sigma \subseteq \R^n$ be an $(\ve,d,\delta_0)$-Reifenberg flat surface. Let $X_k$ be a sequence of maximally $\rho^k$-separated nets (with $\rho$ as in Theorem \ref{cubes}) and let $\mathscr{D}$ be the Christ-David for $\Sigma$ from Theorem \ref{cubes}. Since $\Sigma \subseteq \R^n$, Theorem \ref{cubes}(4) now holds. Let $Q_0 \in \mathscr{D}$ and assume without loss of generality that $0 \in Q_0 \in \mathscr{D}_0$. From now on, when we write $\mathscr{D}$ we mean those cubes $Q$ in $\mathscr{D}$ so that $Q \subseteq Q_0.$ We may also assume $\mathscr{H}^d(Q_0) < \infty$ since otherwise the statement is trivial. We are left to show
\begin{align}\label{e:dTSP3'}
\sum_{\substack{Q \in \mathscr{D} \\ Q \subseteq Q_0}} \beta^{d,p}_\Sigma(C_0B_Q)^2\ell(Q)^d \lesssim \mathscr{H}^d(Q_0)
\end{align}
for $\Sigma \subseteq \R^n$ with constant independent of $n$. Since we are now working in $\R^n$, we can state those results in Sections 5,6 and 7 \cite{azzam2018analyst} whose constants are independent of $n$. The most import of these is the construction of a sequence of approximating surfaces $\Sigma^N$ (Definition \ref{d:sigmaN}) for $\Sigma$ and the local Lipschitz description of these approximating surfaces (Lemma \ref{Graph2}). The basic idea is to push the $\beta$-number estimates for $\Sigma$ onto $\beta$-number estimates for $\Sigma^N$ using Lemma \ref{betaest}. Because (locally) $\Sigma^{N}$ is approximately a Lipschitz graph over the previous surface, we can apply Dorronsoro's Theorem to control the $\beta$-number on $\Sigma^N.$ The error term in Lemma \ref{betaest} is controlled by Lemma \ref{l:Azzam-Stop}(5).

\subsection{Setup and definitions}

Let $M>1$ be a constant to be chosen sufficiently large. For each $Q \in \mathscr{D},$ let $P_Q^\prime$ be the $d$-plane such that
\[\beta_\Sigma^{d,p}(MB_Q) = \beta_\Sigma^{d,p}(MB_Q,P_Q^\prime) < \ve \] 
and $P_Q$ be the $d$-plane through $x_Q$ such that 
\[\beta_\Sigma^{d,p}(MB_Q,P_Q) \leq 2\beta_\Sigma^{d,p}(MB_Q) < 2\ve. \] 
Let $P_0 = P_{Q_0}.$ Let $\alpha > 0.$  

\begin{defn}[The stopping-time condition] For $Q \in \mathscr{D}$ we construct a stopping-time region $S_Q$ as follows. First, add $Q$ to $S_Q.$ Then, inductively, we add cubes $R$ to $S_Q$ if:
\begin{itemize}
\item $R^{(1)} \in S_Q.$ 
\item We have $\angle(P_T,P_Q) < \alpha$ for all siblings $T$ of $R$.  
\end{itemize}
\end{defn}

\subsection{Families of cubes} For each non-negative integer $N$, we define collections of cubes Layer($N$), Stop($N$), and Up($N$), and a collection of stopping-time regions $\mathscr{F}^N$ as follows. Let $\tau \in (0,1)$ be a constant to be chosen sufficiently small. First, set 
\[\text{Stop}(-1) = \mathscr{D}_0.\] 
Assume that Stop($N-1$) has been defined for some $N \geq 0.$ Let $\text{Layer}(N)$ be the collection of cubes from the stopping-time regions defined by those cubes in $\text{Stop}(N-1),$ that is, 
\begin{align*}
\text{Layer}(N) = \bigcup \{S_Q : Q \in \text{Stop}(N-1) \}.
\end{align*}  
Then, define $\text{Stop}(N)$ to be 
\begin{align*}
\text{Stop}(N) = \{Q \in \mathscr{D} : Q \ \text{max. such that there is} \ &Q^\prime \in \text{Child}(Q^{(1)}) \ \text{with} \ \ell(Q^\prime) < \tau d_{\text{Layer}(N)}(Q^\prime)\}. 
\end{align*}

\begin{rem}
	Note the non-standard notation. At the minute, $\text{Stop}(N)$ is not used to denote the minimal cubes of some stopping-time region but rather a collection of disjoint cubes spread out over $\Sigma$.
\end{rem}

The cubes $\text{Up}(N)$ will be those $Q \in \mathscr{D}$ which contain some cube in $\text{Stop}(N).$ Define $\text{Up}(-1) = \emptyset$. Then, assuming Up($N-1$) has been defined for some $N \geq 0$, let  
\begin{align*}
\text{Up}(N) = \text{Up}(N-1) \cup \{Q \in \mathscr{D} : Q \supseteq R \ \text{for some} \ R \in \text{Layer}(N) \cup \text{Stop}(N) \}.
\end{align*}

For each $Q \in \text{Stop}(N-1)$ define an extension of the stopping time region $S_Q$ by setting, 
\[S_Q^\prime = \{R \in \text{Up}(N) : R \subseteq Q\} \supseteq S_Q.\]
In particular, we add to $S_Q$ all those cubes between $\min(S_Q)$ and $\text{Stop}(N).$ It is easy to check that $S_Q^\prime$ still defines a stopping-time region in the sense of Definition \ref{StoppingTime}. Finally, for $N \geq 0,$ set 
\[\mathscr{F}^N = \{S_Q^\prime : Q \in \text{Stop}(N-1)\}.\]
By construction
\begin{align}\label{e:partition}
\mathscr{D} = \bigcup_{N \geq 0} \mathscr{F}^N.
\end{align}
The reason Azzam and Schul work with $S_Q'$ and not just $S_Q$ is because they want to ensure nearby minimal cubes have comparable sizes. For us, this is important as it allows us to prove minimal cubes have bounded overlap with constant only depending on $d$. If we work with the regular stopping-time regions it, is possible for nearby cubes to vary wildly in size. 

\begin{lem}\label{l:Azzam-Stop}
	Let $N \geq 0$, $C_1 > 1$ and $0 < \tau < \min\{\tau_0, (16 + 8C_1)^{-1}\},$ where $\tau_0 = \frac{1}{4(2+\rho^{-1})}$.
	\begin{enumerate}
		\item  If $Q,R \in \emph{Stop}(N)$ are such that $C_1B_Q \cap C_1B_R \not=\emptyset$, then $\ell(Q) \sim \ell(R).$ 
		\item The balls $C_1B_Q, \ Q \in \emph{Stop}(N)$, have bounded overlap in $\Sigma$, with constant depending on $d.$ 
		\item If $x \in \Sigma$ and $d_{\emph{Up}(N)}(x) > 0$, then there is $Q \in \emph{Stop}(N)$ so that $x \in Q.$ 
		\item If $Q \in \emph{Stop}(N)$ then 
		\[\ell(Q) \sim d_{\emph{Up}(N)}(Q). \]
		\item We have 
		\[ \sum_{N \geq0 } \sum_{Q \in \emph{Stop}(N)} \ell(Q)^d \lesssim_{d,\alpha,\ve} \mathscr{H}^d(Q_0). \] 
	\end{enumerate}
\end{lem} 

\begin{proof}
	Items \textit{(1)} and \textit{(4)} can be found in \cite[Section 5]{azzam2018analyst}, and item \textit{(5)} is \cite[Proposition 8.1]{azzam2018analyst}. Item \textit{(3)} can also be found in \cite[Section 5]{azzam2018analyst}, however this relies on \cite[Corollary 5.8]{azzam2018analyst}, whose proof contains an error. Thus, we shall prove both \textit{(2)} and \textit{(3)}.
	
	Let us start with \textit{(2)}. Fix $x \in \Sigma$ and let 
	\[\mathscr{D}_x = \{ Q \in \text{Stop}(N) : x \in Q \}.\]
	If $Q,Q^\prime \in \mathscr{D}_x$, by \textit{(1)}, we have $\ell(Q) \sim \ell(Q^\prime).$  Let $r = \sup_{Q \in \mathscr{D}_x} \ell(Q).$ It follows that $\{ c_0B_Q : Q \in \mathscr{D}_x\}$ are pairwise disjoint balls contained in $B = B(x,(c_0 + C_1)r).$ For $\ve > 0$ small enough, recalling that $\Sigma$ is $(\ve,d)$-Reifenberg flat, we can find a $d$-plane $P$ such that
	\begin{align*}
		\text{dist}(x_Q,P) \leq \ve (c_0 +C_1)r \leq \frac{c_0 \ell(Q)}{2}
	\end{align*}
	for all $Q \in \mathscr{D}_x.$ So, by Lemma \ref{ENV},
	\begin{align*}
		\#\mathscr{D}_x r^d \lesssim \sum_{Q \in \mathscr{D}_x} r_{c_0B_Q}^d \lesssim_d r^d.
	\end{align*}
	In particular, $\#\mathscr{D}_x \lesssim_d 1.$
	
	Now for \textit{(3)}. We shall prove the contrapositive. Fix $x \in \Sigma$ and suppose $x$ is not contained in any cube from $\text{Stop}(N).$ Let $R_j \in \mathscr{D}$ be a sequence of cubes such that $x \in R_j$ and $\ell(R_j) \downarrow 0.$ By assumption, each $R_j$ is not contained in any cube from $\text{Stop}(N)$, and so $\ell(R_j) \geq \tau d_{\text{Layer}(N)}(R_j)$. For each $j$, let $x_j \in R_j$ be so that $d_{\text{Layer}(N)}(x_j) \leq d_{\text{Layer}(N)}(R_j) + j^{-1}.$ Then  
	\begin{align}
		d_{\text{Layer}(N)}(x) &\leq |x - x_{j}| + d_{\text{Layer}(N)}(x_j) \leq |x - x_{j}| + d_{\text{Layer}(N)}(R_j) + j^{-1} \\
		&\leq (2 + \tau^{-1})\ell(R_j) + j^{-1}  \rightarrow 0,
	\end{align}
	and so $d_{\text{Layer}(N)}(x) = 0.$ Since $\text{Layer}(N) \subseteq \text{Up}(N),$ this implies $d_{\text{Up}(N)} = 0$ as required.    	
\end{proof}

\subsection{The sequence of approximating surfaces}

We can now construct the sequence of approximating surfaces alluded to at the beginning of this section. For each $N \geq 0$ we will construct a surface $\Sigma^N$ using Theorem \ref{DT}, where we use the centres of the cubes from $\text{Up}(N)$ as our nets. 

\begin{defn}
For $k \geq 0$ and integer, let $s(k)$ be such that $5\rho^{s(k)} \leq r_k \leq 5\rho^{s(k)-1}.$ We define $\text{Up}(N)_k = \mathscr{D}_{s(k)} \cap \text{Up}(N)$ and let $\mathscr{X}_k^N = \{x_{j,k}\}_{j \in J^N_k}$ be a maximally $r_k$ separated net for
\[\mathscr{C}_k^N \coloneqq \{x_Q : Q \in \text{Up}(N)_k\}. \]
\end{defn}

In \cite{azzam2018analyst}, it was shown that the points $\mathscr{X}_k^N, \ k \geq 0$ satisfy the conditions of Theorem \ref{DT} for each $N \geq 0.$ 

\begin{defn}\label{d:sigmaN}
Let $P_0 = P_{Q_0}$ and for $N \geq0$, let $\Sigma^N, \Sigma_k^N, \  \sigma^N, \ \sigma_k^N$ be the resulting surfaces and functions from Theorem \ref{DT}, respectively.
\end{defn} 

\begin{lem}[{\cite[Lemma 6.2]{azzam2018analyst}}]\label{d=0}
If $d_{\emph{Up}(N)}(x) = 0$ then $x \in \Sigma \cap \Sigma^N.$
\end{lem}

\begin{lem}[{\cite[Lemma 6.7]{azzam2018analyst}}]\label{DistSigma}
Let $x \in \Sigma,$ then 
\begin{align}\label{e:dSS^N}
\dist(x,\Sigma^N) \lesssim \frac{\ve}{\tau} d_{\emph{Up}(N)}(x). 
\end{align}
\end{lem}

The following is a local Lipschitz description of the surfaces $\Sigma^N.$ 

\begin{lem}[{\cite[Theorem 6.10]{azzam2018analyst}}]\label{Graph2}
Let $C_2 > 1$. There exists $M \geq 1$ large enough and $\tau$ small enough so that the following holds. Let $Q \in S \in \mathscr{F}^N$ and $P$ be a $d$-plane such that $d_{B_Q}(P,P_Q) < \theta.$ If $\alpha,\theta$ and $\ve$ are small enough (depending on $C_2$), then there is a $C(\alpha + \theta)$-Lipschitz map $A_{P,Q}^N : P \rightarrow P^\perp$ that is zero outside $P \cap \pi (B(x_Q, 2C_2\ell(Q)))$ such that if $\Gamma^N_{P,Q}$ is the graph of $A_{P,Q}^N$ along $P$, then 
\begin{align}
\Sigma^N \cap D(x_Q,P,C_2\ell(Q)) = \Gamma_{P,Q}^N \cap  D(\pi(x_Q),P,C_2\ell(Q)).
\end{align}
If $Q =Q(S),$ we will set $A_S = A_{P,Q(S)}^N$ and $\Gamma_S = \Gamma_{P,Q(S)}^N.$ 
\end{lem}

\subsection{Dorronsoro's Theorem}\label{s:Dor} We describe the Dorronsoro-type theorem we will use. This is due to Hyt\"onen, Li and Naor \cite{hytonen2016quantitative}. For a Lipschitz function $f : \R^d \rightarrow \R^n$, let $\norm{f}_\text{Lip}$ denote its Lipschitz constant, and for $B \subseteq \R^d$ and $1 < p < \infty$, set
\begin{align}\label{e:defofA}
	\Omega_{f,p}(B) = \inf_A \left(  \fint_B \left( \frac{|f(y) - A(y)|}{r_B}\right)^p \, dy \right)^\frac{1}{p} 
\end{align}
where the infimum is taken over all affine maps $A : \R^d \rightarrow \R^n.$ Additionally, for $p = \infty$, define 
\[\Omega_{f,\infty}(B) = \inf_A \sup_{y \in B} \frac{|f(y) - A(y)|}{r_B}. \]
As a consequence of \cite[Theorem 40]{hytonen2016quantitative}, we get the following.

\begin{thm}\label{ThmDor}
Let $2 \leq p < \infty$ and $f: \R^d \rightarrow \R^n$ be a Lipschitz function with bounded support. Then 
\begin{align}
\int_{\R^n}\int_0^\infty \Omega_{f,p}(B(x,r))^2 \, \frac{dr}{r}dx \lesssim_{d,p} \diam ( \supp f)^d \norm{f}^2_{\emph{Lip}} 
\end{align}
\end{thm}

We can discretize the above theorem as follows. The proof is standard. 

\begin{cor}\label{CorDor}
Let $2 \leq p < \infty$ and $f: \R^d \rightarrow \R^n$ be a Lipschitz function with bounded support. Let $\mathscr{I}$ denote the collection of dyadic cubes in $\R^d$ and for each $I \in \mathscr{I},$ let $B_I = B(x_I,\diam(I)),$ where $x_I$ denotes the centre of $I$. Then  
\begin{align}
\sum_{I \in \mathscr{I}} \Omega_{f,p}(3B_I)^p \ell(I)^d \lesssim_{d,p}  \diam ( \supp f)^d \norm{f}^2_{\emph{Lip}}.
\end{align}

\end{cor}

Later on, it will be useful for us to bound $\Omega_{f,\infty}$ in terms of $\Omega_{f,p},$ in analogy with Lemma \ref{lemma:betap_betainfty}. This will be possible if $f$ is bi-Lipschitz on its image. 

\begin{lem}\label{l:omega_infty}
Let $2 \leq p < \infty$ and $f: \R^d \rightarrow \R^n$ be a function which is bi-Lipschitz onto its image, with bi-Lipschitz constant $L$. Then, for $B \subseteq \R^d,$ 
\[ \Omega_{f,\infty}(\tfrac{1}{2}B) \lesssim_L \Omega_{f,1}(B)^\frac{1}{d+1} \lesssim_p \Omega_{f,p}(B)^\frac{1}{d+1}. \] 
\end{lem}

\begin{proof}
We can assume without loss of generality that $r_B=1.$ Let us begin with the left-most inequality. Let $\ve > 0$ and $A :\R^d \rightarrow H$ be an affine function so that 
\[ \fint_B |f(y) - A(y)| \, dx  \leq   \Omega_{f,1}(B) + \ve. \] 
Let $y \in \tfrac{1}{2}B$ be the point so that $\tau \coloneqq \dist(f(y),A(\R^d))$ is maximal. By the triangle inequality, $B(f(y),\tau/2) \subseteq \{x \in \R^n : \dist(x,A(\R^d)) > t \}$ for all $0 < t < \tau/2.$ Also, using that $f$ is bi-Lipschitz, for each $z \in B(y,\tau/2L) \subseteq B$ we know $f(z) \in B(f(y),\tau/2).$ Hence,
\[ B(y,\tau/2L) \subseteq \{x \in B : \dist(f(x),A(\R^d)) > t\} \subseteq \{x \in B : |f(y) - A(y)| > t \} \] 
for all $0 < t < \tau/2$, and so, 
\begin{align}
	1 \lesssim_L \frac{\mathscr{H}^d(B(y,\tau/2L))}{\tau^d} &\lesssim \frac{1}{\tau^{d+1}} \int_0^\frac{\tau}{2} \mathscr{H}^d(\{ x \in B :|f(y) - A(y)| > t \}) \, dt \\
	&\leq \frac{1}{\tau^{d+1}} \int_B |f(y) - A(y)| \, dx \leq \frac{1}{\tau^{d+1}}\left[ \Omega_{f,1}(B) + \ve \right].  
\end{align}
Taking $\ve \rightarrow 0$ and rearranging gives the first inequality since
\[ \Omega_{f,\infty}(\tfrac{1}{2}B) = \tau \lesssim \Omega_{f,1}(B)^\frac{1}{d+1}. \]    
We now prove the second inequality. The proof is similar to the proof of Lemma \ref{c:beta}. Let $\ve > 0$ and let $A : \R^d \rightarrow \R^n$ be an affine function so that
\[  \left(\fint_B |f(y) - A(y)|^p \, dx \right)^\frac{1}{p}  \leq   \Omega_{f,p}(B) + \ve .\]
The function $x \mapsto x^p$ is convex, so,  by Jensen's inequality, 
\begin{align}
\Omega_{f,1}(B) \leq \fint_B |f(y) - A(y)| \, dx \lesssim \left( \fint_B |f(y) - A(y)|^p \, dx \right)^\frac{1}{p} \leq  \Omega_{f,p}(B) + \ve.
\end{align}
Since $\ve >0$ was arbitrary, this finishes the proof. 
\end{proof}

If $\Sigma$ is a bi-Lipschitz image of $\R^d$, we can give an upper bound for the $\beta$-numbers of $\Sigma$ by a corresponding $\Omega$-number. 

\begin{lem}\label{BetaBound}
Let $2 \leq p < \infty,$ $f: \R^d \rightarrow \R^n$ be an $L$-bi-Lipschitz function and $\Sigma = f(\R^d).$ Let $x \in \Sigma, \ r > 0,$ and let $B$ be a ball in $\R^d$ that contains $f^{-1}(\Sigma \cap B(x,r)).$ Then
\begin{align}
\beta_\Sigma^{d,p}(B(x,r)) \lesssim \left( \frac{\mathscr{H}^d(I)}{r^d}\right)^\frac{1}{p}\Omega_{f,p}(B)  .
\end{align}
\end{lem}

\begin{proof}
	Let $A$ be an affine function which infimises \eqref{e:defofA} for $f$, $p$ and $B$. Since $f$ is bi-Lipschitz, $\Sigma$ is Ahlfors $d$-regular (with constant depending on $L$) and so $\mathscr{H}^d(E) \sim_L \mathscr{H}^d_\infty(E)$ for all $E \subseteq \Sigma.$ Hence, 
	\begin{align}
		\beta^{d,p}_\Sigma(B(x,r))^p &\lesssim \frac{1}{r^d} \int_0^1 \mathscr{H}^d_\infty(\{y \in \Sigma \cap B(x,r) : \dist(y,A(\R^d) > tr \} ) t^{p-1} \, dt \\
		&\lesssim \frac{1}{r^d} \int_0^1 \mathscr{H}^d(\{y \in \Sigma \cap B(x,r) : \dist(y,A(\R^d) > tr \} ) t^{p-1} \, dt \\
		&= \frac{1}{r^d} \int_{\Sigma \cap B(x,r)} \left( \frac{\dist(y,A(\R^d))}{r} \right)^p \, d\mathscr{H}^d(y) \\
		&= \frac{1}{r^d} \int_{f^{-1}(\Sigma \cap B(x,r))} \left( \frac{\dist(f(z),A(\R^d))}{r} \right)^p J_f(z) \, dz \\
		&\lesssim_L \frac{1}{r^d} \int_{B} \left( \frac{|f(z) - A(z)|}{r} \right)^p  \, dz \\
		&= \frac{\mathscr{H}^d(B)}{r^d} \Omega_{f,p}(B)^p.
	\end{align} 
\end{proof}

We now have all the ingredients to finish the proof of Theorem \ref{Thm3}.

\begin{proof}[Proof of Theorem \ref{Thm3}]
We use the collected facts to prove \eqref{e:dTSP3'}. Notice, by \eqref{e:partition}, we have
\begin{align}
\sum_{Q \in \mathscr{D}} \beta_\Sigma^{d,p}(C_0B_Q)^2\ell(Q)^d = \sum_{N \geq 0} \sum_{S \in \mathscr{F}^N}\sum_{Q \in S} \beta_\Sigma^{d,p}(C_0B_Q)^2\ell(Q)^d.
\end{align}
We begin by bounding the sum over cubes in a particular stopping-time region. Let $N \geq 0$ and $S \in \mathscr{F}^N.$ For $Q \in S,$ let $x_Q^\prime \in \Sigma^N$ be the closest point to $x_Q$ and let $B^\prime_Q = B(x_Q^\prime , \ell(Q)).$ By Lemma \ref{DistSigma}, $C_0B_Q \subseteq 2C_0B_Q'$, so Lemma \ref{betaest} implies
\begin{align*}
\sum_{Q \in S}\beta_\Sigma^{d,p}(C_0B_Q)^2\ell(Q)^d &\lesssim \sum_{Q \in S}\beta_{\Sigma^N}^{d,p}(2C_0B_Q^\prime)^2\ell(Q)^d \\
&\hspace{-4em} + \sum_{Q \in S}\left(\frac{1}{\ell(Q)^d}\int_{\Sigma \cap 2C_0B_Q} \left( \frac{\text{dist}(y,\Sigma^N)}{\ell(Q)}\right)^p \, d\mathscr{H}_\infty^d(y) \right)^\frac{2}{p}\ell(Q)^d \\
&= I_1 + I_2. 
\end{align*}
The proof now splits into two claims.
\bigbreak
\noindent
\textbf{Claim 1:} 
\begin{align}
I_2 \lesssim \ve \ell(Q(S))^d. 
\end{align}

\begin{proof}[Proof of Claim 1]
By Lemma \ref{JensenPhi}, we may assume that $p \geq 2.$ Let $y \in \Sigma.$ Assume first that there exists $R \in \text{Stop}(N)$ such that $y \in R.$ By Lemma \ref{DistSigma}, and using that fact that $R \in \text{Up}(N)$ (since $\text{Stop}(N) \subseteq \text{Up}(N)$), we have
\begin{align}\label{e:deR}
\text{dist}(y,\Sigma^N) \stackrel{\eqref{e:dSS^N}}{\lesssim} \ve d_{\text{Up}(N)}(y) \leq \ve \left( \ell(R) + \dist(y,R) \right) = \ve \ell(R).
\end{align}
On the other hand, if there is no $R \in \text{Stop}(N)$ such that $y \in R,$ then by Lemma \ref{l:Azzam-Stop}(3) we have $d_{\text{Up}(N)}(y) = 0.$ In particular, $\text{dist}(y,\Sigma^N) = 0$ by Lemma \ref{d=0}. Using the above, the bounded overlap of the cubes in $\text{Stop}(N)$ (Lemma \ref{l:Azzam-Stop}(2)), and Lemma \ref{l:subsum}, we get  
\begin{align*}
I_2 &\leq \sum_{Q \in S}\left(\frac{1}{\ell(Q)^d}\int_{\Sigma \cap 2C_0B_Q} \sum_{\substack{R \in \text{Stop}(N) \\ R \cap 2C_0B_Q \not= \emptyset}} \mathds{1}_R(y) \left( \frac{\text{dist}(y,\Sigma^N)}{\ell(Q)}\right)^p \, d\mathscr{H}_\infty^d(y) \right)^\frac{2}{p}\ell(Q)^d \\
&\lesssim \sum_{Q \in S}\left(\frac{1}{\ell(Q)^d} \sum_{\substack{R \in \text{Stop}(N) \\ R \cap 2C_0B_Q \not= \emptyset}}\int_{R}  \left( \frac{\text{dist}(y,\Sigma^N)}{\ell(Q)}\right)^p \, d\mathscr{H}_\infty^d(y) \right)^\frac{2}{p}\ell(Q)^d \\
&\stackrel{\eqref{e:deR}}{\lesssim} \ve \sum_{Q \in S} \left(\sum_{\substack{R \in \text{Stop}(N) \\ R \cap 2C_0B_Q \not= \emptyset}}\frac{\ell(R)^{d+p}}{\ell(Q)^{d+p}} \right)^\frac{2}{p} \ell(Q)^d \lesssim \ve \sum_{Q \in S}\sum_{\substack{R \in \text{Stop}(N) \\ R \cap 2C_0B_Q \not= \emptyset}} \frac{\ell(R)^{d\frac{2}{p}+2}}{\ell(Q)^{d(\frac{2}{p}-1) +2}},
\end{align*}
where the final inequality follows from the fact that we have assumed $p \geq 2.$

 Our next step is to swap the order of summation. First, for fixed $R \in \text{Stop}(N)$ and $k \in \Z,$ we show
\begin{align}\label{vol1}
\#\mathscr{D}_{R,k} = \#\{Q \in \mathscr{D}_k \cap S : R \cap 2C_0B_Q \not= \emptyset \} \lesssim 1.
\end{align}
Indeed, let $R \in \text{Stop}(N)$ and $k \in \Z.$ For any $Q \in \mathscr{D}_{R,k}$ we have $Q \in \text{Up}(N)$ since $Q \in S \in \mathscr{F}^N$. Combining this with Lemma \ref{l:Azzam-Stop}(4), and the fact that $R \cap 2C_0B_Q \not=\emptyset$, we get
\begin{align}\label{e:RleqQ}
\ell(R) \sim d_{\text{Up}(N)}(R) \leq \ell(Q) + \dist(Q,R) \leq 3C_0\ell(Q).
\end{align}
Thus, there exists a constant $C >0$ such that $c_0B_Q \subseteq B(x_R,C\ell(Q)) \coloneqq B$. Since $\Sigma$ is $(\ve,d)$-Reifenberg flat, we can find a $d$-plane $P$ through $x_B$ such that $\dist(y,P) \lesssim \ve \ell(Q)$ for all $y \in \Sigma \cap B.$ In particular, 
\[\text{dist}(x_Q,P) \leq \tfrac{c_0\ell(Q)}{2}\]
for $\ve$ small enough. Since the balls $\{c_0B_Q\}_{Q \in \mathscr{D}_{R,k}}$ are disjoint, the above implies that they satisfy the conditions of Lemma \ref{ENV}, which gives
\begin{align*}
\#\mathscr{D}_{R,k} (5\rho^k)^d \lesssim \sum_{Q \in \mathscr{D}_{R,k}}r_{c_0B_Q}^d \lesssim r_B^d \lesssim (5\rho^k)^d
\end{align*}
and proves \eqref{vol1}.

We return to estimating $I_2.$ By our restrictions on $p$ we have $p < 2d/(d-2),$ or equivalently, $d(2/p -1) +2 >0.$ Combining this, \eqref{vol1}, and \eqref{e:RleqQ}, we can swap the order of integration and sum over a geometric series to get,
\begin{align*}
I_2 &\stackrel{\eqref{e:RleqQ}}{\lesssim} \ve \sum_{\substack{R \in \text{Stop}(N) \\ R \cap 2C_0B_{Q(S)} \not = \emptyset}} \sum_{\substack{k \in \Z \\ 5\rho^k \gtrsim \ell(R)}} \sum_{\substack{Q \in \mathscr{D}_k \cap S \\ R \cap 2C_0B_Q \not= \emptyset}}  \frac{\ell(R)^{d\frac{2}{p}+2}}{\ell(Q)^{d(\frac{2}{p}-1) +2}} \\
&\stackrel{\eqref{vol1}}{\lesssim} \ve \sum_{\substack{R \in \text{Stop}(N) \\ R \cap 2C_0B_{Q(S)} \not = \emptyset}} \ell(R)^{d\frac{2}{p}+2} \sum_{\substack{k \in \Z \\ 5\rho^k \gtrsim \ell(R)}} \rho^{-k [ d(\frac{2}{p}-1) +2]} \\
&\lesssim \ve \sum_{\substack{R \in \text{Stop}(N) \\ R \cap 2C_0B_{Q(S)} \not = \emptyset}}\ell(R)^d.
\end{align*}
Now, suppose $R \in \text{Stop}(N)$ is such that $R \cap 2C_0B_{Q(S)} \not= \emptyset.$ Let $x^\prime_R$ be the point in $\Sigma^N$ closest to $x_R,$ and let $P_S = P_{Q(S)}$ (recall the definition of $P_Q$ from the beginning of Section \ref{Sec:StopTime}). Since 
\begin{align*}
|x_R - x_R^\prime| \lesssim \ve d_{\text{Up}(N)}(x_R) \leq \ve \ell(R) \stackrel{\eqref{e:RleqQ}}{\lesssim} C_0\ve\ell(Q(S)),
\end{align*}
taking $C_2$ (the constant appearing in the statement of Lemma \ref{Graph2}) larger than $4C_0$ and $\ve > 0$ small enough, we have
\[B(x_R^\prime,\ell(R)) \subseteq D(x_{Q(S)},P_S,C_2\ell(Q(S))). \]
Let $\Gamma_S$ be the Lipschitz graph from Lemma \ref{Graph2}. Since $\Gamma_S$ is Ahlfors regular and the balls  $\{C_1B_R : R \in \text{Stop}(N)\}$ have bounded overlap by Lemma \ref{l:Azzam-Stop}(2), we get
\begin{align*}
\ve \sum_{\substack{R \in \text{Stop}(N) \\ R \cap 2C_0B_{Q(S)} \not = \emptyset}}\ell(R)^d &\lesssim \ve \sum_{\substack{R \in \text{Stop}(N) \\ R \cap 2C_0B_{Q(S)} \not = \emptyset}}\mathscr{H}^d(B(x_R^\prime, \ell(R)) \cap \Gamma_S) \\
&\lesssim \ve \sum_{\substack{R \in \text{Stop}(N) \\ R \cap 2C_0B_{Q(S)} \not = \emptyset}}\mathscr{H}^d(D(x_{Q(S)},P_S,C_2\ell(Q(S))) \cap \Gamma_S) \\
&\lesssim \ve \ell(Q(S))^d. 
\end{align*}
This proves Claim 1. 
\end{proof}
We turn our attention to estimating $I_1.$
\bigbreak
\noindent
\textbf{Claim 2:}
\begin{align}
I_1 \lesssim \alpha^2 \ell(Q(S))^d.
\end{align} 
\begin{proof}[Proof of Claim 2]
Let $P_S$ be as above. Let $\mathscr{D}_{P_S}$ be the usual dyadic grid for $P_S$, and let $\pi_{S}$ denote the orthogonal projection onto $P_S.$ For each $Q \in S,$ let $I_Q$ be the minimal cube from $\mathscr{D}_{P_S}$ which contains $\pi_{S}(x_Q)$ and satisfies $\ell(I_Q) \geq 8C_0 \ell(Q).$ In particular, this implies 
\begin{align}\label{e:QI_Q}
\ell(Q) \sim \ell(I_Q).
\end{align} Recall that $x_Q'$ was defined to be the closest point to $x_Q$ in $\Sigma^N.$ Since the orthogonal projection onto a $d$-planes are 1-Lipschitz, we can use this, Lemma \ref{DistSigma}, and the fact that $Q \in \text{Up}(N)$ to get
\[ |\pi_{S}(x_Q) - \pi_{S}(x_Q^\prime)| \leq |x_Q - x_Q^\prime| \lesssim \ve d_{\text{Up}(N)}(x_Q) \lesssim \ve \ell(Q). \]
So, for $\ve$ small enough, we have 
\[\pi_{S}(2C_0B_Q^\prime) \subseteq 3I_Q \subseteq 3B_{I_Q}. \]
Using this with Lemma \ref{BetaBound} and the fact that for $C_2 \geq 4C_0$ we have $2C_0B_Q' \subseteq C_2B_{Q(S)}$ for all $Q \in S,$ we obtain
\begin{align}\label{FirstOmega}
\sum_{Q \in S} \beta_{\Sigma^N}^{d,p}(2C_0B_Q^\prime)^2\ell(Q)^d  &\lesssim \sum_{Q \in S}\Omega_{A_S,p}(3B_{I_Q})^2\ell(I_Q)^d \\
&\lesssim \sum_{I \in \mathscr{D}_{P_S}} \Omega_{A_S,p}(3B_I)^2\ell(I)^d \# \{Q \in S : I_Q = I\}. 
\end{align}
Next, we show
\begin{align}\label{e:I_Qbound}
	\# \{Q \in S : I_Q = I\} \lesssim_d 1. 
\end{align}
	Indeed, fix $I \in \mathscr{D}_{P_S}$ and suppose $Q,R \in S$ are such that $I_Q = I_R = I.$ By \eqref{e:QI_Q} we have $\ell(Q) \sim \ell(R).$ Using that $\pi_S$ is bi-Lipschitz on $D(x_{Q(S)},P_S,C_2\ell(Q))$ (see Lemma \ref{Graph2}), along with Lemma \ref{DistSigma}, and the fact that $\pi_S(x_Q),\pi_S(x_R) \in I$, we have
\begin{align}
|x_Q - x_R| &\leq |x_Q^\prime - x_R^\prime| + \text{dist}(x_Q,\Sigma^N) + \text{dist}(x_R,\Sigma^N) \\
 &\lesssim |\pi_S(x_Q') - \pi_S(x_R')| + \text{dist}(x_Q,\Sigma^N) + \text{dist}(x_R,\Sigma^N) \\
&\lesssim |\pi_S(x_Q) - \pi_S(x_R)| + 2\text{dist}(x_Q,\Sigma^N) + 2\text{dist}(x_R,\Sigma^N) \\
&\lesssim \ell(I) + \ve d_{\text{Up}(N)}(Q)+ \ve d_{\text{Up}(N)}(R) \\
&\lesssim \ell(I).
\end{align}
Thus, there exists some $p \in \Sigma$ and $C>0$ such that $x_Q \in B(p,C\ell(I))$ for all $Q \in S$ such that $I_Q = I.$ Since $\Sigma$ is $(\ve,d)$-Reifenberg flat, there is some $d$-plane $P$ through $p$ such that 
\[ \text{dist}(x_Q,P) \lesssim \ve \ell(I) \]
for all cubes in $ \{Q \in S : I_Q = I\}$. Since these cubes have comparable sizes and the balls $c_0B_Q$ are disjoint, another application of Lemma \ref{ENV} gives \eqref{e:I_Qbound}.

Now, continuing on from \eqref{FirstOmega}, we use \eqref{e:I_Qbound}, Corollary \ref{CorDor} and the fact that $\supp(A_S) \subseteq P_S \cap C_2B_{Q(S)}$, to get
\begin{align}
\sum_{Q \in S} \beta_{\Sigma^N}^{d,p}(2C_0B_Q^\prime)^2\ell(Q)^d  &\lesssim \sum_{I \in \mathscr{D}_{P_S}} \Omega_{A_S,p}(3B_I)^2\ell(I)^d \lesssim \norm{A_S}^2_{\text{Lip}} \ell(Q(S))^d  \lesssim \alpha^2 \ell(Q(S))^d,
\end{align}
which completes the proof of Claim 2.
\end{proof}
We combine Claim 1 and Claim 2 to get  
\begin{align*}
\sum_{Q\subseteq Q_0}\beta_{\Sigma}^{d,p}(C_0B_Q)^2\ell(Q)^d &\lesssim \sum_{N \geq 0}\sum_{\substack{S \in \mathscr{F}^N \\ Q(S) \subseteq Q_0}}\sum_{Q \in S} \beta_{\Sigma}^{d,p}(C_0B_Q)^2\ell(Q)^d \\
&\lesssim (\ve + \alpha^2) \sum_{N \geq 0} \sum_{\substack{S \in \mathscr{F}^N \\ Q(S) \subseteq Q_0}} \ell(Q(S))^d \\
&= (\ve + \alpha^2)\left( \ell(Q_0)^d + \sum_{N \geq 0} \sum_{Q \in \text{Stop}(N)} \ell(Q)^d \right) \\
& \lesssim \mathscr{H}^d(Q_0),
\end{align*}
where in the last inequality we used Lemma \ref{l:Azzam-Stop}(5) and the lower regularity of $\Sigma.$ This finishes the proof of \eqref{e:dTSP3'}, and hence, the proof of Theorem \ref{Thm3}.  

\end{proof}

\section{Estimating curvature for general sets}\label{s:Thm1}
In this section we prove the Theorem \ref{Thm1}. We start by making some reductions, which are proven in the appendix. The first allows us to assume $A > 6C_0.$  

\begin{lem}\label{l:reduce-constants1} 
	Suppose that for any $\tilde{C}_0 > 1$ and $\tilde{A} > \max\{6\tilde{C}_0,10^5\}$ there exists $\ve > 0$ so that \eqref{e:dTSP1} holds with constant $\tilde{C}_0,\tilde{A}$ and $\tilde{\ve}.$ Then, for any $C_0 > 1$ and $A > 10^5$ there exists $\ve > 0$ so that \eqref{e:dTSP1} holds with constant $C_0,A$ and $\ve.$
\end{lem}

The second observation allows us to prove Theorem \ref{Thm1} for subsets of Euclidean space, as long as the constants independent of the ambient dimension.

\begin{lem}\label{l:main-red2}
	Suppose for any $n \geq 2$, $1 \leq d < n,$ $1 \leq p < p(d),$ $C_0 > 1$, and $A > \max\{6C_0,10^5\}$, that there exists $\ve > 0$ so that \eqref{e:dTSP1} holds with constants $d,p,C_0,A,$ and $\ve$, for any lower content $d$-regular set $E \subseteq \R^n$, with constants independent of $n$. Then, there exists $C \geq 1$ so that \eqref{e:dTSP1} holds with constant $d,p,C_0,2A$, and $C\ve,$ for all lower content $d$-regular sets $E \subseteq H.$ 
\end{lem}

Let $n \geq 2,$ $1 \leq d < n, \ 1 \leq p < p(d), \ C_0 > 1$ and $A > \max\{6C_0, 10^5\}.$ Let $E \subseteq \R^n$ be lower content $d$-regular with regularity constant $c > 0$. Let $X_k$ be a sequence of maximally $\rho^k$-separated nets and let $\mathscr{D}$ be the Christ-David for $E$ from Theorem \ref{cubes}. Let $Q_0 \in \mathscr{D}$ and assume without loss of generality that $0 \in Q_0 \in \mathscr{D}_0$. By the previous two lemmas, it remains to find $\ve > 0$ so that 
\begin{align}\label{e:dTSP1'}
	1 + \sum_{\substack{Q \in \mathscr{D} \\ Q \subseteq Q_0}} \beta^{d,p}_E(C_0B_Q)\ell(Q)^d \lesssim \mathscr{H}^d(Q_0) + \text{BWGL}(Q_0,A ,\ve)
\end{align}
with constant independent of $n$ (and $E$). This is our goal for the remainder of the section. First, we may assume 
\[\mathscr{H}^d(Q_0) + \BWGL(Q_0,A,\ve) < \infty.\] 
Second, since $E$ is lower content $d$-regular, we have $1 \lesssim \mathscr{H}^d(Q_0),$ so we are left to bound the second term. The rough idea for this is as follows. We begin by running a stopping-time argument where we stop when the set $E$ fails to be flat at some particular scale. In each of these stopping-time regions $S,$ we shall produce a surface $\Sigma^{Q(S)}$ by Theorem \ref{DT} that is Reifenberg flat, and well-approximates $E$ inside $S$. By Lemma \ref{betaest}, we push the $\beta$-numbers estimate for $E$ onto a $\beta$-number estimate for $\Sigma^{Q(S)}$, which we control by Theorem \ref{Thm3}. The error term is controlled by $\text{BWGL}(Q_0,A,\ve).$

Let us begin. Let $\mathscr{G} = \mathscr{D} \setminus \BWGL(A,\ve),$ i.e. those cubes $Q \in \mathscr{D}$ for which $E$ is approximately flat at the scale of $Q$. In particular, if for a ball $B$, we define 
\begin{align}\label{e:vartheta}
	 {b\beta}_E^d(B) = \inf\{ d_B(E,P) : P \mbox{ is a $d$-plane } \}, 
\end{align}
then  
\[ \mathscr{G} = \{Q \in \mathscr{D} : {b\beta}_E^d(AB_Q) < \ve\}. \]

\begin{defn}
For $Q \in \mathscr{G}$,  we define a stopping-time region $S_Q$ as follows. First, add $Q$ to $S_Q.$ Then, inductively, add cubes $R$ to $S_Q$ if
\begin{enumerate}
\item$ R^{(1)} \in S_Q$, 
\item ${b\beta}_E^d(AB_{R^\prime}) < \ve$ for each sibling $R^\prime$ of $R$. 
\end{enumerate}
\end{defn}

We start by stating some results from \cite{azzam2018analyst}. The following is a combination of \cite[Lemma 10.1]{azzam2018analyst} and the discussion before.    

\begin{lem}\label{l:DTworks}
	Let $A > 10^5$ and $Q \in \mathscr{G}$. For $k \geq 0$, set $s(k)$ denote the integer such that 
	\[5\rho^{s(k)} \leq r_k \leq 5\rho^{s(k) -1},\]
	and
	\begin{align}
		\mathscr{C}_k^Q = \{x_R : R \in S_Q \cap \mathscr{D}_{s(k)}\}.
	\end{align}
	Let $\{x_{j,k}\}_{j \in J_k^Q}$ be a maximally $r_k$-separated net in $\mathscr{C}_k^Q$. For $j \in J_k^Q$, let $R_j \in S_Q \cap \mathscr{D}_{s(k)}$ be the cube so that $x_{j,k} = x_{R_j},$ then let $B_{j,k} = B(x_{R_j},r_k)$ and $P_{j,k} = P_{R_j}.$ Then, the triple $(P_Q,\{B_{j,k}\},\{P_{j,k}\}), \ k \geq 0, \ j \in J_k^Q,$ defines a CCBP.
\end{lem}

\begin{rem}\label{r:B_j,k}
	For $j \in J_k^Q$ let $R_j$ be as above. It is not difficult to see that $\rho A B_{j,k} \subseteq AB_{R_j},$ and so ${b\beta}_E^d(\rho AB_{j,k}) \lesssim \ve.$ 
\end{rem}
For $Q \in \mathscr{G},$ let $\sigma_k^Q,\Sigma_k^Q,\sigma^Q$ and $\Sigma^Q$ be the functions and surfaces obtained by Theorem \ref{DT}. Additionally, for $x \in H$ and $R \in \mathscr{D}$, let 
\[ d_Q(x) = d_{S_Q}(x) \quad \text{and} \quad d_Q(R) = d_{S_Q}(R). \]

\begin{lem}[{\cite[Lemma 10.2]{azzam2018analyst}}]\label{6:4}
	Let $Q \in \mathscr{G}.$ If $d_Q(x) = 0,$ then $x \in E \cap \Sigma^Q$. 
\end{lem}

\begin{lem}[{\cite[Lemma 10.3]{azzam2018analyst}}]\label{l:SigmaQ}
	Suppose $A \geq C_1$. For $x \in \Sigma^Q \cap C_1B_Q$, let $k = k_Q(x)$ be the maximal $k$ such that $x \in V_{k-1}^{11}.$ Then 
	\begin{align}\label{SigmaQ}
		\dist(x,E) \lesssim \ve r_k \sim \ve d_Q(x)
	\end{align}
	and there is a $C\ve$-Lipschitz graph $\Gamma_x^Q$ over a $d$-plane $P_{x}^Q$ such that 
	\begin{align}
		B(x,r_k) \cap \Sigma^Q = B(x,r_k) \cap \Sigma^Q_k = B(x,r_k) \cap \Gamma_x^Q. 
	\end{align}
\end{lem}
As in Section \ref{Sec:StopTime}, and unlike what is done in \cite{azzam2018analyst}, we will need to perform a smoothing procedure to extend each of these stopping-time regions. We choose a constant
\[ 0< \tau \leq \tau_0 =  \frac{1}{4(1+\rho^{-1})} \]
to be fixed small enough. Then, for each $Q \in \mathscr{G},$ define
\[ \text{Stop}(Q) = \{R \in \mathscr{D} : R \ \text{is maximal so that there exists $\tilde{R} \in \text{Child}(R^{(1)})$ with } \ \ell(\tilde{R}) < \tau d_Q(\tilde{R}) \}. \] 
We also define
\[ S_Q' = \{R \subseteq Q : R \not\subset T \mbox{ for any } T \in \text{Stop}(Q) \} \supseteq S_Q. \] 
Because $\Stop(Q)$ is defined in terms of siblings, each $S_Q'$ defines a stopping-time region in the sense of Definition \ref{StoppingTime}. Let us state and prove some facts about $\text{Stop}(Q)$ and $S_Q'.$

\begin{lem}\label{6:1.5}
	Let $Q \in \mathscr{G}$. Let $C_1 >1$, $\tau_1 = \min\{\tau_0, (8+4C_1)^{-1} \}$ and suppose $0 < \tau < \tau_1.$ If $Q,R \in \emph{Stop}(Q)$ are such that $C_1B_Q \cap C_1B_R \not=\emptyset$ then 
	\begin{align}\label{e:6sim}
		\ell(Q) \sim \ell(R). 
	\end{align}
	Furthermore, the collection of balls $\{C_1B_R\}_{R \in \Stop(Q)}$ has bounded overlap with constant independent of $n$.
\end{lem}

\begin{rem}
	The proof of Lemma \ref{6:1.5} is essentially the same as the proof of Lemma \ref{l:Azzam-Stop}(2). We omit the details.
\end{rem}

\begin{lem}\label{l:ell-sim-d}
	Let $Q \in \mathscr{G}$ and $R \in \Stop(Q).$ Then $\ell(R) \sim \tau d_Q(R)$.
\end{lem}

\begin{proof}
	Let $\tilde{R} \in S$ be the sibling of $R$ which satisfies $\ell(\tilde{R}) < \tau d_Q(\tilde{R})$. Then, by Lemma \ref{LipCubes}, the fact that $R,R' \subseteq B(x_{R^{(1)}}, \ell(R)/\rho),$ and our assumption on $\tau,$ we have  
	\begin{align}
		\ell(R) &= \ell(\tilde{R}) < \tau d_Q(\tilde{R}) \leq \tau \left(\ell(\tilde{R}) + \dist(\tilde{R},R) + \ell(R) + d_Q(R) \right) \\
		&\leq \tau \left( (2+2\rho^{-1})\ell(R) + d_Q(R) \right) \\
		&\leq \frac{1}{2}\ell(R) + \tau d_Q(R).
	\end{align}
	Rearranging implies $\ell(R) \leq 2\tau d_Q(R).$ Now for the reverse direction. By maximality, we know $\ell(R^{(1)}) \geq \tau d_Q(R^{(1)}).$ Hence, by Lemma \ref{LipCubes} and our choice of $\tau,$ 
	\begin{align}
		\frac{1}{\rho} \ell(R) &= \ell(R^{(1)}) \geq \tau d_S(R^{(1)}) \overset{\eqref{Tr}}{\geq} \tau  \left(d_S(R) - 2\ell(R) - 2\ell(R^{(1)}) \right) \\
		&= \tau d_S(R) - 2(1+\rho^{-1})\tau \ell(R) \geq  \tau d_S(R) - \frac{1}{\rho} \ell(R).
	\end{align}
	Rearranging implies $\ell(R) \geq \tfrac{\rho}{2}\tau d_Q(R).$
\end{proof}

\begin{lem}\label{l:findcube}
	Let $Q \in \mathscr{G}.$ For each $R \in S_Q'$ there exists $T \in S_Q$ such that 
	\begin{align}\label{RsimT}
		\tau \ell(T) \lesssim \ell(R) \leq \ell(T) \quad \text{and} \quad \tau \dist(R,T) \lesssim \ell(R) .
	\end{align}
\end{lem}

\begin{proof}
	If $R \in S_Q$ then the lemma is trivial, so suppose $R \in S_Q' \setminus S_Q.$ We first claim that we can find $\tilde{R} \in \Stop(Q)$ so that $\tilde{R} \subseteq R$ and $x_R \in \tilde{R}$. To see this, let $R' \subseteq R$ be so that $x_R \in R'$ and $\ell(R') < \tfrac{\tau c_0}{2} \ell(R),$ and let $R''$ be a sibling of $R'$ (possibly $R'$ itself). For any $y \in R''$ (assuming $\tau$ to be small enough), we have
	\begin{align}
		 |y - x_R| &\leq |y - x_{R''}| + |x_{R''} - x_{R'}| + |x_{R'} - x_R| \leq \ell(R'') + \rho^{-1}\ell(R') + \ell(R')  \\
		&= (2+\rho^{-1}) \ell(R') \leq \frac{c_0}{2} \tau (2+\rho^{-1}) \ell(R) \leq \frac{c_0}{2}\ell(R)
	\end{align}
	and so $R'' \subseteq \tfrac{c_0}{2}B_R.$ This implies that $\dist(R'',T) \geq \tfrac{c_0}{2}\ell(R)$ for all $T \in S_Q$ with $\ell(T) \leq \ell(R).$ Thus, 
	\[ \tau d_Q(R'') \geq  \frac{\tau c_0}{2}  \ell(R)  > \ell(R'').  \]
	Hence, by maximality, we can find $\tilde{R} \in \Stop(Q)$ so that $R' \subseteq \tilde{R} \subseteq R$ and $x_R \in \tilde{R}$ as required. 
	
	Let us continue with the estimates. Let $\tilde{R} \in \Stop(Q)$ be the cube from above and let $\tilde{T} \in S_Q$ be so that $d_Q(\tilde{R}) \sim  \ell(\tilde{T}) + \dist(\tilde{R},\tilde{T})$. By Lemma \ref{l:ell-sim-d},
	\[\ell(R) \geq \ell(\tilde{R}) \sim  \tau d_Q(\tilde{R}) \sim \tau \left( \ell(\tilde{T}) + \dist(\tilde{R},\tilde{T}) \right) \geq \tau \left( \ell(\tilde{T}) + \dist(R,\tilde{T}) \right), \] 
	so the first and third inequalities are satisfied by $\tilde{T}.$ If $\ell(R) \leq \ell(\tilde{T})$, we set $T = \tilde{T}.$ Otherwise, we set $T$ to be the ancestor of $\tilde{T}$ so that $\ell(T) = \ell(R).$ 
\end{proof}

\begin{lem}\label{6:3}
	Let $Q \in \mathscr{G}.$ If $x \in E$ is such that $x$ is not contained in any cube $R$ from $\emph{Stop}(Q)$, then $d_Q(x) = 0.$
\end{lem}

\begin{proof}
	Let $R_j \in \mathscr{D}$ be a sequence of cubes such that $R_j \ni x$ and $\ell(R_j) \downarrow 0.$ Since $x$ is not contained in any cube from $\text{Stop}(Q)$, each $R_j$ is not contained in any cube from $\text{Stop}(Q).$ It follows that $\ell(R_j) \geq \tau d_Q(R_j)$ otherwise $R_j \in \Stop(Q).$ For each $j$, let $x_j \in R_j$ be so that $d_Q(x_j) \leq d_Q(R_j) + j^{-1}.$ Then  
	\begin{align}
		d_Q(x) \leq |x - x_{j}| + d_Q(x_j) \leq |x - x_{j}| + d_Q(R_j) + j^{-1}  \leq (2 + \tau^{-1})\ell(R_j) + j^{-1}  \rightarrow 0,
	\end{align}
	and so $d_Q(x) = 0.$ 
\end{proof}

\begin{lem}\label{6:1}
Let $Q \in \mathscr{G}$ and $R \in S_Q'.$ Then $\dist(x_R, \Sigma^Q) \lesssim \frac{\ve}{\tau}\ell(R).$
\end{lem}

\begin{proof}
By Lemma \ref{l:findcube}, we can find $T \in S_Q$ such that 
\[  \tau \ell(T) \lesssim \ell(R) \leq \ell(T) \quad \text{and}\quad \tau\text{dist}(R,T) \lesssim \ell(T).\] Recall the definitions of $s(k)$, $\mathscr{C}_k^Q$, $J_{k}^Q$, $B_{j,k}$ and $P_{j,k},$ from Lemma \ref{l:DTworks}. Let $k$ be such that $T \in \mathscr{D}_{s(k)}.$ This means
\[ \ell(T) \sim r_k. \] 
Since $\{x_{j,k}\}_{j \in J_k^Q}$ is a maximal net in $\mathscr{C}_k^Q$, there exists $j \in J_k^Q$ such that $x_T \in B_{j,k}.$ It follows that  
\begin{align}
|x_R - x_{j,k}| &\leq |x_R - x_T| + |x_T - x_{j,k}| \leq \ell(R) + \dist(R,T) + \ell(T) + |x_T - x_{j,k}| \\
&\lesssim  (2+C\tau^{-1})\ell(T) + r_k \lesssim r_k,
\end{align}
so taking $A$ large enough implies $x_R \in \rho A B_{j,k}$. Using this with the fact that ${b\beta}_E^d(\rho A B_{j,k}) \lesssim \ve$ (see Remark \ref{r:B_j,k}) gives
\begin{align}
|x_R - \pi_{j,k}(x_R)| = \text{dist}(x_R, P_{j,k}) \lesssim \ve r_k . 
\end{align}
By Lemma \ref{DT} (11), there exists $y \in \Sigma_k^Q$ such that $|\pi_{j,k}(x_R) - y| \lesssim \ve r_k$. Moreover, this satisfies $\text{dist}(y,\Sigma^Q) \lesssim \ve r_k$, by Lemma \ref{DT} (10). Putting everything together, we have
\begin{align}
\text{dist}(x_R, \Sigma^Q) \leq |x_R - \pi_{j,k}(x_R)| + |\pi_{j,k}(x_R) - y| + \text{dist}(y,\Sigma^Q)  \lesssim \ve r_k \lesssim \frac{\ve}{\tau}\ell(R).
\end{align}
\end{proof}

The following is similar to \cite[Lemma 10.4]{azzam2018analyst}. In \cite{azzam2018analyst}, the sum on the right-hand side of \eqref{e:difficult} is taken over the cubes in $\min(S_Q)$, but their proof contains a mistake. We were not able to prove the estimate stated in \cite{azzam2018analyst}, but the following is sufficient. Additionally, we now have constants independent of $n$.

\begin{lem}\label{l:difficult}
For $Q \in \mathscr{G},$
\begin{align}\label{e:difficult}
\mathscr{H}^d(\Sigma^Q \cap C_1B_Q) \lesssim \sum_{\substack{R \in \emph{Stop}(Q)\\ R \subseteq Q}} \ell(R)^d + \mathscr{H}^d(\{ x \in \overline{Q} : d_Q(x) = 0\}).
\end{align}
\end{lem}

\begin{proof}
Let $k_Q(x)$ be as in the statement of Lemma \ref{l:SigmaQ}. Let $\delta >0$ and set
\begin{align*}
E_1 = \{x \in \Sigma^Q \cap C_1B_Q : d_Q(x) >0 \ \text{and} \ \text{dist}(x,Q) < \delta r_{k_Q(x)} \}. 
\end{align*}
Our first goal is to bound $\mathscr{H}^d(E_1).$ Let $\{x_j\}_{j \in J}$ be a maximal net in $E_1$ such that 
\[|x_i - x_j| \geq \max\{r_{k_Q(x_i)},r_{k_Q(x_j)}\}\]
and let $B_j = B(x_j,2 r_{k_Q(x_j)})$ for $j \in J.$ In this way, the collection of balls $\tfrac{1}{4}B_j, \ j \in J,$ are disjoint and 
\begin{align*}
E_1 \subseteq \bigcup_{j \in J} B_j.
\end{align*}
Let $j \in J,$ $k_j = k_Q(x_j)$, and let $z_j$ be the point in $Q$ closest to $x_j.$ We claim that there exists $Q_j \in \text{Stop}(Q)$ so that $Q_j \subseteq Q,$ $z_j \in Q_j$, and 
\begin{align}\label{CompR}
	\ell(Q_j) \sim r_{k_j}.
\end{align} 
Let us see why this is true. By definition, we have
\begin{align}\label{e:deltak_j}
 |x_j -z_j| = \dist(x_j,Q)  < \delta r_{k_j}. 
\end{align}
Then, by \eqref{SigmaQ} and the fact that $d_Q$ is 1-Lipschitz, we have
\[ r_{k_j} \sim d_Q(x_j) \leq d_Q(z_j) +  |x_j - z_j|  < d_Q(z_j) + \delta r_{k_j}, \]
so for $\delta>0$ small enough 
\begin{align}\label{e:d_Q(z_J)}
	d_Q(z_j) \gtrsim r_{k_j}.
\end{align}
Now, Lemma \ref{6:3} implies the existence of a cube $Q_j \in \Stop(Q)$ so that $z_j \in Q_j$. Notice that $Q_j \subseteq Q$ since $z_j \in Q \cap Q_j$. We are left to show \eqref{CompR}. Let $T_j \in S_Q$ be the cube from Lemma \ref{l:findcube} for $Q_j$. By \eqref{e:d_Q(z_J)}, and the fact that $z_j \in Q_j,$ we have 
\begin{align}
r_{k_j} \lesssim d_Q(z_j) \leq \ell(T_j) + \dist(z_j,T_j) \leq \ell(T_j) + 2\ell(Q_j) + \dist(Q_j,T_j) \lesssim  \ell(Q_j).	
\end{align}
For the reverse direction, using Lemma \ref{l:ell-sim-d}, \eqref{CompR}, \eqref{e:deltak_j}, and the fact that $z_j \in Q_j$, we get
\begin{align}
	\ell(Q_j) \lesssim d_Q(Q_j) \leq d_Q(z_j) \leq |x_j - z_j| + d_Q(x_j) \lesssim r_{k_j}.
\end{align}
 
Now that we have defined $Q_j$, we control how often $Q_j = R$ for some fixed cube $R$. Let $R \in \text{Stop}(Q)$ and let
\begin{align*}
\mathscr{C}_R = \{j \in J : Q_j = R\}.
\end{align*}
By \eqref{e:deltak_j} and \eqref{CompR}, for $i,j \in \mathscr{C}_R$ and $\delta >0$ small enough, we have
\[|x_j - x_i| \leq |x_j - z_j| + |z_j - z_i| + |z_i-x_i| \leq \delta r_{k_j} + 2\ell(R) + \delta r_{k_i} \leq 3 \ell(R). \]
Pick some $j^* \in \mathscr{C}_R,$ and set $p = x_{j^*}$. Then $p \in \Sigma^Q$ and the above inequality implies $x_j \in B(p,3\ell(R))$ for all $j \in \mathscr{C}_R$. Using this with \eqref{CompR} and the fact that $\Sigma^Q$ is $C\ve$-Reifenberg flat, there exists a $d$-plane $P$ so that
\[\text{dist}(x_j,P) \leq 3 \ve \ell(R) \leq \tfrac{1}{2}r_{B_j/4}.\]
Hence, we can apply Lemma \ref{ENV} to the disjoint balls $\tfrac{1}{4}B_j, \ j \in J$, and get
\begin{align}\label{e:ENVhard}
\sum_{j \in \mathscr{C}_R} r_{k_j}^d \sim \sum_{j \in \mathscr{C}_R} r_{B_j/4}^d \lesssim \ell(R)^d.
\end{align}
Combining this, with the local Lipschitz description of $\Sigma^Q$ inside $B_j$ from Lemma \ref{l:SigmaQ}, we have
\begin{align}
\mathscr{H}^d(E_1) &\leq \sum_{j \in J} \mathscr{H}^d(E_1 \cap B_j) \leq \sum_{j \in J} \mathscr{H}^d(\Sigma^Q \cap B_j) \lesssim \sum_{j \in J} r_{k_j}^d \\
& \leq \sum_{\substack{R \in \text{Stop}(Q)\\ R \subseteq Q}} \sum_{j \in \mathscr{C}_R} r_{k_j}^d \lesssim \sum_{\substack{R \in \text{Stop}(Q)\\ R \subseteq Q}} \ell(R)^d. 
\end{align}
Now, set 
\begin{align*}
E_2 = \{x \in \Sigma^Q \cap C_1B_Q : d_Q(x) >0 \ \text{and} \ \text{dist}(x,Q) \geq \delta r_{k_Q(x)} \}.
\end{align*}
We will control $\mathscr{H}^d(E_2).$ Let $\{x_i \}_{i \in I}$ be a maximal net in $E_2$ such that 
\begin{align}
|x_i - x_j| \geq \max \{ r_{k_Q(x_i)}, r_{k_Q(x_j)} \}
\end{align}
and define $B_i = B(x_i , r_{k_Q(x_i)}).$ Let $Q_j \in S_Q$ be a cube such that 
\begin{align}\label{e:Q_jd_Q}
\ell(Q_j) + \text{dist}(x_j,Q_j) < 2d_Q(x_j),
\end{align}
and let $R_j \in S_Q$ be the maximal ancestor of $Q_j$ for which 
\[\ell(R_j) < 2d_Q(x_j).\]
By \eqref{e:Q_jd_Q}, we have $\dist(x_j,Q_j) < 2d_Q(x_j),$ which, since $Q_j \subseteq R_j,$ implies
\begin{align}\label{e:maxan1}
\text{dist}(x_j,R_j) < 2d_Q(x_j).
\end{align}
Additionally,
\begin{align}\label{e:maxan}
\ell(R_j) \sim d_Q(x_j) \sim r_{k_j}.
\end{align}
Indeed, the second is true by \eqref{SigmaQ}. For the first equivalence, the forward inequality is clear by definition. If $R_j \not= Q$ then the backwards inequality follows from maximality, if $R_j = Q$ then the backwards inequality follows from the fact that $x_j \in C_1B_Q$, since
\[ d_Q(x_j) \leq \ell(Q) + \dist(x_j,Q) \lesssim \ell(Q).\]
Now, for each $i \in I$, let 
\[B_i^\prime = \tfrac{c_0}{8}B_{R_i}.\]
We will show that the balls $\{B_i^\prime\}_{i \in I}$ have bounded overlap in $\Sigma^Q \cap C_1B_Q.$ Indeed, let $x \in \Sigma^Q \cap C_1B_Q$ and define 
\begin{align}
I(x) = \{i \in I : x \in B_i^\prime \}.
\end{align}
Suppose $i,j \in I(x).$ We begin by showing $r_{B_j^\prime} \sim r_{B_i^\prime}$. Assume without loss of generality that $\ell(R_i) \leq \ell(R_j).$ We are left to show $\ell(R_j) \lesssim \ell(R_i).$ Let $x_i^\prime$ be the point in $E$ closest to $x_i.$ Recalling that $x_i \in Q \subseteq \Sigma_Q \cap C_1B_Q$, we have
\begin{align}\label{e:distE}
	|x_i - x_i^\prime| = \text{dist}(x_i,E)  \stackrel{\eqref{SigmaQ}}{\lesssim} \ve r_{k_i} \stackrel{\eqref{e:maxan}}{\sim} \ve \ell(R_i).
\end{align}
So, for $\ve >0$ small enough, 
\begin{align}\label{distQ}
	\text{dist}(x_i^\prime, Q) \geq \text{dist}(x_i,Q) - |x_i - x_i^\prime| \geq \delta r_{k_i} - C\ve r_{k_i} \gtrsim \delta r_{k_i} >0.
\end{align}
Let $x^\prime$ be the point in $E$ closest to $x$. Since $x \in \Sigma^Q \cap B_i' \subseteq \Sigma^Q \cap C_1B_Q,$ 
\begin{align} \label{distE}
|x'- x| = \text{dist}(x,E) \stackrel{\eqref{SigmaQ}}{\lesssim} \ve d_Q(x) \leq  \ve ( | x - x_i| + d_Q(x_i) ) \stackrel{\eqref{e:maxan}}{\lesssim} \ve \ell(R_i).
\end{align}
The above estimate also holds true for the index $j$. Taking $\ve$ small enough, and using \eqref{distE} for $j,$ we have 
\[ |x' - x_{R_j}| \leq |x' - x| + |x - x_{R_j}| \leq C\ve \ell(R_j) + \frac{c_0}{8}\ell(R_j) \leq \frac{c_0}{2}\ell(R_j) \] 
which implies 
\[x^\prime \in E \cap \frac{c_0}{2}B_{R_j} \subseteq Q.\]
Using this, the fact that $x_i' \in E \setminus Q$ (by \eqref{distQ}), \eqref{e:distE} and \eqref{distE}, we get
\begin{align}
\frac{c_0}{2}\ell(R_j) & \leq \text{dist}(x_i^\prime, E \cap \tfrac{c_0}{2}B_{R_j}) \leq |x_i^\prime - x'|  \leq |x_i^\prime - x_i| + |x_i - x| + |x - x'| \\
& \lesssim \ell(R_i).
\end{align}
This proves $r_{B_i'} \sim r_{B_j'}.$ 

We will now apply Lemma \ref{ENV}. Let $R = \sup_{i \in I(x)} r_{k_i}.$ Since the balls $\tfrac{1}{4}B_i^\prime$, $i \in I(x),$ have comparable radii, we can find $i^* \in I(x)$ and a constant $C >0$ such that $\tfrac{1}{4}B_i^\prime \subseteq B(x_{i^*},CR)$ for all $i \in I(x)$. Furthermore, since $\Sigma^Q$ is Reifenberg flat, there exists a plane $P$ through $x_{i^*}$ such that $\text{dist}(x_{B_i^\prime},P) \lesssim r_{B_i^\prime}/2$ for all $i \in I(x).$ Finally, since the balls $\tfrac{1}{4}B_i'$ are disjoint, we can apply Lemma \ref{ENV} to get 
\begin{align}\label{eee}
	I(x) \lesssim_d 1 
\end{align}
as required. 

Combining \eqref{e:maxan} with \eqref{eee}, and the fact that $\Sigma^Q$ is lower regular (since it is Reifenberg flat, see for example the proof of Lemma \ref{al:lowerreg}), we obtain
\begin{align}
\sum_{i \in I} r_{k_i}^d \lesssim \sum_{i \in I} \ell(R_i)^d \lesssim \mathscr{H}^d( \Sigma^Q \cap B_i^\prime) \lesssim \mathscr{H}^d\left(\Sigma^Q \cap \bigcup_{i \in I} B_i^\prime \right).
\end{align}
Suppose $x \in \Sigma^Q \cap B_i'$ for some $i \in I.$ Recall by \eqref{distE} that
\begin{align}
\dist(x,Q) = \text{dist}(x,E) \lesssim \ve d_Q(x) \sim \ve r_{k_Q(x)}.
\end{align}
This and the fact that $E \cap B_i' \subseteq R_i \subseteq Q$ imply, for $\ve$ small enough, that $x$ belongs to either $E_1$ or $E_0$, where
\[ E_0 = \{x \in E \cap C_1B_Q : d_Q(x) = 0 \} \subseteq \{x \in \overline{Q} : d_Q(x) = 0 \}. \]
We conclude the proof by noticing
\begin{align}
\mathscr{H}^d(E_2) &\leq \sum_{i \in I} \mathscr{H}^d( E \cap B_i) \leq \sum_{i \in I} \mathscr{H}^d(\Sigma^Q \cap B_i) \lesssim \sum_{i \in I} r_{k_i}^d \\
&\lesssim  \mathscr{H}^d\left(\Sigma^Q \cap \bigcup_{i \in I} B_i^\prime \right) \leq \mathscr{H}^d(E_0 \cup E_1) \\
&\lesssim \sum_{\substack{R \in \text{Stop}(Q)\\ R \subseteq Q}} \ell(R)^d + \mathscr{H}^d(E_0).
\end{align}

\end{proof}

Let ${ \Top}_0$ be the maximal collection of cubes from $\mathscr{G}$ contained in $Q_0.$ Then, supposing ${\Top}_k$ has been defined for some $k \geq 0,$ define ${\min}_k$ to be the cubes $R$ so that $R \in \min(S_Q)$ for some $Q \in {\Top}_k.$ We let ${\Stop}_k$ be the cubes $R$ so that $R \subseteq Q$ and $R \in \Stop(Q)$ for some $Q \in {\Top}_k$, and let ${\Top}_{k+1}$ be the maximal collection of cubes from $\mathscr{G}$ contained in those cubes from ${\Stop}_k.$ Finally, we define
\[ \Top = \bigcup_{k \geq 0} {\Top}_k, \quad \min = \bigcup_{k \geq 0} {\min}_k \quad \text{and} \quad \Stop = \bigcup_{k \geq 0} {\Stop}_k . \]

\begin{lem}\label{l:goodlem}
	We have 
	\[ \sum_{Q \in \Stop} \ell(Q)^d \lesssim \mathscr{H}^d(Q_0) + \BWGL(Q_0,A,\ve) \] 
	and 
	\[ \sum_{Q \in \Top} \mathscr{H}^d(\Sigma^Q \cap C_1B_Q) \lesssim \mathscr{H}^d(Q_0) + \BWGL(Q_0,A,\ve). \] 
\end{lem}

Before proving Lemma \ref{l:goodlem}, we need some preliminary results.

\begin{lem}\label{l:coverKae}
	Let $K$ be the set of points in $Q_0$ which are contained in infinitely many cubes $Q \subseteq Q_0$ so that $Q \in \BWGL(A,\ve).$ Then $\mathscr{H}^d(K) = 0.$  
\end{lem}

\begin{proof}
	Let $\mathscr{C}_1$ be the maximal collection of cubes from $\BWGL(A,\ve)$ contained in $Q_0.$ Then, assuming $\mathscr{C}_k$ has been defined for some $k \geq 1,$ let $\mathscr{C}_{k+1}$ be the maximal collection of cubes in $\BWGL(A,\ve)$ which are strictly contained in the cubes from $\mathscr{C}_k.$ By maximality each $\mathscr{C}_k$ forms a cover for $K$. By construction, any $Q \in \BWGL(A,\ve)$ is contained in $\mathscr{C}_k$ for at most one $k$, and if $Q \in \mathscr{C}_k$, then $\ell(Q) \leq 5\rho^{k}.$ Thus, since $\BWGL(Q_0,A,\ve) < \infty$ (recall this from the beginning of the section), it follows that 
	\[ \lim_{k \rightarrow \infty} \sum_{Q \in \mathscr{C}_k} \ell(Q)^d = 0. \] 
	Hence, 
	\[ \mathscr{H}^d(K) = \lim_{k \rightarrow \infty} \mathscr{H}^d_{5\rho^k}(K) \leq \lim_{k \rightarrow \infty} \sum_{Q \in \mathscr{C}_k} \ell(Q)^d = 0. \] 
\end{proof}

\begin{lem}\label{l:Stop_k} Let $k\geq 0.$ Then 
	\[ \sum_{Q \in {\Stop}_k} \ell(Q)^d \lesssim \sum_{Q \in {\min}_{k+1}} \ell(Q)^d + \sum_{Q \in{\Top}_{k+1}} \mathscr{H}^d( \{x \in Q : d_Q(x) = 0\}). \]
\end{lem}

\begin{proof}
Let $Q \in {\Stop}_k.$ By Lemma \ref{l:coverKae}, the cubes $ \{R \in \Top_{k+1} : R \subseteq Q \}$ cover $Q$ up to a set of $\mathscr{H}^d$-measure zero. Using this with the fact that $E$ is lower content $d$-regular, we have 
	\[ \ell(Q)^d \lesssim \mathscr{H}^d_\infty(Q) \lesssim \sum_{\substack{R \in {\Top}_{k+1} \\ R \subseteq Q}} \ell(R)^d. \] 
	Let $R \in {\Top}_{k+1}$ be so that $R \subseteq Q$ and let $x \in E \cap c_0B_R.$ One of two things can happen, either $x \in T$ for some $T \in \min(S_R)$ (which by definition means that $T \in \min_{k+1}$), or $d_R(x) = 0.$ Again, since $E$ is lower content $d$-regular, this implies 
	\begin{align}
		\ell(R)^d \lesssim \mathscr{H}^d_\infty(E \cap c_0B_R) &\lesssim \sum_{\substack{T \in {\min}_{k+1} \\ T \subseteq R}} \ell(T)^d + \mathscr{H}_\infty^d(\{x \in R : d_R(x) = 0\})  \\
		& \leq \sum_{\substack{T \in {\min}_{k+1} \\ T \subseteq R}} \ell(T)^d + \mathscr{H}^d(\{x \in R : d_R(x) = 0\}).
	\end{align} 
	Using the above and taking the sum over all $R \in {\Top}_{k+1}$ so that $R \subseteq Q$, we have 
	\[ \ell(Q)^d \lesssim \sum_{\substack{ R \in {\min}_{k+1} \\ R \subseteq Q}} \ell(R)^d + \sum_{\substack{R \in {\Top}_{k+1} \\ R \subseteq Q}} \mathscr{H}^d(\{x \in R : d_R(x) = 0\}). \] 
	Since the cubes in ${\Stop}_k$ are disjoint, taking the sum over ${\Stop}_k$ finishes the proof. 
\end{proof}

\begin{lem}\label{l:d(Q)}
	For $Q \in \Top$ define 
	\[d(Q) = \{x \in \overline{Q} : d_Q(x) = 0 \}.\]
	If $Q,Q' \in \Top$ are such that $Q \not=Q',$ then $\mathscr{H}^d( d(Q) \cap d(Q')) = 0.$
\end{lem}

\begin{proof}
	If $Q,Q' \in {\Top}_k$ for some $k$ then by construction $Q \cap Q' = \emptyset$ and so $d(Q) \cap d(Q') \subseteq \partial Q,$ where  $\partial Q = \{x \in \overline{Q} : B(x,r) \cap E\setminus Q \not= \emptyset \ \text{for all} \ r>0 \}$. By \cite[Lemma 10.5]{azzam2018analyst} we have $\mathscr{H}^d(\partial Q) = 0$ which finishes the proof in this case. Suppose then that $Q \in {\Top}_k$ and $Q' \in {\Top}_m$, and assume without loss of generality that $k > m.$ It follows that $Q \subseteq R$ for some $R \in \min(S_{Q'}).$ Hence, 
	\[ d(Q) \cap Q^o \subseteq R^o \subseteq Q' \setminus d(Q'). \] 
	This implies $d(Q) \cap d(Q') \subseteq \partial Q$ and the lemma follows again by \cite[Lemma 10.5]{azzam2018analyst}. 
\end{proof}

\begin{proof}[Proof of Lemma \ref{l:goodlem}]
	The second inequality is an easy consequence of the first inequality and Lemma \ref{l:difficult}. Let us prove the first inequality. Since each $Q \in \min$ is the minimal cube of some stopping-time region, it has a child, which we denote by $Q'$, so that $Q' \in \BWGL(A,\ve).$ Using Lemma \ref{l:Stop_k}, Lemma \ref{l:d(Q)}, and taking the sum over all $k \geq 0,$ we get 
	\begin{align} \sum_{Q \in \Stop} \ell(Q)^d &\lesssim \sum_{Q \in \min} \ell(Q)^d + \sum_{Q \in \Top} \mathscr{H}^d( d(Q)) \lesssim \sum_{Q \in \min} \ell(Q')^d + \sum_{Q \in \Top} \mathscr{H}^d( d(Q)) \\ 
	&\lesssim \BWGL(Q_0,A,\ve) + \mathscr{H}^d(Q_0) ,
	\end{align}
as required. 
\end{proof}

Let 
\[ \mathscr{F} = \{S_Q' : Q \in \Top\}. \] 
For $S \in \mathscr{F},$ set 
\[\Sigma^S = \Sigma^{Q(S)}, \ \Stop(S)= \Stop(Q(S)) \ \text{and} \ d_S = d_{Q(S)}. \] 
The following finishes the proof of \eqref{e:dTSP1'}.

\begin{prop}\label{ThmPart1}
We have 
\[ \sum_{Q \in \mathscr{D}} \beta_E^{d,p}(C_0B_Q) \ell(Q)^d \lesssim \mathscr{H}^d(Q_0) + \BWGL(Q_0,A,\ve) \]	
\end{prop}

\begin{proof}
By construction, for each $Q \in \mathscr{G},$ there exists $S \in \mathscr{F}$ so that $Q \in S.$ Hence
\[ \sum_{Q \in \mathscr{D}} \beta_E^{d,p}(C_0B_Q) \ell(Q)^d \lesssim \sum_{S \in \mathscr{F}} \sum_{Q \in S} \beta_E^{d,p}(C_0B_Q) \ell(Q)^d + \BWGL(Q_0,A,\ve). \]   
It remains to bound the first term. First, by Lemma \ref{betaest}, we have
\begin{align}
\sum_{S \in \mathscr{F}} \sum_{Q \in S} \beta_E^{d,p}(C_0B_Q) \ell(Q)^d &\lesssim \sum_{S \in \mathscr{F}} \sum_{Q \in S} \beta_{\Sigma^S}^{d,p}(2C_0B_Q)\ell(Q)^d \\
& \hspace{-5em} +  \sum_{S \in \mathscr{F}} \sum_{Q \in S}\left( \frac{1}{\ell(Q)^d} \int_{E\cap 2C_0B_Q} \left(\frac{\text{dist}(x,\Sigma^S)}{\ell(Q)}\right)^p \, d\mathscr{H}^d_\infty(x)\right)^\frac{2}{p}\ell(Q)^d \\
&= I_1 + I_2. 
\end{align}
The proof now split into two claims. 
\bigbreak
\noindent
\textbf{Claim 1:} 
\begin{align}
I_2 \lesssim \mathscr{H}^d(Q_0) + \BWGL(Q_0,A,\ve). 
\end{align}
\begin{proof}[Proof of Claim 1]
Let $S \in \mathscr{F}.$ By Lemma \ref{JensenPhi}, we may assume $p > 2$ (as in the proof of Theorem \ref{Thm3}), and we take $\tau$ small enough so that Lemma \ref{6:1} and Lemma \ref{6:1.5} hold. By Lemma \ref{6:3}, if $x  \in E \cap 2C_0B_Q$ is not contained in any cubes from $\Stop(S)$ then $d_S(x) = 0.$ This implies $\text{dist}(x,\Sigma^S) = 0,$ by Lemma \ref{6:4}. Suppose instead that $x \in R$ for some $R \in \Stop(S).$ Then, by Lemma \ref{6:1}, 
\begin{align}
\text{dist}(x,\Sigma^S) \leq |x - x_R| + \text{dist}(x_R , \Sigma^S) \lesssim 2\ell(R). 
\end{align} 
Using Lemma \ref{l:subsum} and using our assumption that $p\geq2$ we may write
\begin{align}
I_2 \lesssim \sum_{S \in \mathscr{F}} \sum_{Q \in S} \left( \sum_{\substack{R \in \Stop(S) \\ R \cap 2C_0B_Q \not= \emptyset}} \frac{\ell(R)^{d+p}}{\ell(Q)^{d+p}}\right)^\frac{2}{p} \ell(Q)^d  \lesssim \sum_{S \in \mathscr{F}} \sum_{Q \in S} \sum_{\substack{R \in \Stop(S) \\ R \cap 2C_0B_Q \not= \emptyset}}\frac{\ell(R)^{d\frac{2}{p} + 2 }}{\ell(Q)^{d(\frac{2}{p} - 1) + 2}}.
\end{align}
For fixed $R \in \Stop(S)$, we will show
\begin{align}\label{e:Stop(S)<1}
	\{Q \in \mathscr{D}_k \cap S : R \cap 2C_0B_Q \not= \emptyset\} \lesssim 1.
	\end{align}
First, by Lemma \ref{6:1} (2), if we choose $\ve$ small enough there exists a point $x_R' \in \Sigma^S$ such that
\[ |x_R - x_R'| \leq \ell(R). \] 
Suppose $Q \in \mathscr{D}_k \cap S$ is such that $R \cap 2C_0B_Q \not= \emptyset.$ Since 
\[ \ell(R) < \tau d_S(R) \leq \tau(\ell(Q) + \text{dist}(Q,R)) \lesssim \tau\ell(Q), \]
we get
\[ |x_Q - x_R'| \leq |x_Q - x_R| + |x_R - x_R'| \leq \ell(R) + 2C_0\ell(Q) +\ell(R) \leq C\ell(Q). \]
In particular, this implies that $x_Q \in B(x_R',C\ell(Q)).$ Since $\Sigma_S$ is $(\ve,d)$-Reifenberg flat, we can find a plane $P$ through $x_R'$ such that 
\[d_{x_R',C\ell(Q)}(P,\Sigma_S) \lesssim \ve.\]
Let $x_Q^\prime$ be the point in $\Sigma_S$ which is closest to $x_Q$. Then
\begin{align*}
\text{dist}(x_Q , P) \lesssim |x_Q - x_Q^\prime| + \text{dist}(x_Q^\prime,P) \lesssim \ve \ell(Q).
\end{align*}
So, for $\ve$ small enough, we have $\text{dist}(x_Q,P) \leq c_0\ell(Q)/2.$ Then, \eqref{e:Stop(S)<1} follows from Lemma \ref{ENV} by the standard argument. 

We return to estimating $I_2.$ By assumption $p \leq 2d/(d-2),$ which is equivalent to $d(2/p -1) +2 >0.$ Combining this, with the fact that \eqref{e:Stop(S)<1}, we can swap the order of integration and sum over a geometric series to get 
\begin{align}
I_2 &= \sum_{S \in \mathscr{F}}\sum_{\substack{R \in \Stop(S) \\ R \cap 2C_0B_{Q(S)}\not=\emptyset}}\sum_{\substack{Q \in S \\ R \cap 2C_0B_Q \not=\emptyset}} \frac{\ell(R)^{d\frac{2}{p} + 2 }}{\ell(Q)^{d(\frac{2}{p} - 1) + 2}} \lesssim \sum_{S \in \mathscr{F}}\sum_{\substack{R \in \Stop(S) \\ R \cap 2C_0B_{Q(S)}\not=\emptyset}} \ell(R)^d \\
&\lesssim \sum_{S \in \mathscr{F}} \sum_{\substack{R \in \Stop(S) \\ R \cap 2C_0B_{Q(S)}\not=\emptyset}} \mathscr{H}^d(\Sigma^S \cap c_0B_R) \lesssim \sum_{S \in \mathscr{F}} \mathscr{H}^d(\Sigma^S \cap C_1B_{Q(S)}) \\
&\lesssim \mathscr{H}^d(Q_0) + \BWGL(Q_0,A,\ve).
\end{align}
This completes the proof of Claim 1.
\end{proof}
\noindent
\textbf{Claim 2:}
\[ I_1 \lesssim \mathscr{H}^d(Q_0) + \BWGL(Q_0,A,\ve).  \]
\begin{proof}[Proof of Claim 2]
Fix $S \in \mathscr{F}$ and let $k \in \Z$ be such that $Q(S) \in \mathscr{D}_k.$ Recall, this means that 
\[ \ell(Q(S)) = 5\rho^k. \]  
For $n \geq k$ let $\mathscr{X}_n$ be a maximal $\rho^n$-separated net for $\Sigma^S$, and let $\mathscr{D}^S$ be the Christ-David cubes from Theorem \ref{cubes} for $\Sigma^S$ with respect to the nets $\mathscr{X}_n.$ For $Q \in S$, let $x_Q'$ be the point in $\Sigma^S$ closest to $x_Q$ and let $\tilde{Q} \in \mathscr{D}^S$ be the cube so that $x_Q' \in \tilde{Q}$ and $\ell(Q) = \ell(\tilde{Q}).$ By Lemma \ref{6:1}, for any $y \in 2C_0B_Q,$ we have
\begin{align}
|y - x_{\tilde{Q}}| &\leq |y-x_Q| + |x_Q - x_Q'| + |x_Q' - x_{\tilde{Q}}|  \\
& \leq 2C_0\ell(Q) + C\frac{\ve}{\tau}\ell(Q) + \ell(\tilde{Q}) \\
&\leq 3C_0\ell(\tilde{Q})
\end{align} 
for $\ve$ small enough with respect to $\tau.$ Hence, $2C_0B_Q \subseteq 3C_0B_{\tilde{Q}}.$  Since points in $\{x_Q : Q \in S \cap \mathscr{D}_k\}$ are maximally separated and lie close to $\Sigma_S$ (which is Reifenberg flat), Lemma \ref{ENV} implies the map $Q \rightarrow \tilde{Q}$ is at most $C$-to-$1$. Let $\{R_i\}_{i \in I}$ be the cubes $R \in \mathscr{D}^S$ which contains some $\tilde{Q}$ obtained above and so that $\ell(Q) = \ell(Q(S))$. The centres of these cubes are all maximally separated and contained in the ball $C_1B_{Q(S)}$, so $\#I \lesssim 1$ by Lemma \ref{ENV}. Using this with Lemma \ref{lemma:monotonicity} and Theorem \ref{Thm3} gives
\begin{align}
\sum_{Q \in S} \beta^{d,p}_{\Sigma^S}(2C_0B_Q)^2\ell(Q) \lesssim \sum_{i \in I} \sum_{\substack{Q \in \mathscr{D}^S \\ Q \subseteq R_i}}\beta_{\Sigma_S}^{d,p}(3C_0B_Q)^2\ell(Q)^d \lesssim \sum_{i \in I} \mathscr{H}^d(R_i) \lesssim \mathscr{H}^d(\Sigma^S \cap C_1B_{Q(S)}). 
\end{align}
Finally, we can apply Lemma \ref{l:goodlem} to get
\begin{align}
I_1 \lesssim \sum_{S  \in \mathscr{F}} \mathscr{H}^d(\Sigma^S \cap C_1B_{Q(S)}) \lesssim \mathscr{H}^d(Q_0) + \BWGL(Q_0,A,\ve).
\end{align}
This completes the proof of Claim 2. 
\end{proof}
Now, Claim 1 and Claim 2 finish the proof of the lemma. 
\end{proof}

\section{Estimating measure}\label{s:Thm2}

In this section we prove Theorem \ref{Thm2}. Let $1 \leq d < \infty, \ 1 \leq p < \infty$, $\ C_0 > 2\rho^{-1}$, $A > 1$ and $\ve > 0$. Let $E \subseteq H$ be lower content $d$-regular, with regularity constant $c$. Let $X_k$ be a sequence of maximally $\rho^k$-separated nets and let $\mathscr{D}$ be the Christ-David cubes for $E$ from Theorem \ref{cubes}. Let $Q_0 \in \mathscr{D}$ and assume without loss of generality that $0 \in Q_0 \in \mathscr{D}_0$. Since we cannot assume $\mathscr{H}^d(E)$ to be finite, we do not know that a maximally separated net in $E$ has a finite number of points. Thus, we cannot reduce to the Euclidean case, as we did in the previous sections. Although, now that we have Theorem \ref{DT} for subsets of Hilbert space, the proof of \eqref{e:dTSP1} is almost identical to that written in \cite[Sections 11-12]{azzam2018analyst}. We note that we must give a direct proof of Lemma \ref{l:bi-lip-BWGL}, which for Euclidean space is a corollary of a result of David and Semmes. This requires some work. Since the rest of the construction will be useful for the proof of Theorem \ref{t:BWGL}, we fill in some details. 

Recall that we are trying to show  
\begin{align}\label{e:Section6'} \mathscr{H}^d(Q_0) + \BWGL(Q_0,A,\ve) \lesssim \ell(Q_0)^d + \sum_{\substack{Q \in \mathscr{D} \\ Q \subseteq Q_0}} \beta_E^{d,p}(C_0B_Q)^2\ell(Q)^d.
\end{align}
By the following (which is proven in the appendix), it suffices to prove the same inequality holds with $C_0$ replaced by some larger constant $M$ and setting $p =1.$ 

\begin{lem}
	Let $\ve > 0$, $A \geq 1$ and $M \geq C_0$, and suppose
	\begin{align}\label{e:Section6} \mathscr{H}^d(Q_0) + \BWGL(Q_0,A,\ve) \lesssim \ell(Q_0)^d + \sum_{\substack{Q \in \mathscr{D} \\ Q \subseteq Q_0}} \beta_E^{d,1}(MB_Q)^2\ell(Q)^d.
	\end{align}
Then, \eqref{e:Section6'} holds with constant $\ve,A$ and $C_0$.  
\end{lem}

\begin{rem}\label{r:setM}
	From here well shall set $M = \max\{50A,C_0\}.$
\end{rem}

Our goal for the rest of the section is to prove \eqref{e:Section6} for $\ve,A$ and $M$. From here onwards, when we write $\mathscr{D}$ we mean those cubes $Q \in \mathscr{D}$ so that $Q \subseteq Q_0.$ Our first task is to split $\mathscr{D}$ into stopping-time regions. For each $Q \in \mathscr{D}$ such that $E \cap C_0B_Q \not= \emptyset$, we define a stopping-time region $S_Q$ as follows. 

Let $\lambda > 0.$ Begin by adding $Q$ to $S_Q$ and inductively on scales, add cubes $R$ to $S_Q$ if each of the following holds,
\begin{enumerate}
\item $R^{(1)} \in S_Q$.
\item Every sibling $R^\prime$ of $R$ satisfies
\[\sum_{R^\prime \subseteq T \subseteq Q} \beta_E^{d,1}(MB_T)^2 < \lambda^2.\]
\end{enumerate}

\begin{rem}
If $\beta_F^{d,1}(MB_Q) \geq \lambda$ then $S_Q = \{Q\}.$
\end{rem}

We partition $\{Q \in \mathscr{D} : E \cap MB_Q \not=\emptyset\}$ into a collection of stopping-time regions as follows. First, add $S_{Q_0}$ to $\mathscr{S}.$ Then, if $S$ has been added to $\mathscr{S},$ and if $Q \in \text{Child}(R)$ for some $R \in \min(S)$ such that $E \cap MB_Q \not= \emptyset,$ add $S_Q$ to $\mathscr{S}.$ Let $\mathscr{S}$ be the collection of stopping-time regions obtained by repeating this process indefinitely. Note that 
\[\sum_{Q \in \mathscr{D}} \beta_E^{d,1}(MB_Q)^2 \ell(Q)^d = \sum_{S \in \mathscr{S}} \sum_{Q \in S} \beta_E^{d,1}(MB_Q)^2 \ell(Q)^d.\]

For each $S \in \mathscr{S}$ which is not a singleton (i.e. $S \not= \{Q\}$) we construct a bi-Lipschitz surface which well approximates $E$ inside $S$. The following is essentially \cite[Lemma 11.3]{azzam2018analyst}, however, the precise statement below is \cite[Lemma 3.7]{hyde2020analyst}.

\begin{lem}\label{Sigma}
There exists $\lambda >0$ small enough so that for each $S \in \mathscr{S}$ which is not a singleton, there is a surface $\Sigma_{S}$ satisfying the follow. First, $\Sigma_S = f_S(\R^d) \cap MB_{Q(S)}$, where $f_S : \R^d \rightarrow H$ is $(1+C\lambda)$-bi-Lipschitz map onto its image. Second, for each $R \in S$ and $y \in F \cap (M^\frac{1}{3}/4)B_R$, 
\begin{align}\label{Closeness}
\emph{dist}(y,\Sigma_{S}) \lesssim \lambda^\frac{1}{d+1}\ell(R).
\end{align}
Finally, for each ball $B$ centred on $\Sigma_S$ and contained in $MB_{Q(S)}$, we have
\begin{align}\label{ee:lowerreg}
 \frac{\omega_d}{2}r_B^d \leq \mathscr{H}^d( \Sigma_{S} \cap B) \lesssim r_B^d.
\end{align}
\end{lem}

The following lemma is \cite[Lemma 11.4]{azzam2018analyst}. The same proof works now that we have Theorem \ref{DT} for subsets of $H.$ 

\begin{lem}\label{l:control-beta}
We have 
\[ \sum_{S \in \mathscr{S}} \ell(Q(S))^d \lesssim \ell(Q_0)^d + \sum_{Q \in \mathscr{D}} \beta^{d,1}_E(MB_Q)^2 \ell(Q)^d. \] 
\end{lem}

Let
\[ G = \left\{x \in Q_0 : \sum_{Q \in \mathscr{D}} \beta^{d,1}_E(MB_Q) \mathds{1}_Q(x) < \infty \right\}. \]
This is the collection of points in $Q_0$ which are only contained in a finite number of stopping-time regions. Thus, for each $x \in G$ there exists a minimal stopping-time region $S$ so that $x \in Q(S)$. This implies $x$ is contained in arbitrarily small cubes from $S$, hence $x \in \Sigma_S$ by \eqref{Closeness}. In particular, one has 
\[ G \subseteq \bigcup_{S \in \mathscr{S}} \Sigma_S. \] 
Let $E_0 = Q_0 \setminus G$ (the collection of points in $Q_0$ which are contained in an infinite number of stopping-time regions). The following gives the required bound for $\mathscr{H}^d(Q_0)$ in \eqref{e:Section6}. 

\begin{lem}
We have $\mathscr{H}^d(E_0) = 0$ and 
\begin{align}\label{e:upper-measure}
\mathscr{H}^d(Q_0) \lesssim \ell(Q_0)^d + \sum_{Q \in \mathscr{D}} \beta^{d,1}_E(MB_Q)^2\ell(Q)^d.
\end{align}
\end{lem}

\begin{proof}
	The fact that $\mathscr{H}^d(E_0) = 0$ follows from \cite[Corollary 11.7]{azzam2018analyst}. Since $\Sigma_S$ is contained in the ball $MB_{Q(S)}$, \eqref{ee:lowerreg} and Lemma \ref{l:control-beta} imply  
	\begin{align}
		\mathscr{H}^d(Q_0) \lesssim \mathscr{H}^d\left( \bigcup_{S \in \mathscr{S}} \Sigma_S \right) \lesssim \sum_{S \in \mathscr{S}} \ell(Q(S))^d \lesssim \ell(Q_0)^d + \sum_{Q \in \mathscr{D}} \beta_E^{d,1}(MB_Q)^2 \ell(Q)^d.
	\end{align}
\end{proof}
	
To finish the proof of \eqref{e:Section6}, we are left to show 
\begin{align}\label{e:bound-BWGL}
	\text{BWGL}(Q_0,A,\ve) \lesssim \ell(Q_0)^d + \sum_{Q \in \mathscr{D}} \beta^{d,1}_E(MB_Q)^2 \ell(Q)^d.
\end{align} 
Let us introduce some notation. Recall the definition of $\beta_{E,\infty}^d$. Now, for a ball $B$ centred on $E$ and a $d$-plane $L$, define the similar quantity 
\[ \eta_{E,\infty}^d(B,L) = \frac{1}{r_B} \sup_{x \in L \cap B} \dist(x,E) \] 
and 
\[ \eta_{E,\infty}^d(B) = \inf \{  \eta_{E,\infty}^d(B,L) : L \mbox{ is a $d$-plane } \}. \] 
Recall that if $Q \in \BWGL(A,\ve)$ then ${b\beta}_E^d(AB_Q) \geq \ve$, with ${b\beta}_E^d$ as in \eqref{e:vartheta}. For such a $Q$, \cite[Lemma 12.6]{azzam2018analyst} implies there exists a constant $C$ (independent of $Q$) so that either $\beta^{d}_{E,\infty}(AB_Q) \geq \ve/C$ or $\eta_{E,\infty}^d(AB_Q) \geq \ve/C.$ Hence, 
\begin{align}
	\BWGL(Q_0,A,\ve) &\leq \sum \{ \ell(Q)^d : Q \in \mathscr{D}, \ \beta^d_{E,\infty}(AB_Q) \geq \ve/C \}  \\
	&\hspace{2em} + \sum \{ \ell(Q)^d : Q \in \mathscr{D}, \ \eta^d_{E,\infty}(AB_Q) \geq \ve/C \}
\end{align}
By Lemma \ref{lemma:betap_betainfty}, we have the following for the first sum;  
\begin{align}
	\sum \{ \ell(Q)^d : Q \in \mathscr{D}, \ \beta^d_{E,\infty}(AB_Q) \geq \ve/C \} &\leq \sum\{ \ell(Q)^d : Q \in \mathscr{D}, \  \beta^{d,1}_E(2AB_Q) \geq c \ve^{d+1} \} \\
	&\leq \frac{1}{(c\ve^{d+1})^2} \sum_{Q \in \mathscr{D}} \beta^{d,1}_E(2AB_Q)^2\ell(Q)^d \\
	&\lesssim  \sum_{Q \in \mathscr{D}} \beta^{d,1}_E(MB_Q)^2\ell(Q)^d,
\end{align}
where the final inequality follows from Lemma \ref{lemma:monotonicity} and Remark \ref{r:setM}, since $2AB_Q \subseteq MB_Q.$ The second sum is dealt with below.  
\begin{lem}\label{l:eta-bound}
	Let $S \in \mathscr{S}.$ Then 
	\[ \sum \{ \ell(Q)^d : Q \in \mathscr{D}, \ \eta^d_{E,\infty}(AB_Q) \geq \ve/C \} \lesssim \ell(Q_0)^d + \sum_{Q \in \mathscr{D}} \beta^{d,1}_E(MB_Q)^2\ell(Q)^d. \] 
\end{lem}

To prove Lemma \ref{l:eta-bound}, will need some preliminary lemmas. In \cite{azzam2018analyst}, the following is stated as a corollary of a more general result for UR sets in $\R^n$. This more general result is exactly Theorem \ref{Thm3} (for subsets of $\R^n$) and so we need a direct proof.

\begin{lem}\label{l:bi-lip-BWGL}
	Let $\delta > 0$ and $S \in \mathscr{S}$. Then 
	\[ \sum \{\ell(Q)^d : Q \in S, \ {b\beta}^d_{\Sigma_S}(AB_Q) \geq \delta \} \lesssim_{\delta} \ell(Q(S))^d. \] 
\end{lem}

The proof of Lemma \ref{l:bi-lip-BWGL} will be split into several lemmas. Fix $S \in \mathscr{S}.$ Let $\Sigma = \Sigma_S$ and let $f = f_S$ be the $(1+C\lambda)$-bi-Lipschitz from Lemma \ref{Sigma} so that $\Sigma = f(\R^d) \cap MB_{Q(S)}.$ Let $\mathscr{I}$ be the usual dyadic grid for $\R^d.$ For $I \in \mathscr{I},$ let $x_I$ be its centre, and define $B_I = B(x_I,\diam(I)).$ Now, for each $Q \in S$, let $x_Q'$ be the point in $\Sigma$ which is closest to $x_Q$ and let $B_Q' = B(x_Q',\ell(Q)).$ By \eqref{Closeness}, for $\lambda$ small enough, we have 
\[4AB_Q \subseteq 5AB_Q'. \]
Finally, define $I_Q$ to be the cube in $\mathscr{I}$ so that $f^{-1}(x_Q') \in I_Q$ and $\tfrac{1}{2}\diam(I_Q) <  6A\ell(Q) \leq \diam(I_Q).$ 

\begin{lem}\label{l:f-inclusion}
	Let $Q \in S.$ Then 
	\[  f^{-1}(\Sigma \cap 4AB_Q) \subseteq f^{-1}(\Sigma \cap 5AB_Q') \subseteq 4B_{I_Q}. \] 
\end{lem}

\begin{proof}
	The first inclusion follows from the fact that $4AB_Q \subseteq 5AB_Q'.$ Now for the second inclusion. Take $z \in f^{-1}(\Sigma \cap 5AB_Q')$ and let $y \in \Sigma \cap 5AB_Q'$ so that $z = f^{-1}(y).$ By the definition of $I_Q$, and the fact that $f$ is $(1+C\lambda)$-bi-Lipschitz, we have 	
	\begin{align}
		|z - x_I| &\leq |z - f^{-1}(x_Q')| + |f^{-1}(x_Q') - x_I| \leq (1+C\ve)|y - x_Q'| + |f^{-1}(x_Q') - x_I| \\ 
		&\leq (1+C\lambda)5A\ell(Q) + \diam(I_Q) \leq 2\diam(I_Q) \leq 4\diam(I_Q).
	\end{align}
\end{proof}

\begin{lem}
	Let $I \in \mathscr{I}.$ Then 
	\begin{align}\label{e:S_I}
		S_I = \{Q \in S : I_Q = I \} \lesssim_d 1.
	\end{align}
\end{lem}

\begin{proof}
	First, since $\ell(I) \sim \ell(Q)$ for each $Q \in S_I,$ there are a bounded number of $k \in \Z$ for which $S_I \cap \mathscr{D}_k \not=\emptyset.$ Fix $k$ for which this intersection is non-empty. For any $Q_1,Q_2 \in S_I \cap  \mathscr{D}_k,$ since $c_0B_{Q_1} \cap c_0B_{Q_2} = \emptyset,$ we know by \eqref{Closeness} that $|x_{Q_1}' - x_{Q_2}'| \geq c_0\ell(Q_1)/2 \sim \ell(I).$ Then, since $f$ is $(1+C\lambda)$-bi-Lipschitz, this implies the preimages $f^{-1}(x_Q'), \  Q \in S_I \cap \mathscr{D}_k$, are a collection of approximately $\ell(I)$ separated points in $\R^d$ which are contained in a ball of radius approximately $\ell(I)$ (by definition of $I$). A standard volume argument in $\R^d$ implies $\#(S_I \cap \mathscr{D}_k)\lesssim_d 1,$ which finishes the proof. 
	
\end{proof}

Let $0 < \alpha \ll \delta.$ Let $S_{\geq}$ be those cubes $Q \in S$ so that $\Omega_{f,2}(8B_{I_Q}) \geq \alpha$ and let $S_\leq = S \setminus S_\geq.$ The definition of $\Omega_{f,2}$ can be found at the beginning of Section \ref{s:Dor}. Then, for each $Q \in S_\leq,$ choose an affine map $\calA_Q$ so that $\Omega_{f,2}(8B_{I_Q},\calA_Q) \leq 2\Omega_{f,2}(8B_{I_Q}).$

\begin{lem}
	Let $Q \in S_\leq.$ For $y \in 4B_{I_Q}$ we have 
	\begin{align}\label{e:fcloseA}
		|f(y) - \calA_Q(y)| \lesssim \alpha^\frac{1}{d+1}\diam(I).
	\end{align}
	Additionally, for $y,z \in \R^d$ we have 
	\begin{align}\label{e:lowerA}
		|\calA_Q(y) - \calA_Q(z)| \geq (1+C \lambda)^{-1}|y-z| - \frac{C\alpha^\frac{1}{d+1}}{\gamma}\diam(I), 
	\end{align}
	where $\gamma = \gamma_{y,z} = \min\{1,4\diam(I)/|y-z|\}$. 
\end{lem}

\begin{proof}
	For brevity, let us write $I = I_Q$ and $\calA = \calA_Q.$ The first statement is a consequence of Lemma \ref{l:omega_infty}, since for any $y \in 4B_{I}$ we have
	\begin{align}
		|f(y) - \calA(y)| &\leq \Omega_{f,\infty}(4B_{I},\calA) \diam(I) \lesssim \Omega_{f,\infty}(4B_{I}) \diam(I) \\ &\lesssim \Omega_{f,2}(8B_{I})^\frac{1}{d+1}\diam(I) \lesssim \alpha^\frac{1}{d+1} \diam(I).  
	\end{align}	
	Let us focus on the second statement. We start with the points $y',z' \in 4B_{I}.$ Notice, for these points, we have
	\begin{align}\label{e:points4B_I}
		|\calA(y') - \calA(z')| &\geq |f(y') - f(z')| - |f(y') - \calA(y')| - |f(z') - \calA(z')| \\
		& \geq (1+C\lambda)^{-1}|y'-z'| - C\alpha^\frac{1}{d+1}\diam(I).
	\end{align} 	
	Now consider $y,z \in \R^d$ to be any arbitrary points. Let $\gamma = \gamma_{y,z} = \min\{1,4\diam(I)/|y-z|\}$ and let $\calL$ denote the linear part of $\calA$. Then 
	\begin{align}
		|\calA(y) - \calA(z)| &= |\calL(y-z)| = \frac{1}{\gamma}|\calL(\gamma\cdot(y-z) + x_I) - \calL(x_I))| = \frac{1}{\gamma}|\calA(\gamma\cdot(y-z) + x_I) - \calA(x_I)| \\
		&\geq \frac{1}{\gamma} \left( (1+C\lambda)^{-1}|\gamma\cdot(y-z)| - C\alpha^\frac{1}{d+1}\diam(I) \right) \\
		& \geq (1+C\lambda)^{-1}|y-z| -\frac{C\alpha^\frac{1}{d+1}}{\gamma}\diam(I),
	\end{align}	
	where the penultimate inequality follows from \eqref{e:points4B_I} since $x_I,x_I+ \gamma\cdot(y-z) \in 4B_I.$
	
\end{proof}

\begin{lem}\label{l:proj-covers}
	There exists $\alpha > 0$ small enough so that for any $Q \in S_\leq$ we have
	\begin{align}\label{e:proj-covers}
		\pi_{\calA_Q} \circ f( 4B_{I_Q} ) \supseteq \calA_Q( 3B_{I_Q} ) \supseteq \calA_Q ( \R^d ) \cap AB_Q,
	\end{align}
	where $\pi_{\calA_Q}$ denotes the orthogonal projection onto the $d$-plane $\calA_Q(\R^d).$  
\end{lem}

\begin{proof}
	
	We start with the second inclusion. Again, let us write $I = I_Q$ and $\calA = \calA_Q.$ Let $y_0 \in \calA(\R^d) \cap AB_Q$ and let $z_0 \in \R^d$ be such that $y_0 = \calA(z_0).$ We will be done with this part once we show $z_0  \in 3B_{I_Q}.$ By \eqref{e:fcloseA}, \eqref{e:lowerA}, and recalling $f$ is $(1+C\lambda)$-bi-Lipschitz and $f^{-1}(x_Q') \in I,$ we have 		 
	\begin{align}\label{e:top-upper}
		|z_0 - x_{I}| &\leq (1+C\lambda)\left[|y_0 - \calA(x_{I})| + \frac{C\alpha^\frac{1}{d+1}}{\gamma}\diam(I) \right] \\
		&\hspace{-2em} \leq (1+C\lambda)\left[ |y_0 - x_Q'| + |f(f^{-1}(x_Q')) - f(x_{I})| + |f(x_{I}) - \calA(x_{I})| + \frac{C\alpha^\frac{1}{d+1}}{\gamma}\diam(I) \right] \\
		& \leq 2A\ell(Q) + (1+C\lambda) \diam(I) + C\alpha^\frac{1}{d+1} \diam(I) + \frac{C\alpha^\frac{1}{d+1}}{\gamma}\diam(I) \\
		&\leq 2\diam(I) + \frac{C\alpha^\frac{1}{d+1}}{\gamma}\diam(I).
	\end{align}
	Suppose, first of all, that $z_0 \not\in 4B_{I}.$ In this case, $\gamma = 4\diam(I)/|z_0 - x_I|$, and \eqref{e:top-upper} implies $|z_0 - x_I| \leq 2\diam(I) + \tfrac{1}{3}|z_0-x_I|$ for $\alpha$ small enough. Rearranging, we get $|z_0-x_I| \leq 3\diam(I)$ which is a contradiction. It follows that $z_0 \in 4B_I$ and $\gamma = 1.$ Then, for $\alpha$ small enough, \eqref{e:top-upper} implies $|z_0 - x_I| \leq 3\diam(I)$ as required for the second inclusion.
	
	We turn our attention to the first inclusion, which requires some topological consideration. Let $y_1 \in \calA(3B_I)$ and let $z_1 \in 3B_I$ be so that $y_1 = \calA(z_1).$ Notice, since $f$ is bi-Lipschitz and $\calA$ well approximates $f,$ then $\text{Rank}(\calA) = d.$ In particular, $\calA$ is invertible. Now, define $g = \pi_\calA \circ f$ and for $z \in \partial = \partial(4B_{I_Q})$ consider the map 
	\begin{align}\label{e:map-g}
		z \mapsto \frac{\calA^{-1}(g(z)) - \calA^{-1}(y_1)}{||\calA^{-1}(g(z)) - \calA^{-1}(y_1)||} = \frac{\calA^{-1}(g(z)) - z_1}{||\calA^{-1}(g(z)) - z_1||} .
	\end{align} 
	Let us see why the above map is well-defined. Let $z \in \partial.$ By definition, $g(z)$ is the point in $\calA(\R^d)$ which is closest to $f(z).$ Using this, with Lemma \ref{l:omega_infty}, \eqref{e:fcloseA}, and the fact that $f$ is $(1+C\lambda)$-bi-Lipschitz, we have   
	\begin{align}
		\diam(I_Q) &\leq |z-z_1| \leq (1+C \lambda)|f(z) - f(z_1)| \lesssim |f(z) - g(z)| + |g(z) - y_1| + |y_1 - f(z_1)|  \\ 
		&\leq |f(z) - \calA(z)| + |g(z) - y_1| + |\calA(z_1) - f(z_1)| \lesssim |g(z) - y_1| + \alpha^\frac{1}{d+1}\diam(I_Q). 
	\end{align}
	Taking $\alpha$ small enough and rearranging implies $|g(z) - y_1| \gtrsim \diam(I_Q)$ and so the map is well-defined. Additionally, since 
	\[ |g(z) - \calA(z)| \leq |g(z) - f(z)| + |f(z) - \calA(z)| \leq 2|f(z) - \calA(z)| \lesssim \alpha^\frac{1}{d+1}\diam(I_Q), \] 
	we see that the segment $[g(z) , \calA(z)]$ does not contain the point $y_1$ for $\alpha$ small enough. Hence, the map in \eqref{e:map-g} is homotopic to the map defined by
	\begin{align}
		z \mapsto \frac{\calA^{-1}(\calA(z)) - z_1}{||\calA^{-1}(\calA(z)) - z_1||} = \frac{z-z_1}{||z- z_1||}. 
	\end{align}
	Moving $z_1$ to $x_{I}$ along the segment $[z_1,x_{I}]$ (observing this does not meet $\partial$), this map is homotopic to  
	\[ z \mapsto \frac{z-x_{I}}{||z-x_{I}||} \]
	which has degree equal to 1. Thus, the map defined in \eqref{e:map-g} has degree equal to 1. This implies $y_1 \in \pi \circ f(4B_{I_Q})$ since otherwise we could define a homotopy from \eqref{e:map-g} to a constant by setting
	\[h(z,t) = \frac{\calA^{-1}(g(tz + (1-t)x)) - \calA^{-1}(y_1)}{|| A^{-1}(g(tz + (1-t)x)) - \calA^{-1}(y_1) ||} \] 
	for $z \in \partial, \ 0 \leq t \leq 1,$ which is impossible. 
\end{proof}

\begin{lem}\label{l:f-in-Sigma}
	For each $Q \in S,$ we have 
	\[ f(4B_{I_Q}) \subseteq \Sigma. \] 
\end{lem}

\begin{proof}
	Let $Q \in S$ and $y \in f(4B_{I_Q}).$ Clearly $f^{-1}(y) \in 4B_{I_Q}.$ Then, by \eqref{Closeness}, and the fact that $f^{-1}(x_Q') \in I_Q \subseteq 4B_{I_Q}$ and $f$ is $(1+C\lambda)$-bi-Lipschitz, we have 
	\begin{align}
		|y - x_Q| &\leq |y-x_Q'| + |x_Q' - x_Q| = |f(f^{-1}(y)) -f(f^{-1}(x_Q'))| + |x_Q' - x_Q| \\ 
		&\leq (1+C\lambda)|f^{-1}(y) - f^{-1}(x_Q')| + |x_Q' - x_Q| \\
		&\leq (1+C\lambda)\diam(4I_Q) + C\lambda^\frac{1}{d+1} \ell(Q) \\
		&\leq 50A\ell(Q) \leq M \ell(Q).
	\end{align}
	This implies that 
	\[f(4B_{I_Q}) \subseteq f(\R^d) \cap MB_Q \subseteq f(\R^d) \cap MB_{Q(S)} = \Sigma \]
	and completes the proof.  
\end{proof}

\begin{proof}[Proof of Lemma \ref{l:bi-lip-BWGL}] 
	
	Starting with $S_\geq,$ by \eqref{e:S_I} and Corollary \ref{CorDor}, we know 
	\begin{align}
		\sum \{ \ell(Q)^d : Q \in S_\geq, \  {b\beta}_{\Sigma}^d(AB_Q) \geq \delta \}  &\lesssim \sum \{ \ell(I_Q)^d : Q \in S_\geq\} \\ &\lesssim_d \sum \{ \ell(I)^d : I \in \mathscr{I}, \ \Omega_{f,2}(8B_I) \geq \alpha \} \\ 
		&\lesssim_\alpha \sum_{\substack{I \in \mathscr{I}  \\  I \subseteq I_{Q(S)}}} \Omega_{f,2}(8B_I)^2 \ell(I)^d \lesssim \ell(Q(S))^d.
	\end{align}
	We are left to bound sum over cubes in $S_\leq.$ For $Q \in S_\leq$, by Lemma \ref{lemma:betap_betainfty}, Lemma \ref{l:f-inclusion}, and Lemma \ref{BetaBound}, we know 
	\[ \beta_{\Sigma,\infty}^d(2AB_Q,\calA_Q(\R^d)) \lesssim \beta^{d,2}_{\Sigma}(4AB_Q,\calA_Q(\R^d))^\frac{1}{d+1} \lesssim \Omega_{f,2}(8B_{I_Q})^\frac{1}{d+1}. \]
	We wish to prove a similar bound for $\eta^d_{\Sigma_S,\infty}.$ Let $y \in  \calA(\R^d) \cap AB_Q$. By Lemma \ref{l:proj-covers} and Lemma \ref{l:f-in-Sigma}, we can find $z' \in 4B_{I_Q}$ and $z \in f(4B_{I_Q}) \subseteq \Sigma$ so that $z = f(z')$ and $y =\pi_{\calA_Q(\R^d)}(z)$. Since $y$ is the point in $\calA_Q(\R^d)$ which is closest to $z,$ and $z' \in 4B_{I_Q}$, we have 
	\[ |y - z| \leq |\calA_Q(z') - f(z')| \lesssim \alpha^\frac{1}{d+1} \diam(I) \] 
	and so $z \in \Sigma \cap 2AB_Q$ for $\alpha$ small enough. We also have
	\[ \dist(y,\Sigma) \leq |y-z| = \dist(z,\calA_Q(\R^d)) \] 
	which after taking supremum and using the fact that $z \in \Sigma \cap 2AB_Q$, implies 
	\[ \eta_{\Sigma,\infty}^d(AB_Q,\calA_Q(\R^d)) \leq \beta_{\Sigma,\infty}^d(2AB_Q,\calA_Q(\R^d)). \] 
	Combining the two, we get 
	\[ {b\beta}^d_\Sigma(AB_Q) \leq {b\beta}_\Sigma^d(AB_Q,\calA_Q(\R^d)) \lesssim \beta^d_{\Sigma,\infty}(2AB_Q,\calA_Q(\R^d)) \lesssim \Omega_{f,2}(8B_{I_Q})^\frac{1}{d+1}, \] 
	hence, 
	\begin{align}
		\sum\{ \ell(Q)^d : Q \in S_\leq, \  {b\beta}_{\Sigma}^d(AB_Q) \geq \delta \} &\leq \sum\{ \ell(Q)^d : Q \in S_\leq, \   \Omega_{f,2}(8B_{I_Q}) \geq c \delta^{d+1} \} \\
		&\lesssim_d \sum \{ \ell(I)^d : I \in \mathscr{I}, \  \Omega_{f,2}(8B_{I_Q}) \geq c \delta^{d+1} \} \\ 
		&\lesssim_\delta \sum_{\substack{I \in \mathscr{I}  \\  I \subseteq I_{Q(S)}}} \Omega_{f,2}(8B_I)^2 \ell(I)^d \lesssim \ell(Q(S))^d. 
	\end{align}
	This finishes the proof.
	
\end{proof}

\begin{proof}[Proof of Lemma \ref{l:eta-bound}] 
	For $S \in \mathscr{S}$ let $S_\ve$ be the set of cubes in $S$ so that $\eta_{E,\infty}^d(AB_Q) \geq \ve/C$ but $\eta_{\Sigma_S,\infty}^d(AB_Q) < \ve/3C.$ By \cite[Proposition 12.7]{azzam2018analyst}, 
\[ \sum_{Q \in S_\ve} \ell(Q)^d \lesssim \ell(Q(S))^d. \] 
Combining this with Lemma \ref{l:control-beta} and Lemma \ref{l:bi-lip-BWGL} (with $\delta = \ve/3C$), we have 
\begin{align}
\sum_{\substack{Q \in \mathscr{D} \\ \eta^d_{E,\infty}(AB_Q) \geq \ve/C}} \ell(Q)^d &= \sum_{S \in \mathscr{S}} \left[ \sum_{Q \in S_\ve} \ell(Q)^d + \sum_{\substack{Q \in S \\ \eta_{\Sigma_S,\infty}^d(AB_Q) \geq \ve/3C}} \ell(Q)^d \right] \lesssim \sum_{S \in \mathscr{S}} \ell(Q(S))^d \\
&\lesssim \ell(Q_0)^d + \sum_{Q \in \mathscr{D}} \beta^{d,1}_E(MB_Q)^2\ell(Q)^d. 
\end{align}
This finishes to proof of Lemma \ref{l:eta-bound} and hence the proof of \eqref{e:Section6}.
\end{proof} 

\section{Bilateral Weak Geometric Lemma}\label{s:BWGL}

	In this section we prove Theorem \ref{t:BWGL}. We start with the following direction, the proof of which is based on the arguments in \cite[Section 16]{david1991singular}. 
\begin{prop}\label{p:BWGL1}
	Suppose $E \subseteq H$ is Ahlfors $d$-regular and satisfies the $\BWGL$ (Definition \ref{d:BWGL}), then $E$ is \emph{UR} (Definition \ref{d:UR}).
\end{prop}

The main part of the proof of Proposition \ref{p:BWGL1} is to show the following.  

\begin{lem}\label{l:BWGL1}
	Let $\ve > 0.$ Then there exists $L \geq 1$ which depends on $\ve$ so that the following holds. Let $x \in E$ and $0 < r< \diam(E).$ Then there exists a set $F \subseteq E \cap B(x,r)$ so that $\mathscr{H}^d(E \cap B(x,r) \setminus F) \leq \ve r^d$ and an $L$-bi-Lipschitz map $f : F \rightarrow \R^d.$
\end{lem}
Before proving Lemma \ref{l:BWGL1}, let us see why it implies Proposition \ref{p:BWGL1}.
\begin{proof}[Proof of Proposition \ref{p:BWGL1}]
	Let $A$ denote the Ahlfors regularity constant for $E$, set $\ve = (2A)^{-1}$ and let $L = L(\ve)$ be the constant so that Lemma \ref{l:BWGL1} holds. Let $x \in E$, $0 < r< \diam(E)$, and let $F$ and $f : F \rightarrow \R^d$ be the set and $L$-bi-Lipschitz map, respectively, from Lemma \ref{l:BWGL1} for $x$ and $r.$ Additionally, define the function $h : F \rightarrow \R^d$ by setting $h(y) = \tfrac{f(y) - f(x)}{L}.$ In this way, $h$ is $L^2$-bi-Lipschitz and $h(F) \subseteq B_d(0,r).$ Since $h$ is $L^2$-bi-Lipschitz, it is invertible with $L^2$-bi-Lipschitz inverse $h^{-1} : h(F) \rightarrow H.$ By the Kirszbraun extension theorem, we can extend $h^{-1}$ to a Lipschitz function from $\R^d$ to $H$ (whose Lipschitz constant depends only on $L$). Then,  
	\begin{align}
		\mathscr{H}^d(E \cap B(x,r) \cap h^{-1}(B_d(0,r))) &\geq \mathscr{H}^d(E \cap B(x,r) \cap F) \\
		&\geq \mathscr{H}^d(E \cap B(x,r) ) - \mathscr{H}^d(E \cap B(x,r) \setminus F) \\
		&\geq (2A)^{-1} r^d.
	\end{align} 
Since $A$ and $L$ are independent of $x$ and $r,$ this finishes the proof. 
\end{proof}
Let us now focus on proving Lemma \ref{l:BWGL1}. Let $\ve > 0$, $x \in E$ and $0 < r < \diam(E).$ Our first task is to define $F,$ which will be chosen by running a certain stopping-time argument, then choosing parts of $E \cap B(x,r)$ for which we do not stop too many times, and do not lie too near the boundaries of certain cubes. 

Let $C_0 > 1,$ $\mathscr{D}$ be the Christ-David cubes for $E$, $k$ be such that $5\rho^{k+1} < r \leq 5\rho^k,$ and $\{Q_{0,i}\}_{i \in I}$ be the cubes in $\mathscr{D}_k$ that intersect $B(x,r).$ For each $i \in I,$ using the stopping-time construction for the cubes contained in $Q_{0,i}$, outlined at the beginning of Section \ref{s:Thm2} (with $M = C_0$), we find a collection of stopping-time regions $\mathscr{S}_i$ so that the following holds. First, 
\begin{align}\label{e:S_i}
	\{Q \in \mathscr{D} : Q \subseteq Q_{0,i}\} = \bigcup_{S \in \mathscr{S}_i} S. 
\end{align} 
Second, by Lemma \ref{l:control-beta},
\[ \sum_{S \in \mathscr{S}_i } \ell(Q(S))^d \lesssim \ell(Q_{0,i})^d + \sum_{Q \subseteq Q_{0,i}} \beta^{d,1}_E(C_0B_Q)^2\ell(Q)^d \lesssim \mathscr{H}^d(Q_{0,i}) \lesssim \ell(Q_{0,i})^d, \]
where the penultimate inequality follows from Theorem \ref{Thm1} and the fact that $E$ satisfies the BWGL, and the final inequality follows from Ahlfors regularity. Finally, for each $S \in \mathscr{S}_i$ which is not a singleton, there is a surface $\Sigma_S$ satisfying the properties of Lemma \ref{Sigma}. 

Let $R_0$ be the union of the cubes $Q_{0,i}$ and
\[ \mathscr{S} = \bigcup_{i \in I} \{S : S \in \mathscr{S}_i\}. \]
Using the same terminology as \cite{david1991singular}, we call a cube $Q \subseteq R_0$ a \textit{transition cube} if it is the top cube, or a minimal cube of some $S \in \mathscr{S}$. Let $T$ denote the set of transition cubes. For $Q \in \mathscr{D},$ let $N(Q)$ denote the number of transition cubes containing $Q$ (this is the quantity $\ell(Q)$ in \cite{david1991singular}). Let $\eta > 0$ (which will be chosen small momentarily) and define 
\[ \sigma(Q) =  \{ x \in Q : \dist(x,E \setminus Q) \leq \eta \rho^k \}. \] 
Let $N \geq 1$ (which will be chosen large) and define $F \subseteq E \cap B(x,r)$ by setting 
\[ F = E \cap B(x,r) \cap \left(\bigcup_{Q \in T} \sigma(Q) \right)^c \cap\left( \bigcup_{\substack{Q \in T\\ N(Q) \geq N}} Q \right)^c. \]
Now that we have defined $F$, let us show that it takes up a large portion of $E \cap B(x,r).$ We first need the following. 

\begin{lem}
	We have 
	\begin{align}\label{e:T}
		\sum_{Q \in T} \ell(Q)^d \lesssim r^d. 
	\end{align}
\end{lem} 

\begin{proof}
	Since $E$ is doubling (by virtue of being Ahlfors $d$-regular), there are a bounded number of cubes in $R_0$, with constant depending only on $d$ and the Ahlfors regularity constant. In other words 
	\[ \#I \lesssim 1. \]
	Using this with \eqref{e:S_i}, and the fact that $r \sim \ell(Q_{0,i})$ for each $i \in I$, we have 
	\begin{align}\label{e:S}
		\sum_{S \in \mathscr{S}} \ell(Q(S))^d \lesssim \sum_{i \in I} \sum_{S \in \mathscr{S}_i} \ell(Q(S))^d \lesssim \sum_{i \in I} \ell(Q_{0,i})^d \lesssim r^d. 
	\end{align} 
	Since $T$ is the collection of all top cubes and minimal cubes from the stopping-time regions in $\mathscr{S}$, and every minimal cube has a child which is a top cube, \eqref{e:T} follows from \eqref{e:S}.
\end{proof}

\begin{lem}
	For $\eta$ small enough and $N$ large enough, $\mathscr{H}^d(E \cap B(x,r) \setminus F) \leq \ve r^d$.
\end{lem}

\begin{proof}
	By Lemma \ref{cubes}(4) and \eqref{e:T}, for $\eta > 0$ small enough, we have 
	\[ \sum_{Q \in T} \mathscr{H}^d(\sigma(Q)) \leq C \eta^{\frac{1}{C}} \sum_{Q \in T} \mathscr{H}^d(Q) \leq C \eta^{\frac{1}{C}} \sum_{Q \in T} \ell(Q)^d \leq C\eta^\frac{1}{C} r^d \leq \ve r^d/2. \] 
	Similarly, for $L$ large enough, if 
	\[ A = \bigcup_{\substack{Q \in T \\ N(Q) \geq N }} Q \] 
	then 
	\[ \mathscr{H}^d(A) \leq N^{-1} \sum_{Q \in T} \mathscr{H}^d(Q) \leq C N^{-1} \sum_{Q \in T} \ell(Q)^d \leq C N^{-1}r^d \leq \ve r^d/2. \] 
	Combing the above estimates, we have 
	\begin{align}\label{e:bigFmeasure}
		\mathscr{H}^d(E \cap B(x,r) \setminus F) \leq \ve r^d.
	\end{align}
\end{proof}

The next step in the proof of Lemma \ref{l:BWGL1} is to construct the bi-Lipschitz map from $F$ to $\R^d$, which will be done by stitching together the maps associated to each stopping-time region $S \in \mathscr{S}$. The amount of stitching is controlled by the constant $N$. We begin by defining a preliminary map $g$ on the set of cubes $Q$ contained in $R_0$.

\begin{lem}
	 There exists a map $g$ from the set of transition cubes $T$ so that for each such $Q$, $g(Q)$ is a cube in $\R^d$ with $\diam(g(Q)) = C_3^{-N(Q)} \ell(Q)$ for some large constant $C_3 > 1$, and for any $Q,Q' \in T$ satisfying $Q \subseteq Q'$ then $g(Q) \subseteq \tfrac{1}{2}g(Q').$
\end{lem}

\begin{proof}
	We define $g$ inductively, starting with the cubes $Q_{0,i}.$ Since $\#I \lesssim 1,$ we can find $C > 1$ large enough so that we can set $\{g(Q_{0,i})\}_{i \in I}$ to be a collection of cubes in $\R^d$ with diameters $5\rho^k$ and mutual distances between $5\rho^k$ and $5C\rho^k$ (with $k$ as defined after the statement of Proposition \ref{p:BWGL1}). Since $N(Q_{i,0}) = 0$ for each $i \in I,$ the cubes $g(Q_{i,0})$ have the correct diameter.  
	
	We now define $g(Q)$ for the remaining transition cubes. Suppose $g(Q)$ has been defined for some $Q \in T_m = T \cap \mathscr{D}_m$ with $m \geq k.$ Assume first of all that $Q$ is a minimal cube of a stopping-time region. We will define $g(R)$ for each child $R$ of $Q$. By definition $N(R) = N(Q) + 1$ for each child $R$ of $Q$. Since $E$ is doubling, $Q$ has a bounded number of children (with constant depending only on $d$ and the Ahlfors regularity constant). Thus, for $C_3> 1$ large enough, we can choose $g(R)$ to be a collection of cubes with diameters $C_3^{-N(R)}\ell(R)$ contained in $\tfrac{1}{2}g(Q)$ and so that 
	\begin{align}\label{e:g-estimate}
	 	\dist(g(R),g(R')) \geq C_3^{-N(R)} \ell(R).
	 \end{align} 

	Assume now that $Q$ is the top cube of a stopping-time region $S \in \mathscr{S}$, which is not a singleton. We want to define $g(R)$ for the minimal cubes of $S$. Here, we deviate slightly from that written in \cite{david1991singular}. Let $h = h_S$ be the bi-Lipschitz map from Lemma \ref{Sigma} and let $\Sigma = \Sigma_S = h(\R^d) \cap MB_{Q}.$ Since $h$ is bi-Lipschitz with constant close to 1, $h^{-1}(\Sigma \cap B_Q)$ is contained in some ball of radius $2\ell(Q).$ Thus, since $\diam(g(Q)) = C_3^{-N(Q)}\ell(Q),$ we can find an affine map $\phi = \phi_Q$ so that
\begin{align}\label{e:in-g(Q)}
\phi(h^{-1}(\Sigma \cap B_Q)) \subseteq \frac{1}{3}g(Q)
\end{align}
and 
\begin{align}\label{e:phi-est} C_4^{-1}C_3^{-N(Q)}|p-q| \leq |\phi(p) - \phi(q)| \leq C_3^{-N(Q)}|p-q| 
\end{align}
for some constant $C_4 \geq 1$ and each $p,q \in \R^d.$ 

For $R \in \min(S)$, let $z_R \in \Sigma$ be the point closest to $x_R$ (the centre of $R$). Define $g(R)$ to be the cube centred at $\phi(h^{-1}(z_R))$ of size $C_3^{-N(Q)-1} \ell(R) = C_3^{-N(R)}\ell(R).$ Clearly then $g(R) \subseteq \tfrac{1}{2}g(Q).$ 
\end{proof}

Continuing with the assumption that $Q$ is the top cube of a stopping-time region $S$, which is not a singleton, similarly to \eqref{e:g-estimate}, we will need to estimate the relative distance between $g(R),g(R')$ for $R,R' \in \min(S)$. This will facilitate our forthcoming estimates on the function $f$ (still to be defined). 
 
\begin{lem}
For $R,R' \in \min(S)$ there is $C_3 > 1$ large enough so that 
\begin{align}\label{e:g-lip} 
	C_3^{-N(Q) -1}    [\ell(R) + \ell(R') + \dist(R,R')] &\leq \dist(g(R),g(R')) \\
	 &\hspace{2em}\leq   3C_3^{-N(Q)} [\ell(R) + \ell(R') + \dist(R,R')] .
\end{align}	
\end{lem}
\begin{proof}
Let $R_1,R_2 \in \min(S).$ Since $x_{R_i} \in c_0B_{R_i}$ for each $i=1,2,$ and $c_0B_{R_1} \cap c_0B_{R_2} = \emptyset$, we can find a constant $C_5$ so that   
\begin{align}
	C_5^{-1} [ \ell(R_1) + \ell(R_2) + \dist(R_1,R_2) ] \leq |x_{R_1} - x_{R_2}| \leq \ell(R_1) + \ell(R_2) + \dist(R_1,R_2).
\end{align}
Recalling the definition of $z_R$ (just below \eqref{e:phi-est}), by Lemma \ref{Sigma} we know $|x_{R_i} - z_{R_i}| \leq C_6 \ve^\frac{1}{d+1} \ell(R_i)$ for some constant $C_6$ and for each $i =1,2.$ Recall also that, $h$ is $(1+C_7\ve)$-bi-Lipschitz for some constant $C_7.$ Choosing $\ve^\frac{1}{d+1} \leq \min\{(2C_5C_6)^{-1}, C_7^{-1}\},$ we have
\begin{align}
	(2C_5)^{-1}  [\ell(R_1) + \ell(R_2) + \dist(R_1,R_2)]  &\leq C_5^{-1} [\ell(R_1) + \ell(R_2) + \dist(R_1,R_2)] \\ 
	&\hspace{4em} - C_6 \ve^\frac{1}{d+1} \ell(R_1) - C_6 \ve^\frac{1}{d+1} \ell(R_2)   \\
	&\leq |x_{R_1} - x_{R_2}| - C_6 \ve^\frac{1}{d+1} \ell(R_1) - C_6 \ve^\frac{1}{d+1} \ell(R_2) \\
	&\leq |z_{R_1} - z_{R_2}| \leq (1+C_7\ve)|h^{-1}(z_{R_1}) - h^{-1}(z_{R_2})| \\
	&\leq 2 |h^{-1}(z_{R_1}) - h^{-1}(z_{R_2})|.
\end{align} 
After rearranging, this gives
\begin{align}\label{e:lower-h}
	|h^{-1}(z_{R_1}) - h^{-1}(z_{R_2})| \geq (4C_5)^{-1}  [\ell(R_1) + \ell(R_2) + \dist(R_1,R_2)].
\end{align}
Similarly, one can also show 
\begin{align}
	|h^{-1}(z_{R_1}) - h^{-1}(z_{R_2})| \leq 2  [\ell(R_1) + \ell(R_2) + \dist(R_1,R_2)].
\end{align}
Let $y_{R_1} \in g(R_1)$ and $y_{R_2} \in g(R_2)$ be points so that $|y_{R_1}-y_{R_2}| = \dist(g(R_1),g(R_2)).$ Recalling that $g(R_i)$ is centred at $\phi(h^{-1}(z_{R_i}))$ and $\diam(g(R_i)) = C_3^{-N(Q)-1} \ell(R_i)$, choosing $C_3 \geq 8C_4C_5$, and using \eqref{e:phi-est} and \eqref{e:lower-h}, we have
\begin{align}
	\dist(g(R_1),g(R_2)) &= |y_{R_1}-y_{R_2}|  \\
	 & \geq |\phi(h^{-1}(z_{R_1})) - \phi(h^{-1}(z_{R_2}))| - |y_{R_1}- \phi(h^{-1}(z_{R_1}))| - |y_{R_2} - \phi(h^{-1}(z_{R_2}))| \\
	 & \geq (4C_4C_5)^{-1}C_3^{-N(Q)} [\ell(R_1) + \ell(R_2) + \dist(R_1,R_2)] \\
	 &\hspace{4em} - C_3^{-N(Q)-1} \ell(R_1) - C_3^{-N(Q)-1}\ell(R_2) \\
	 & \geq C_3^{-N(Q)-1} [ \ell(R_1) + \ell(R_2) + \dist(R_1,R_2) ], 
\end{align}
The reverse inequality follows similarly;
\begin{align}
	\dist(g(R_1),g(R_2)) &= |y_{R_1}-y_{R_2}|  \\
	& \leq |\phi(h^{-1}(z_{R_1})) - \phi(h^{-1}(z_{R_2}))| + |y_{R_1}- \phi(h^{-1}(z_{R_1}))| + |y_{R_2} - \phi(h^{-1}(z_{R_2}))| \\
	& \leq 2C_3^{-N(Q)} [\ell(R) + \ell(R') + \dist(R,R')] + C_3^{-N(Q)-1} \ell(R) + C_3^{-N(Q)-1}\ell(R') \\
	& \leq 3C_3^{-N(Q)} [ \ell(R) + \ell(R') + \dist(R,R') ].
\end{align}

\end{proof}

Now that we have defined $g,$ we can use it to define a function $f$. By definition, for each $x \in F$ there is a minimal transition cube, which we denote by $Q(x),$ so that $x \in Q(x).$ It must be that $Q(x)$ is top cube of a stopping-time region $S(x)$ (which is not a singleton), otherwise $x$ would be contained in a smaller transition cube. Furthermore, $x$ is contained in arbitrarily small cubes from $S(x)$ and so $x \in F \cap \Sigma_{S(x)}$ by \eqref{Closeness}. Since $x \in \Sigma_{S(x)},$ $h_{Q(x)}^{-1}(x)$ makes sense, and we can define 
\[ f(x) = \phi_{Q(x)}(h^{-1}_{Q(x)}(x)). \] 
Notice, by \eqref{e:in-g(Q)}, $f(x) \in g(Q(x)).$ The proof of Lemma \ref{l:BWGL1} is finished once we show the following. 

\begin{lem}
	The map $f: F \rightarrow \R^d$ is $L$-bi-Lipschitz, where $L$ depends on $\ve.$ 
\end{lem}

\begin{proof}
Let $x,y \in F$ be distinct points. Let $Q_1$ be the maximal cube contained in $R_0$ so that $x \in Q_1$ but $y \not \in Q_1.$ Assume to begin with that $Q_1$ is contained in some stopping-time region $S$, but that $Q_1$ is not the top cube of $S.$ By maximality, any cube containing $Q_1$ contains both $x$ and $y$, in particular, this holds true for $Q(S).$ Assume further that there are cubes $R_x,R_y \in \min(S)$ so that $x \in R_x$ and $y \in R_y.$ This implies, first, that 
\begin{align}\label{e:R_x-low}
	 |x-y| \leq \ell(R_x) + \ell(R_y) + \dist(R_x,R_y).
\end{align}  
Secondly, since $x \not \in \sigma(R_x)$ and $y \not\in \sigma(R_y)$ (by the definition of $F$), we also have 
\begin{align}\label{e:R_x-up} 
	|x-y| \gtrsim  \ell(R_x) + \ell(R_y) + \dist(R_x,R_y).
\end{align}
The estimates on $f$ then follow from the fact that $f(x) \in g(R_x),$ $f(y) \in g(R_y),$ $N(Q) \leq N$ and the estimates \eqref{e:g-lip}, \eqref{e:R_x-low}, and \eqref{e:R_x-up}, since
\begin{align}
	|x -y| &\leq  \ell(R_x) + \ell(R_y) + \dist(R_x,R_y) \lesssim_L \dist(g(R_x),g(R_y)) \leq |f(x) - f(y)| \\
	&\leq \dist(g(R_x),g(R_y)) + \diam(R_x) + \diam(R_y) \lesssim_N \ell(R_x) + \ell(R_y) + \dist(R_x,R_y) \\
	&\lesssim |x-y|.
\end{align} 
If it is the case that $x$ is not contained in any cube from $\min(S)$, then $f(x) = \phi_{Q(S)}(h^{-1}_{Q(S)}(x))$. By similar arguments to \eqref{e:g-lip}, it can be shown that
\[ C_3^{-N(Q) -1}    [ \ell(R_y) + \dist(x,R_y)] \leq \dist(x,g(R_y)) \leq   3C_3^{-N(Q)} [\ell(R_y) + \dist(x,R_y)], \]
and the estimates follow in the same way as above. Swapping the roles of $x$ and $y$, the same is true if $y$ is not contained in any cube from $\min(S).$ If neither $x$ nor $y$ are contained in minimal cubes, the estimates follow directly from an argument similar to that which established \eqref{e:g-lip}. 

Now we are left with the case that $Q_1$ is the top of a stopping-time region. Let $Q_2$ be the parent of $Q_1.$ It must be that $Q_2$ is a transition cube, and by maximality, $x,y \in Q_2.$ Since $x,y \in Q_2$, $x \in Q_1\setminus \sigma(Q_1)$ and $y \not\in Q_1$ we have $|x-y| \sim \ell(Q_1).$
To estimate the distance between $f(x)$ and $f(y)$, first note that since $x,y \in Q_2,$ we have $f(x),f(y) \in g(Q_2)$ and so 
\[|f(x) - f(y)| \leq \diam(g(Q_2)) \lesssim \ell(Q_2) \sim \ell(Q_1) \sim |x-y| .\]
Next, let $Q_1'$ be the sibling of $Q_1$ so that $y \in Q_1'.$ By \eqref{e:g-lip} we have
\[ |f(x) - f(y)| \geq \dist(g(Q_1),g(Q_1')) \gtrsim \ell(Q_1) \sim |x-y| \]
and this finishes the proof of the bi-Lipschitz estimates. 

\end{proof}

Now for the converse to Proposition \ref{p:BWGL1}. 

\begin{prop}\label{p:BWGL2}
	Suppose $E$ is \emph{UR}, then $E$ satisfies the $\BWGL.$ 
\end{prop}

We will need the following two lemmas. 

\begin{lem}\label{l:ext-bi-lip}
	Let $A \subseteq \R^d$ be compact and $f : A \rightarrow H$ be an $L$-bi-Lipschitz function. Then there exists constant $m$ and $M$ (depending only on $d$ and $L$), and an extension $g : \R^d \rightarrow \tilde{H} \coloneqq H \times \R^m$ of $f$ which is $M$-bi-Lipschitz.
\end{lem}

\begin{rem}
	The above lemma is essentially \cite[Proposition 17.4]{david1991singular}. The proof is identical up to notational changes (swapping $\R^n$ for $H$). Additionally, the constant $M$ in \cite[Proposition 17.4]{david1991singular} is actually independent of $n$.  
\end{rem}

\begin{lem}\label{l:geometriclemma}
	Let $E \subseteq H$ and suppose there are constants $C$ and $\theta > 0$ so that for each $x \in E$ and $r \in (0,\diam E),$ there exists an Ahlfors $d$-regular set $\tilde{E}$ with regularity constant at most $C$, so that 
	\[ \mathscr{H}^d(E \cap B(x,r) \cap \tilde{E}) \geq \theta r^d, \]
	and so that $\tilde{E}$ satisfies the $(2,2)$-geometric lemma, meaning 
	\[ \sum_{\substack{R \in \mathscr{D}^{\tilde{E}} \\ R \subseteq Q}} \beta_{\tilde{E}}^{d,2}(C_0B_R)^2\ell(R)^d \lesssim \ell(Q)^d \] 
	for all $C_0 >1$ and $Q \in \mathscr{D}^{\tilde{E}}.$ Then $E$ also satisfies the $(2,2)$-geometric lemma. 	
\end{lem} 

\begin{rem}
	The above is a consequence of \cite[Theorem IV.1.3]{david1993analysis}, whose proof is found in \cite[Section IV.1.4]{david1993analysis} and remains valid for subsets of $H$. For the definition of the general $(p,q)$-geometric lemma, see \cite[Definition IV.1.2]{david1993analysis}. This is stated in terms of a Carleson measure condition, however, the above formulation is equivalent and is similar to how we formulated the BWGL, see after Definition \ref{d:BWGL}.  
\end{rem}

\begin{proof}[Proof of Proposition \ref{p:BWGL2}]
	Let $A \geq 1$ and $\ve >0$. We need to show $\BWGL(Q,A,\ve) < C(A,\ve)\ell(Q)^d$ for all $Q \in \mathscr{D}.$ Let $C_0 \geq 2\rho^{-1}.$ By Theorem \ref{Thm1}, with $p = 2,$ we know
	\[ \BWGL(Q,A,\ve) \lesssim_{A,\ve} \ell(Q)^d + \sum_{\substack{R \in \mathscr{D} \\ R \subseteq Q}} \beta^{d,2}_E(C_0B_R)^2\ell(R)^d \] 
	for all $Q \in \mathscr{D}.$ Thus, we are done if we can show that
	\begin{align}\label{e:letsshow0} 
		\sum_{\substack{R \in \mathscr{D} \\ R \subseteq Q}} \beta^{d,2}_E(C_0B_R)^2\ell(R)^d \lesssim \ell(Q)^d
	\end{align}
	for all $Q \in \mathscr{D}.$ This will follow from the fact that $E$ satisfies the $(2,2)$-geometric lemma, which we verify by Lemma \ref{l:geometriclemma}. First, by \cite[Corollary 1.2]{schul2009bi}, $E$ has \textit{big pieces of bi-Lipschitz images of} $\R^d,$ meaning there are constants $L,\theta>0$ such that for all $x \in E$ and $r \in (0,\diam E),$ there is a compact subset $A \subseteq \R^d$ and an $L$-bi-Lipschitz map $f:A \rightarrow H$ such that 
	\begin{align}\label{e:BPBLI}
		\mathscr{H}^d(E \cap B(x,r) \cap f(A)) \geq \theta r^d.
	\end{align} 
	Let $x \in E$, $r \in (0,\diam E)$, and let $A = A_{x,r}$ and $f = f_{x,r}$ be as above. By Lemma \ref{l:ext-bi-lip} (and identifying $\tilde{H}$ with $H$), we can extend $f$ to a bi-Lipschitz function $g: \R^d \rightarrow H$, and we set $\tilde{E} = g(\R^d).$ Clearly
	\[ \mathscr{H}^d(E \cap B(x,r) \cap \tilde{E}) \geq \theta r^d. \] 
	Additionally, $\tilde{E}$ satisfies the $(2,2)$-geometric lemma. To see this, notice that $\tilde{E}$ satisfies the BWGL by Lemma \ref{l:bi-lip-BWGL}, and so, by Theorem \ref{Thm1}, 
	\[ \sum_{\substack{Q \in \mathscr{D}^{\tilde{E}} \\ R \subseteq Q}} \beta_{\tilde{E}}^{d,2}(C_0B_R)^2\ell(R)^d \lesssim \mathscr{H}^d(Q) \lesssim \ell(Q)^d \]
	for all $C_0 >1$ and $Q \in\mathscr{D}^{\tilde{E}}.$ Equation \eqref{e:letsshow0} and hence the proposition now follows by Lemma \ref{l:geometriclemma}. 
\end{proof}

\appendix

\section{Parameterisation in Hilbert space}\label{a:DT}

In this section, we prove some of the main results in Theorem \ref{DT}. The main objective is to prove the local Lipschitz description of the surfaces $\Sigma_k$ -- Theorem \ref{DT} (7). We will prove this and all the results necessary preliminary results, which include Theorem \ref{DT} (10) and (11). In addition, we will prove Theorem \ref{DT}(13). The results Theorem \ref{DT} (2) and (4)-(6) are immediate from the definition in Section \ref{s:DT}. For the remaining points of Theorem \ref{DT} ((1),(3),(8),(9) and (12)), we direct the reader to the proofs contained in \cite{david2012reifenberg}, which are self-contained given the results we prove in this section. We shall on comment on this after we are finished in (7). 

\subsection*{Preliminaries} We collect a series of preliminary lemmas, to aid the proof of Theorem \ref{DT} (7). Included here is the proof of Theorem \ref{DT} (11).   

\begin{lem}\label{l:angAB}
Let $P,P'$ be $d$-planes in $H$, $x \in P$ and $B$ a ball centred on $x.$ Let $\ve >0$ and suppose 
\begin{align}\label{e:d_B1}
d_B(P,P') \leq \ve.
\end{align}
Let $P''$ be the plane parallel to $P'$ such that $x \in P''.$ Then for any $A >0$ we have 
\begin{align}\label{e:d_B2}
	d_{AB}(P,P'') \lesssim \ve
\end{align}
and
\begin{align}\label{e:d_B3} 
	d_{AB}(P,P') \lesssim \ve (1 + 1/A). 
\end{align}
\end{lem}

\begin{proof}
Let $A >0$ and $P''$ be as above. Since $P'$ and $P''$ are parallel, $x \in P''$ and $\dist(x,P') \leq \ve r_B$, we have 
\begin{align}\label{e:dP'P''} 
\dist(P',P'') \leq \ve r_B. 
\end{align}
So, for each $y \in P' \cap AB,$ we can find $z \in P'' \cap AB$ such that $|y-z| \lesssim \ve r_B.$ Similarly, for each $z \in P'' \cap AB$, we can find $y \in P' \cap AB$ such that $|y-z| \lesssim \ve r_B.$ This implies
\begin{align}
r_{AB}d_{AB}(P,P') \lesssim \ve r_B + r_{AB}d_{AB}(P,P'')
\end{align}
and so \eqref{e:d_B3} will follow once we show \eqref{e:d_B2}. Furthermore, since $x \in P \cap P''$ we have 
\[ d_{\frac{1}{2}B}(P,P'') = d_{AB}(P,P''),\] 
so it suffices to show 
\[ d_{\frac{1}{2}B}(P,P'') \lesssim \ve.\]
This will be the goal for the rest of the proof. Let $z \in P'' \cap \tfrac{1}{2}B.$ By \eqref{e:dP'P''} there exists $z' \in P'$ such that 
\begin{align}\label{e:zz'}
|z - z'| \leq \ve r_B.
\end{align}
It follows that $z' \in P' \cap (\tfrac{1}{2} + \ve)B$ which by \eqref{e:d_B1}. implies the existence of a point $y' \in P \cap B$ such that
\begin{align}\label{e:z'y'}
|z' - y'| \leq \ve r_B. 
\end{align}
By \eqref{e:z'y'}, since $z' \in (\tfrac{1}{2} + \ve)B,$ it actually follows that $y' \in P \cap (\tfrac{1}{2}+2\ve)B$. Since $B$ is centred on $x$ and $x \in P$ there exists a point $y \in P \cap \tfrac{1}{2}B$ such that
\begin{align}\label{e:y'y}
|y' - y| \leq 2\ve r_B.
\end{align}
Combing the above, we have
\begin{align}
\dist(z, P \cap \tfrac{1}{2}B) \leq |z - y| \leq |z - z'| + |z' - y'| + |y'-y| \leq 4\ve r_B.
\end{align}
Since $z$ was an arbitrary point in $P'' \cap \tfrac{1}{2}B$, this gives 
\[ \sup_{y \in P'' \cap \frac{1}{2}B} \dist(y, P \cap \tfrac{1}{2}B) \leq 4\ve r_B. \] 
Swapping the roles of $P$ and $P''$ we also get
\[ \sup_{y \in P \cap \frac{1}{2}B} \dist(y, P'' \cap \tfrac{1}{2}B) \leq 4\ve r_B \]
which finishes the proof. 
\end{proof}

\begin{lem}\label{l:proj}
Suppose $P,P'$ are two $d$-planes in $H$ which contain the origin and satisfy $d_{B(0,1)}(P,P') \leq \ve.$ Then for $y \in P'$, we have 
\begin{align}\label{e:proj}
|\pi^\perp_P(y)| \lesssim \ve |y|, 
\end{align}
where $\pi_P^\perp$ denotes the orthogonal projection on $P^\perp,$ the linear space orthogonal to $P$
\end{lem}

\begin{proof}
Let $B = B(0,1).$ Since $0 \in P \cap P'$ and $d_{B}(P,P') \leq \ve,$ we have $d_{AB}(P,P') \leq \ve$ for all $A >0.$ So, taking $A = |y|,$ we get
\begin{align}
|\pi^\perp_P(y)| = |y - \pi_P(y)| = \dist(y,P) \leq |y|d_{B(0,|y|)}(P,P') \leq \ve |y|.
\end{align}
\end{proof}

\begin{lem}\label{l:pcont}
Let $B$ be a ball and suppose $P,P'$ are $d$-planes in $X$. For each $\alpha >0$ there exists $\ve > 0$ such that if 
\[ d_B(P,P') \leq \ve \] 
then 
\[  P \cap (1-\alpha)B \subseteq \pi_P( P' \cap B). \] 
\end{lem}

The proof of Lemma \ref{l:pcont} is similar to the topological argument used to prove Lemma \ref{l:proj-covers}, we shall omit the details. The following is Lemma \ref{DT} (11).

\begin{lem}\label{l:sig-pi}
	$\emph{(11)}$ holds.
\end{lem}

\begin{proof}
	Let $y \in \Sigma_k \cap V_k^8$. By Lemma \ref{l:spu} this mean that 
	\[\sum_{j \in J_k} \theta_{j,k}(y) = 1.\]
	Then, if $i \in J_k$ is such that $y \in 10B_{i,k}$, we have 
	\begin{align}
		|\sigma_k(y) - \pi_{i,k}(y)| &= \big|y + \sum_{j \in J_k} \theta_{j,k}(y)[\pi_{j,k}(y)-y] - \pi_{i,k}(y) \big|   \\
		&=\big| \sum_{j \in J_k} \theta_{j,k}(y)[\pi_{j,k}(y) - \pi_{i,k}(y)] \big| \leq \sum_{j \in J_k} |\theta_{j,k}(y)| |\pi_{j,k}(y) - \pi_{i,k}(y)|. 
	\end{align}
	Since $\supp \theta_{j,k} \subseteq 10B_{j,k}$, if $j \in J_k$ is such that $\theta_{j,k}(y) \not=0$, then $y \in 10B_{j,k}$. In this case $|x_{j,k} - x_{i,k}| \leq 20r_k.$ This, along with the definition of $\ve_k(y)$, gives 
	\begin{align}\label{e:sig-pi1}
		d_{x_{i,k},100r_k}(P_{i,k},P_{j,k}) \leq \ve_k(y)
	\end{align}
	which implies
	\begin{align}\label{e:sig-pi2} 
		|\pi_{j,k}(y) - \pi_{i,k}(y)| \leq C\ve_k(y) r_k.
	\end{align}
	As shown in the proof of Lemma \ref{l:spu}, the balls $10B_{j,k}, \ j \in J_k$ have bounded overlap, hence
	\[ |\sigma_k(y) - \pi_{i,k}(y)| \leq \sum_{j \in J_k} |\theta_{j,k}(y)| C\ve_k(y) r_k \leq C \ve_k(y) r_k  \]
	as required. 
\end{proof}

\begin{lem}\label{l:sig-pi-lip}
	Let $k \geq 0$, $y,z \in \Sigma_k \cap V_k^8$ and $i \in J_k$ so that $y,z \in 10B_{i,k},$ then 
	\[ |(\sigma_k(y) - \sigma_k(z)) - (\pi_{i,k}(y) - \pi_{i,k}(z))| \lesssim \ve_k(y) |y-z|  \] 	
\end{lem}

\begin{proof}
	Using the fact that $y,z \in V_k^8$ and so $\sum_{j \in J_k} \theta_{j,k}(y) = 1$, with the analogous quantity holding for $z,$ we have 
	\begin{align}
		|(\sigma_{k}(y) - \sigma_k(z)) - (\pi_{i,k}(y) - \pi_{i,k}(z))| &\leq  \sum_{j \in J_k} |\theta_{j,k}(y)| | \pi_{j,k}(y-z) - \pi_{i,k}(y-z)|  \\
		&\hspace{4em} +  \sum_{j \in J_k} |\theta_{j,k}(y) - \theta_{j,k}(z)||\pi_{j,k}(z) - \pi_{i,k}(z)| \\
		&= S_1 + S_2.
	\end{align}
	We start by estimating $S_1.$ By \eqref{e:sig-pi1}, for each $j \in J_k$ such that $\theta_{j,k}(y) \not= 0,$ we have 
	\[d_{x_{i,k},100r_k}(P_{i,k},P_{j,k}) \leq \ve_k(y). \]
	Then, Lemma \ref{l:proj} implies 
	\begin{align}
		|\pi_{j,k}(y-z) - \pi_{i,k}(y-z)| &\leq |\pi_{j,k}(y-z) - \pi_{i,k}(\pi_{j,k}(y-z)) | + |\pi_{i,k}(\pi_{j,k}(y-z)) - \pi_{i,k}(y-z)| \\
		&=| \pi_{i,k}^\perp(\pi_{j,k}(y-z))| + | \pi_{i,k}(\pi_{j,k}^\perp(y-z))| \\
		&\lesssim \ve_k(y) |y-z|.
	\end{align}
	Since the ball $10B_{j,k}, \ j \in J_k,$ have bounded overlap, we conclude that $S_1 \lesssim \ve_k(y) |y-z|.$ We also see that $S_2 \lesssim \ve_k(y) |y-z|$ by the bounded overlap of the balls $10B_{j,k},$ \eqref{e:sig-pi2}, and Lemma \ref{l:tpu}, since 
	\begin{align}
		|\theta_{j,k}(y) - \theta_{j,k}(z)||\pi_{j,k}(z) - \pi_{i,k}(z)| \lesssim r_k^{-1}|y-z| \ve_k(z) r_k \lesssim \ve_k(z) |y-z|. 
	\end{align}

\end{proof}

\subsection*{Local Lipschitz description}

We can now begin the proof of Theorem \ref{DT}(7), which follows from the proposition below. This is the analogue of \cite[Proposition 5.1]{david2012reifenberg}.

\begin{prop}\label{p:DT}
There are constant $\{C_i\}_{i=1}^7$ such that the following holds for all $k \geq 0$ and $j \in J_k.$ First, there is a function $A_{j,k} : P_{j,k} \cap 49B_{j,k} \rightarrow P_{j,k}^\perp$ such that 
\begin{align}\label{e:Anorm}
|A_{j,k}(x_{j,k})| \leq C_1\ve r_k.,
\end{align}
\begin{align}\label{e:Alip}
|DA_{j,k}(z)| \leq C_2 \ve, \quad \text{for all} \ z \in P_{j,k} \cap 49B_{j,k} 
\end{align}
and if $\Gamma_{A_{j,k}}$ is the graph of $A_{j,k}$ over $P_{j,k}$ then 
\begin{align}\label{e:graphA}
\Sigma_{k} \cap D(x_{j,k},P_{j,k},49r_k) = \Gamma_{A_{j,k}} \cap D(x_{j,k},P_{j,k},49r_k).
\end{align}
Next, there is a function $F_{j,k} : P_{j,k} \cap 40B_{j,k} \rightarrow P_{j,k}^\perp$ such that
\begin{align}\label{e:Fnorm}
|F_{j,k}(x_{j,k})| \leq C_3\ve r_k,
\end{align}
\begin{align}\label{e:Flip}
|DF_{j,k}(z)| \leq C_4\ve  \quad \text{for} \ z \in P_{j,k} \cap 40B_{j,k},
\end{align}
\begin{align}\label{e:Flip1}
|DF_{j,k}(z)| \leq C_5\ve \quad \text{for} \ z \in P_{j,k} \cap 7B_{j,k},
\end{align}
and if $\Gamma_{F_{j,k}}$ is the graph of $F_{j,k}$ over $P_{j,k}$ then 
\begin{align}\label{e:graphF}
\Sigma_{k+1} \cap D(x_{j,k},P_{j,k},40r_k) = \Gamma_{F_{j,k}} \cap D(x_{j,k},P_{j,k},40r_k).
\end{align}
\end{prop}

\begin{rem}\label{r:char}
For a given plane $P$, a ball $B$, and a differentiable function $f:P \rightarrow H$, it is not difficult to show that $|Df(z)| \leq C$ for all $z \in P \cap B$ if and only if $|f(x) - f(y)| \leq C|x-y|$ for all $x,y \in P \cap B.$ In the following proof, we will make use of the second characterization. 
\end{rem}

We prove Proposition \ref{p:DT} by induction. Since $\Sigma_0 = P_0$ it is clear that \eqref{e:Anorm}-\eqref{e:graphA} hold for $k=0$, which establishes the base case. The induction will be carried out in two stages.

\subsection*{Stage 1} In this section, we show the following.
\begin{lem}\label{l:Stage1}
	If \eqref{e:Fnorm}-\eqref{e:graphF} hold for $k$, then \eqref{e:Anorm}-\eqref{e:graphA} hold for $k+1.$
\end{lem}
Let $j \in J_{k+1}.$ Our goal is to define a map $A_{j,k+1}$ so that \eqref{e:Anorm}-\eqref{e:graphA} hold. We need some preliminary results. First, by Definition \ref{d:CCBP} (1) there is $i \in J_k$ such that
\begin{align}\label{e:2B_ik}
 x_{j,k+1} \in 2B_{i,k}.
\end{align}
By translating we may assume $x_{i,k} = 0,$ so that $P_{i,k}$ is a linear subspace. To simplify notation we shall write $V = P_{i,k}$ and $W=P_{j,k+1}$. By Definition \ref{d:CCBP} (5), we have
\begin{align}\label{e:dVW}
d_{0,20r_k}(V,W) \leq \ve,
\end{align}
so we know that there exists $p \in W \cap B(0,20r_k)$ satisfying 
\begin{align}\label{e:p}
|p| \leq 20\ve r_k.
\end{align}
Let $W'$ denote the linear subspace parallel to $W$. By Lemma \ref{l:angAB} we have for any $A >0$, 
\begin{align}\label{e:dVW1}
d_{0,20Ar_k}(V,W') \lesssim \ve \quad \text{and} \quad d_{0,20Ar_k}(V,W) \lesssim \ve(1+1/A). 
\end{align}
Let $\pi$ denote the orthogonal projection onto $V$, $\pi_W$ denote the orthogonal projection on $W$, $\pi_W^\perp$ denote the orthogonal projection on $W^\perp$ (the linear subspace orthogonal to $W'$) and $\pi_{W'}$ denote the orthogonal projection onto $W'$. In this way, 
\begin{align}\label{e:defpiw}
\pi_{W}(x) = \pi_{W}^\perp(p) + \pi_{W'}(x).
\end{align}
Since $W^\perp$ and $W'$ are linear subspaces, and orthogonal projections on linear subspaces are bounded linear operator with norm at most 1, we have
\begin{align}\label{e:pi_W}
| \pi_{W}(y) | \leq |\pi_{W}^\perp(p)| + |\pi_{W'}(y)| \leq |p| + |y| \leq 20\ve r_k + |y|.
\end{align}
Let $F = F_{i,k}$ be the map satisfying \eqref{e:Fnorm}-\eqref{e:graphF} for $k$ and $i \in J_k$ and define a map $f: W \cap 39B_{i,k} \rightarrow W$ by setting
\[ f(y) = \pi_W ( \pi(y) + F(\pi(y))). \] 
See Figure \ref{f:Defnf}. We have the following.
\begin{figure}
  \centering
  \includegraphics[scale=0.8]{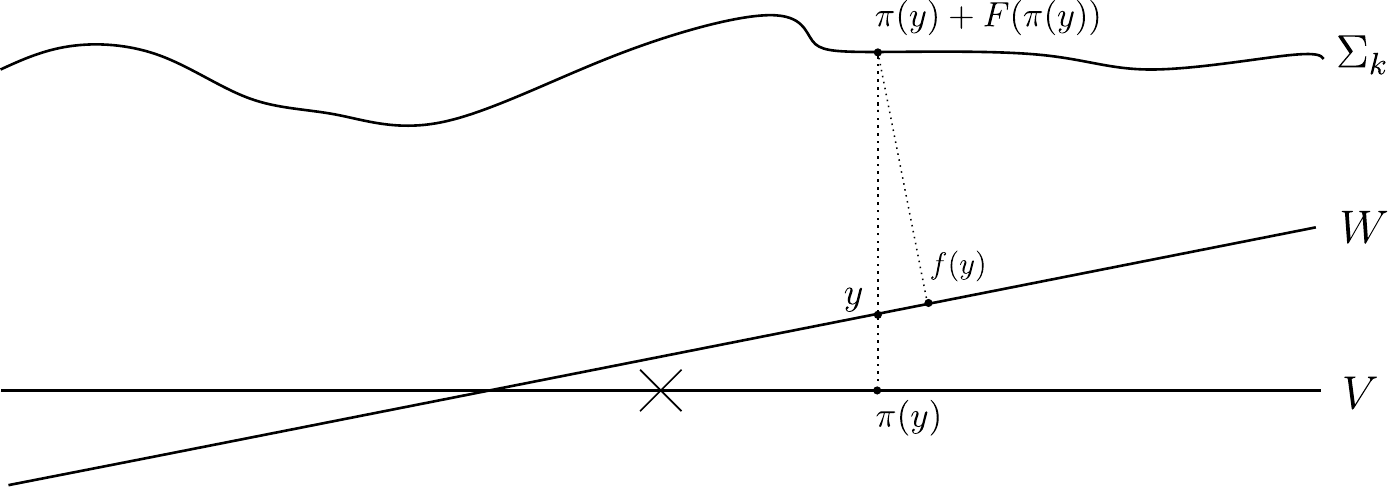}
\caption{Definition of $f$.}\label{f:Defnf}
\end{figure}
\bigbreak
\noindent
\begin{lem}
	 The map $f$ is well-defined and has a Lipschitz inverse $f^{-1}: W \supseteq U \rightarrow W \cap 39B_{i,k}$ such that
\begin{align}\label{e:fprop}
\emph{Lip}(f^{-1}) + ||f^{-1}|| \leq 3.
\end{align}
\end{lem}

\begin{proof}
We start by considering the action of $\pi$ on elements of $W \cap 39B_{i,k}$. Along the way, we show $f$ is well-defined. Let $y \in W \cap 39B_{i,k}$. Recall that since $x_{i,k} = 0$ we have $|y| \leq 39r_k.$ By \eqref{e:dVW1} (with $A = |y|/20r_k$) we have
\begin{align}\label{e:pi-y}
|\pi(y) - y| &= \dist(y,V) \leq |y|d_{B(0,|y|)}(V,W) \stackrel{\eqref{e:dVW1}}{\leq} C\ve|y|( 1 + 20r_k/|y|) \\
&= C\ve(|y| + 20r_k) \lesssim \ve r_k.
\end{align}
Hence, if $\ve$ is small enough $\pi(y) \in V \cap 40B_{i,k} = \text{Dom}(F)$ for all $y\in W \cap 39B_{i,k}$, and so $f$ is well-defined. Next, we establish some bi-Lipschitz estimates for $\pi$. Let $y,z \in W.$ Since $\pi$ is linear and $y-z \in W'$, by \eqref{e:dVW1} and Lemma \ref{l:proj}, we have
\begin{align}\label{e:pi-yz}
|(\pi(y) - \pi(z)) - (y-z)| = |\pi(y-z) - (y-z)| \leq C\ve |y-z|.
\end{align}
Now that we have information for $\pi,$ we can prove estimates similar to \eqref{e:pi-y} and \eqref{e:pi-yz} for $f$. Notice, if $y,z \in W \cap 39B_{i,k}$ then 
\begin{align}\label{e:fdist}
|f(y) -y| &= |\pi_W( \pi(y) - y + F(\pi(y)))| \\
&\leftstackrel{\eqref{e:pi_W}}{\leq} 20\ve r_k + |\pi(y) - y + F(\pi(y))| \\
&\leq 20\ve r_k +  |\pi(y) -y | + |F(\pi(y))| \stackrel{\eqref{e:pi-y} \atop \eqref{e:Fnorm}}{\leq} (C + C_3)\ve r_k
\end{align}
and
\begin{align}\label{bilipf}
|f(y) - y - (f(z) - z)| &= |\pi_{W'}( \pi(y) - \pi(z) - (y-z) +  F(\pi(y)) -  F(\pi(z)) )| \\
&\leq |\pi(y) - \pi(z) - (y-z)| + |F(\pi(y)) -  F(\pi(z))| \\
&\leftstackrel{\eqref{e:pi-yz} \atop \eqref{e:Flip}}{\leq} (C+ C_4) \ve |y-z|.
\end{align}
By \eqref{bilipf}, choosing $\ve$ small enough, we see that $f$ is bijection from $W \cap 39B_{i,k}$ onto $U = f(W \cap 39B_{i,k}) \subseteq W$. Thus, it has an inverse $f^{-1} : U \rightarrow W \cap 39B_{i,k}$. It is not difficult to show that \eqref{e:fdist} and \eqref{bilipf} imply \eqref{e:fprop} for $\ve$ small enough.
\end{proof}

Now, we can define the function $A: U \rightarrow W^\perp$ by setting 
\[ A(y) = \pi^\perp_W( \pi(f^{-1}(y)) + F( \pi(f^{-1}(y))) ) - \pi_W^\perp(p). \]
See Figure \ref{f:DefnA}.  
\begin{figure}
  \centering
  \includegraphics[scale=0.8]{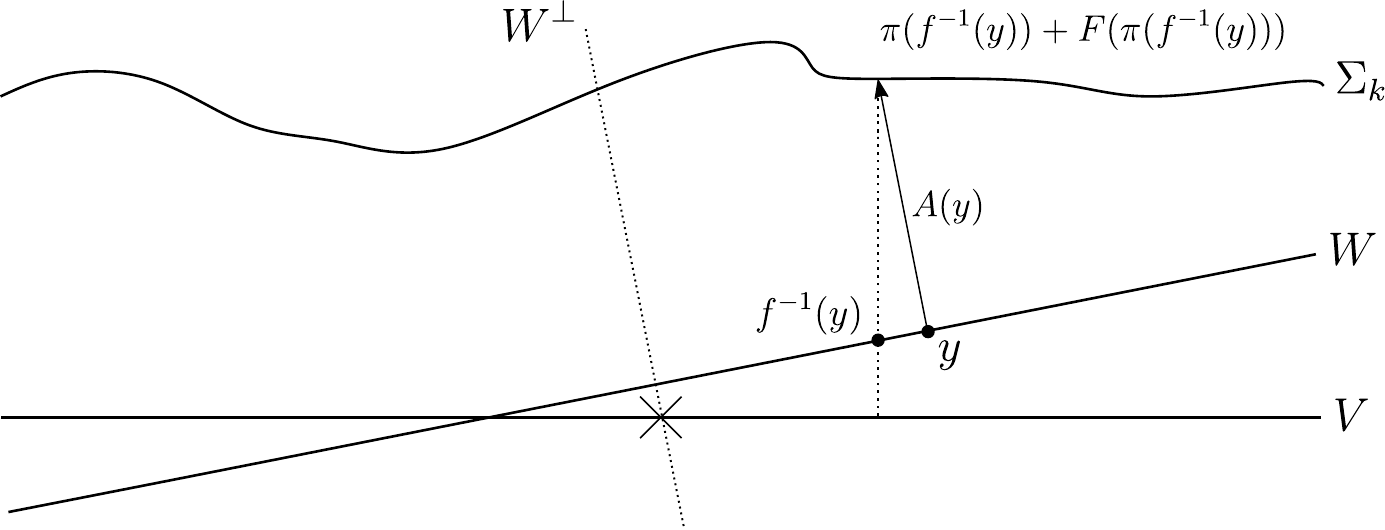}
\caption{Definition of $A$.}\label{f:DefnA}
\end{figure}
By the definition of $f$ (at the point $f^{-1}(y))$, we have 
\[y = \pi_W ( \pi(f^{-1}(y)) + F( \pi(f^{-1}(y)))).\]
Using this, and \eqref{e:defpiw}, for any $y \in U$ we have
\begin{align}\label{e:AF}
	\begin{split}
 y + A(y) &= \pi_W ( \pi(f^{-1}(y)) + F( \pi(f^{-1}(y)))) + \pi^\perp_W( \pi(f^{-1}(y)) + F( \pi(f^{-1}(y))) ) - \pi_W^\perp(p) \\
 &=  \pi_{W'} ( \pi(f^{-1}(y)) + F( \pi(f^{-1}(y)))) + \pi^\perp_W( \pi(f^{-1}(y)) + F( \pi(f^{-1}(y))) ) \\
 &=\pi(f^{-1}(y)) + F( \pi(f^{-1}(y))).
 \end{split}
\end{align}
Now that we have defined $A$, we will establish \eqref{e:Anorm}-\eqref{e:graphA}. We start with \eqref{e:graphA}, which is the following.
\begin{lem}\label{l:graphA}
	We have
\[ \Gamma_A \cap D(x_{j,k+1},W,49r_{k+1}) = \Sigma_{k+1} \cap D(x_{j,k+1},W,49r_{k+1}). \]
\end{lem}

Before we can show this, we need the following.
\begin{lem} We have 
\begin{align}\label{e:D}
D(x_{j,k+1},W,49r_{k+1}) \subseteq D(0,V,35r_k).
\end{align}
\end{lem}
\begin{proof}
Let $y \in D(x_{j,k+1},W,49r_{k+1}).$ We want to show $|\pi(y)| \leq 35r_k$ and $|\pi(y) -y|  \leq 35 r_k$.  We start by showing $|\pi(y) - \pi_W(y)| \lesssim \ve r_k$ as follows. First,
\begin{align*}
	|\pi(y) - \pi_W(y)| &\leq |\pi(y) - \pi_{W'}(y)| + |\pi_{W}^\perp(p)| \\
	&\leftstackrel{\eqref{e:p}}{\leq} | \pi(y) - \pi (\pi_{W'}(y))| + |\pi(\pi_{W'}(y)) - \pi_{W'}(y)| + 20\ve r_k \\
	&= |\pi(\pi_W^\perp(y))| + |\pi^\perp(\pi_{W'}(y))| + 20\ve r_k \\
	&\leq C \ve(|\pi_W^\perp(y)| + |\pi_W(y)| + r_k),
\end{align*}
where the last inequality follows from \eqref{e:dVW1} and Lemma \ref{l:proj}. Continuing from this, since $y,p \in W$ and $W^\perp$ is orthogonal to $W,$ we must have $\pi_W^\perp(y) = \pi_W^\perp(p).$ So,
\begin{align*}
	|\pi(y) - \pi_W(y)| &\leq C \ve(|\pi_W^\perp(p)| + |\pi_W(y)| + r_k) \\
	&\stackrel{\eqref{e:pi(y)}}{\leq} C\ve ( |p| + 49r_{k+1} + r_k ) \\
	&\leq C\ve r_k. 
\end{align*}
Our required bounds for $|\pi(y)|$ and $|\pi(y) - y|$ now follow for $\epsilon$ small enough since
\begin{align}\label{e:pi(y)}
\begin{split}
|\pi(y)| & \leq |\pi(y) - \pi_W(y)| + |\pi_W(y) - x_{j,k+1}| + |x_{j,k+1}| \\
&\leftstackrel{\eqref{e:2B_ik}}{\leq} |\pi(y) - \pi_W(y)| + 49r_{k+1} + 2r_k \\
&\leq  |\pi(y) - \pi_W(y)| + 7r_k \leq 8r_k \leq 35r_k
\end{split}
\end{align} 
and 
\begin{align}\label{e:pD2}
\begin{split}
|\pi(y) -y | &\leq |\pi(y) - \pi_W(y)| + |\pi_W(y) - y| \\
&\leq C\ve r_k + 49r_{k+1} \leq 5 r_k \leq 35r_k.
\end{split}
\end{align}

\end{proof}

\begin{proof}[Proof of Lemma \ref{l:graphA}] Since $\Gamma_A \subseteq \Sigma_{k+1}$ by \eqref{e:AF}, the forward inclusion is clear. Our goal now is to show the backward inclusion. Let $p \in \Sigma_{k+1} \cap D(x_{j,k+1},W,49r_{k+1}),$ we must show $p \in \Gamma_A.$ By \eqref{e:D} we have $p \in D(0,V,35r_k),$ which, since $\Sigma_{k+1}$ is the graph of $F$ over $D(0,V,38r_k)$, means we can find a point $y \in V \cap 35B_{i,k}$ such that 
\[ p = y + F(y). \] 
For $\ve$ small enough we have $V \cap 35B_{i,k} \subseteq \pi(W \cap 38B_{i,k})$ by Lemma \ref{l:pcont}, so we can find a point $z \in W \cap 38B_{i,k}$ such that $y = \pi(z).$ Then,
\[ p = \pi(z) + F(\pi(z)) . \] 
Let $w = f(z) \in U.$ Since $f$ is bijective we can write
\[ p = \pi(f^{-1}(w)) + F( \pi(f^{-1}(w))) \] 
and so \eqref{e:AF} implies
\[ p = w + A(w) \in \Gamma_A \] 
as required. 
\end{proof}

To finish the proof of Lemma \ref{l:Stage1} and hence Stage 1, we are left to show the following. 

\begin{lem}
	We have $|A(x_{j,k+1})| \leq C_1\ve r_k$ and $|DA(z)| \leq C_2 \ve$ for all $z \in W \cap 49B_{j,k+1}.$
\end{lem}
\begin{proof}
We start with by estimating $|A(x_{j,k+1})|.$ To simplify notation we will write $x = x_{j,k+1}.$ Notice, in the definition of $A$, we can replace $p$ with any point in $W.$ Choosing the point $f^{-1}(x) \in W,$ we have
\begin{align*}
|A(x)| &= |  \pi^\perp_W( \pi(f^{-1}(x)) + F( \pi(f^{-1}(x))) ) - \pi_W^\perp(f^{-1}(x)) | \\
&=|\pi^\perp_W( \pi(f^{-1}(x)) - f^{-1}(x) + F( \pi(f^{-1}(x)))  )| \\
&\leq |\pi(f^{-1}(x)) - f^{-1}(x)| + |F( \pi(f^{-1}(x)))| \\
&\leftstackrel{\eqref{e:pi-y} \atop \eqref{e:Fnorm}}{\leq} (C + C_3)\ve |f^{-1}(x)| \leq C(1 +C_3) \ve |x| \\
&\leftstackrel{\eqref{e:2B_ik}}{\leq} C(1+C_3)\ve r_{k+1}.
\end{align*}
The required estimate holds as long as we choose
\[ C(1+C_3) \leq C_1. \]
Now we estimate $|DA(z)|,$ recall Remark \ref{r:char}. Let $y,z \in W \cap 49B_{j,k+1}.$ By \eqref{e:fdist}, we know that $f^{-1}(y),f^{-1}(z) \in W \cap 49.5B_{j,k+1}$, which by \eqref{e:pi-y}, implies
\[\pi(f^{-1}(y)),\pi(f^{-1}(z)) \in V \cap 50B_{j,k+1} \stackrel{\eqref{e:2B_ik}}{\subseteq} V \cap 7B_{i,k}.\]
Then
\begin{align*}
|A(y) - A(z)| &\leq |\pi( f^{-1}(y) - f^{-1}(z))| + |F( \pi(  f^{-1}(y)) - F( \pi(f^{-1}(z)))| \\
&\leftstackrel{\eqref{e:Flip1}}{\leq} (1 + C_5) |\pi(f^{-1}(y)) - \pi(f^{-1}(z))| \\
&\leq C(1 + C_5)\ve |f^{-1}(y)-f^{-1}(z)| \\
&\leq C(1+C_5)\ve |y-z|,
\end{align*}
where the third inequality follows from \eqref{e:dVW1} since $f^{-1}(y) - f^{-1}(z) \in W'.$ So, if
\[ C(1+C_5) \leq C_2 \] 
then \eqref{e:Anorm} holds.
\end{proof}

\subsection*{Stage 2} In this section we show the following. 

\begin{lem} 
	If \eqref{e:Anorm}-\eqref{e:graphA} hold for $k$ then \eqref{e:Fnorm}-\eqref{e:graphF} also hold for $k$.
\end{lem}

Let $j \in J_k$. Our ultimate goal is to define a map $F = F_{j,k}$ so that \eqref{e:Fnorm}-\eqref{e:graphF} hold. We start with some preliminaries. By translation, we may assume $x_{j,k} = 0.$ Let $P_{j,k},B_{j,k}$ and $A_{j,k}$ satisfy \eqref{e:Anorm}-\eqref{e:graphA}. For ease of notation we denote 
\[P = P_{j,k},  \ B = B_{j,k} \ \text{and} \  A =A_{j,k}.\]
Also, we define $\pi$ to be the orthogonal projection onto the plane $P$. Before we can define $F$, we need to prove some preliminary estimates and define an auxiliary function. The first of these estimates is Theorem \ref{DT} (10).  

\begin{lem} 
	Let $y \in \Sigma_k.$ Then
	\begin{align}\label{e:sig-y}
		|\sigma_k(y) - y| \leq C(C_1 + 1) \ve r_k. 
	\end{align}
\end{lem}
\begin{proof}
	If $y \not\in V_k^{10}$ then by definition $\sigma_k(y) = y$ and \eqref{e:sig-y} is trivial. Assume instead that $y \in V_k^{10}.$ Let $j \in J_k$ such that $y \in 10B_{j,k}$ Let $A_{j,k}$ be the map which satisfies \eqref{e:Anorm}-\eqref{e:graphA}. If $i \in J_k$ also satisfies $y \in 10B_{i,k}$ then by the Definition \ref{d:CCBP} (3) we have 
	\[ d_{x_{j,k},100r_k}(P_{j,k},P_{i,k}) \leq \ve \]
	and so
	\begin{align}\label{e:before}
		|\pi_{i,k}(y) - y| &= \dist(y,P_{i,k})  \leq  | \pi_{j,k}(y)- y|  + \dist( \pi_{j,k}(y),P_{i,k}) \\
		&\leq  |A_{j,k}(\pi_{j,k}(y))| + 100\ve r_k \stackrel{\eqref{e:Anorm}}{\leq} (C_1 + C )\ve r_k.
	\end{align}
	Since the functions $\theta_{j,k}$ are supported on $10B_{j,k}$, which have bounded overlap (by the arguments in the proof of Lemma \ref{l:spu}), we can use this and the above to get
	\begin{align}
		|\sigma_k(y) -y| \leq \sum_{i \in J_k} |\theta_{i,k}(y)||\pi_{i,k}(y) - y| \leq C(C_1 +1) \ve r_k, 
	\end{align} 
	which finishes the proof. 
\end{proof}

\begin{lem}
	Let $y,z \in \Sigma_k \cap 45B.$ Then,
\begin{align}\label{e:siglip}
| (\sigma_k(y) - y) - (\sigma_k(z) - z)| \leq C(C_1 +C_2 +1)\ve|y-z|.
\end{align}
\end{lem}

\begin{proof}

First, we write
\begin{align}
|(\sigma_k(y)-y) - (\sigma_k(z) -z)| &\leq \sum_{i \in J_k}|\theta_{i,k}(y) - \theta_{i,k}(z)||\pi_{i,k}(y) - y| \\
&\quad+ \sum_{i \in J_k} |\theta_{i,k}(y)| |\pi_{i,k}^\perp(y - z)| \\ 
&=S_1 + S_2.
\end{align}
Since $y,z \in 45B,$ if for some $i \in J_k$ we have either $\theta_{i,k}(y) \not=0$ or $\theta_{i,k}(z) \not=0$ then $|x_{i,k}| \leq 100r_k.$ By the Definition \ref{d:CCBP} (3),  we get
\begin{align}\label{e:PPik}
d_{0,100r_k}(P,P_{i,k}) \leq \ve.
\end{align}
By exactly the same estimate as in \eqref{e:before}, this implies
\begin{align}
|\pi_{i,k}(y) - y|  \leq (C_1 + C)\ve r_k. 
\end{align}
Then, since the balls $10B_{i,k}$ have bounded overlap and $|\theta_{i,k}(y) - \theta_{i,k}(z)| \lesssim r_k^{-1}$ by Lemma \ref{l:tpu}(3), it follows that 
\begin{align}
S_1 \leq C(C_1 + 1)\ve|y-z|.
\end{align}
We now control $S_2.$ Recall that $\Sigma_k \cap 45B$ is contained in the graph of $A$ over $P \cap 49B$. In particular, $y = \pi(y) + A ( \pi(y) ),$ with a similar expression holding for $z$. Again, if $\theta_{i,k}(y)\not=0$ then \eqref{e:PPik} holds, in particular, this implies $d_{0,100r_k}(P,P_{i,k}') \lesssim \ve$ by Lemma \ref{l:angAB}, where $P_{i,k}'$ is the linear space orthogonal to $P_{i,k}$. We use these facts, with Lemma \ref{l:proj} and the fact that the balls $10B_{i,k}$ have bounded overlap, to control $S_2$ as follows, 
\begin{align}
S_2 &=  \sum_{i \in J_k} | \theta_{i,k}(y) |\left| \pi_{i,k}^\perp\left[ \pi(y) + A(\pi(y)) - (\pi(z) + A(\pi(z)) \right] \right| \\
&\leq \sum_{i \in J_k} |\theta_{i,k}(y)|\left( | \pi_{i,k}^\perp (\pi(y)  - \pi(z) ) | + \left| A(\pi(y)) -  A(\pi(z))\right| \right) \\
&\leftstackrel{\substack{\eqref{e:Anorm} \\ \eqref{e:proj}}}{\leq} \sum_{i \in J_k} |\theta_{i,k}(y)| ( C+ C_2 )\ve |\pi(y)  - \pi(z)| \\
&\leq C(C_2+1)\ve|y-z|.
\end{align}
This completes the proof. 

\end{proof}

We are almost ready to define $F_{j,k}.$ Let us start with an auxiliary function. For $x \in P \cap 40B,$ write 
\[ h(x) = x+ \pi( \sigma_k(x+A(x)) - (x + A(x))). \] 
See Figure \ref{f:Defnh}.
\begin{figure}
  \centering
  \includegraphics[scale=0.8]{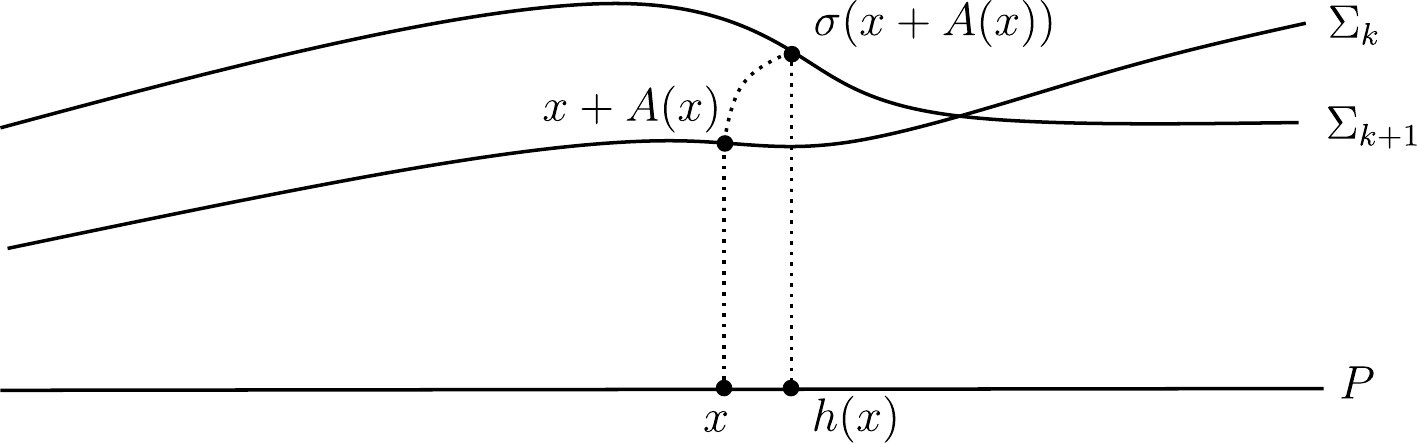}
\caption{Definition of $h(x)$.}\label{f:Defnh}
\end{figure}
We have 
\begin{align}\label{e:h-y}
|h(y) - y| \leq |\sigma_k(y+A(y)) - (y +A(y))| \stackrel{\eqref{e:sig-y}}{\leq} C\ve r_k
\end{align}
and
\begin{align}\label{e:liph}
|(h(y) - h(z)) - (y-z)| &\leq |\sigma_k(y+A(y)) -\sigma_k(z+A(z)) - \left( y+A(y) - (z+A(z)) \right)| \\
&\stackrel{\eqref{e:siglip}}{\leq} C \ve |y +A(y) - (z+A(z))| \\
&\leq C\ve|y-z|.
\end{align}
As in Stage 1 this implies that $h : P \cap 41B \rightarrow U$ is bijective, where $U = h(P \cap 41B).$ Denote its inverse by $h^{-1}: U \rightarrow P \cap 41B.$ By \eqref{e:h-y} and \eqref{e:liph}, $h^{-1}$ satisfies
\begin{align}\label{e:h}
|h^{-1}(x) - x| \lesssim \ve r_k, \quad \text{Lip}(h^{-1}) \leq 2. 
\end{align}
We don't include the details here, but a degree theory argument similar to that in the proof of Lemma \ref{l:bi-lip-BWGL} (see also \cite{david2012reifenberg}) implies 
\[U \supset P \cap 40B. \] 
Define
\[ F(y) = \pi^\perp ( \sigma_k ( h^{-1}(y) + A(h^{-1}(y)))). \] 
See Figure \ref{f:DefnF1}. We will now verify \eqref{e:Fnorm}-\eqref{e:graphF} for $F$. Let us start with \eqref{e:graphF}, which is the following.
\begin{figure}
  \centering
  \includegraphics[scale=0.7]{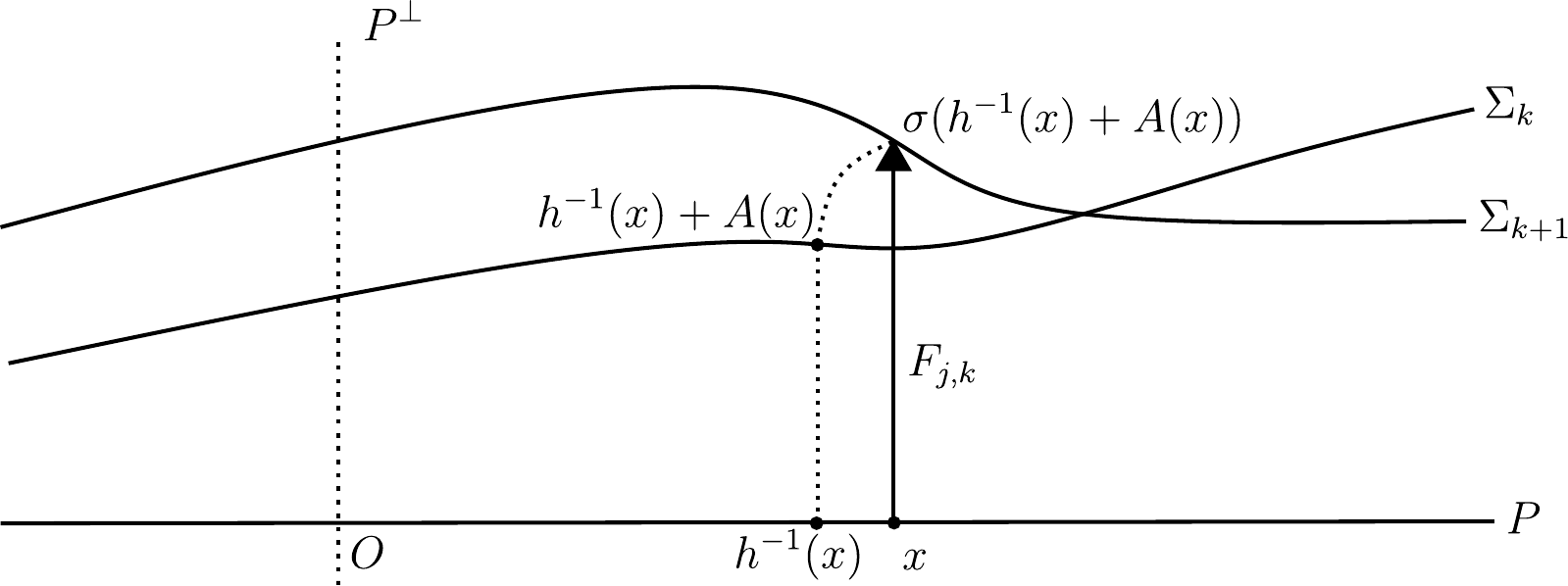}
\caption{Definition of $F_{j,k}$.}\label{f:DefnF1}
\end{figure}

\begin{lem}
	We have
	\[\Sigma_{k+1} \cap D(0,P,40B) \subseteq \Gamma_F \cap D(0,P,40B). \] 
\end{lem}

\begin{proof}
Since $\Gamma_F \subseteq \Sigma_{k+1}$ by definition, the only non-trivial direction is to show
\[ \Sigma_{k+1} \cap D(0,P,40B) \subseteq \Gamma_F \cap D(0,P,40B). \] 
Let $p \in \Sigma_{k+1} \cap D(0,P,40B)$ and let $z \in \Sigma_k$ such that $p = \sigma_{k}(z).$ Since $|z - p| \lesssim \ve r_k$ by \eqref{e:sig-y}, we see that $\pi(z) \in P \cap 41B.$ This implies that $z$ is contained in the graph of the function $A$ over $P \cap D(0,P,49B)$, hence 
\[ p = \sigma_k(z) = \sigma_k( \pi(z) + A(\pi(z))). \] 
Since $h^{-1} : U \rightarrow 41B$ is bijective there exists $y \in U$ such that $h^{-1}(y) = \pi(z).$ This gives 
\begin{align} p = \sigma_k( h^{-1}(y) + A(h^{-1}(y))) &= \pi( \sigma_k( h^{-1}(y) + A(h^{-1}(y)))  ) \\
&\hspace{4em} + \pi^\perp ( \sigma_k( h^{-1}(y) + A(h^{-1}(y)))  ) \\
&= y + \pi^\perp( \sigma_k( h^{-1}(y) + A(h^{-1}(y)))  \\
& = y + F(y) ,
\end{align}
where the penultimate equality follows from the definition of $h$. This finishes the proof. 
\end{proof}

We now focus on proving \eqref{e:Fnorm}-\eqref{e:Flip1} for $F.$ Let us start with \eqref{e:Flip}. 
\begin{lem}
	For $y,z \in P \cap 40B$ we have $|F(y) - F(z)| \leq C_4 \ve |y-z|.$
\end{lem}
\begin{proof}
Let $y,z \in P \cap 40B.$ To simplify notation we will write
\[ \tilde{y} = h^{-1}(y) + A(h^{-1}(y)), \]
with $\tilde{z}$ defined analogously. Then 
\begin{align*}
|F(y) - F(z)| &= |\pi^\perp(\sigma_k(\tilde{y})) -\pi^\perp(\sigma_k(\tilde{z}))| \\
&\leq|\pi^\perp ( (\sigma_k(\tilde{y}) - \tilde{y}) - (\sigma_k(\tilde{z}) - \tilde{z}) ) + |\pi^\perp ( \tilde{y} - \tilde{z})| \\
&=|\pi^\perp ( (\sigma_k(\tilde{y}) - \tilde{y}) - (\sigma_k(\tilde{z}) - \tilde{z}) ) +| A(h^{-1}(y)) - A(h^{-1}(z)| \\
&\leftstackrel{\substack{\eqref{e:Alip} \\ \eqref{e:siglip}}}{\leq} C(C_1 + C_2 +1) \ve | \tilde{y} - \tilde{z}| + C_2 \ve|y-z| \\
&\leq C(C_1 + C_2 +1)\ve |y-z|. 
\end{align*} 
If $C_4$ is chosen so that 
\[ C_4 \geq C(C_1 + C_2 +1) \]
then \eqref{e:Flip} follows.
\end{proof}

Finally, we prove \eqref{e:Fnorm} and \eqref{e:Flip1}. Both of the constants $C_3$ and $C_5$ will be chosen independent of all other constants and this closes the induction.
\begin{lem}
We have $|F(x_{j,k})| \leq C_3 \ve r_k$ and $|F(y) - F(z)| \leq C_5 \ve |y-z|$ for all $y,z \in P \cap 7B.$	
\end{lem}

\begin{proof}
Let us start with the first statement, it will be useful to prove the same bound for any point in $P \cap 7B.$ Let $y \in P \cap 7B$ and let $w = h^{-1}(y) +A(h^{-1}(y)).$ By \eqref{e:sig-y} and \eqref{e:h} we have
\begin{align}\label{e:above}
|\sigma_k(w) - y| &\leq |\sigma_k(w) - w| + |w -y| \\
&\leq C\ve r_k + |h^{-1}(y) - y| + |A(h^{-1}(y))| \\
&\leq C(1+C_1) \ve r_k
\end{align}
so that $\sigma_k(w) \in 8B$ for $\ve$ small enough. By Lemma \ref{l:sig-pi}, this implies
\begin{align}
|F(y)| = \dist(\sigma_k(w),P) \leq  | \sigma_k(w) - \pi(w) | \leq C \ve r_k.
\end{align}
Taking $y = x_{j,k}$ proves \eqref{e:Fnorm} if we take $C_3$ large enough. \\

We prove the second statement. Let $y,z \in 7B$ and let $w_y = h^{-1}(y) + A(h^{-1}(y))$ and $w_z = h^{-1}(z) + A(h^{-1}(z)).$ Examining \eqref{e:above}, we see $w_y,w_z \in 8B$ for $\ve$ small enough. By Lemma \ref{l:sig-pi-lip}, \eqref{e:liph}, and \eqref{e:Alip}, we have 
\begin{align*}
|F(y) - F(z)| &= |\pi^\perp ( \sigma_k(w_y) - \sigma_k(w_z) ) | \\
&=|\pi^\perp ( \sigma_k(w_y) - \sigma_k(w_z) - \pi(w_y) - \pi(w_z) ) | \\
&\leq C\ve |w_y - w_z| \\
&\leq C\ve ( |h^{-1}(y) - h^{-1}(z)| + |A(y) - A(z)| )\\
&\leq C\ve ( |y-z| + C_2\ve|y-z| ) \\
&\leq C \ve |y-z|.
\end{align*} 
This completes the proof of \eqref{e:Flip} if we take $C_5$ large enough.
\end{proof}

Finally, we show Theorem \ref{DT}(13), which is the following.

\subsection*{Remaining results} First, let us proof Theorem \ref{DT}(13), which is the following.  
\begin{lem}\label{al:lowerreg}
	Let $\Sigma$ be the surface produced by Theorem \ref{DT}. Then, for $x \in \Sigma$ and $r > 0$, 
	\[ \mathscr{H}^d(\Sigma \cap B(x,r)) \gtrsim r^d. \] 
\end{lem}

\begin{proof}
	Consider the map $A:H \rightarrow H$ given by $A(y) =\tfrac{y-x}{r},$ and the set $\tilde{\Sigma} = A(\Sigma).$ Then, $0 \in \tilde{\Sigma}$, and since $\Sigma$ is $C\ve$-Reifenberg flat, the same is true of $\tilde{\Sigma}$. Let $P$ be a $d$-plane through $0$ so that 
	\[ d_{B(0,1)}(\tilde{\Sigma},P) \leq  C\ve. \] 
	By \cite[Proposition 2.9]{edelen2018effective}, for $\ve > 0$ small enough, there exists a bi-H\"older map $g: P \rightarrow H$ which is the identity outside of $B(0,3/2)$, satisfies $|g(y) -y| \lesssim \ve$ for $y \in P$, and so that $\tilde{\Sigma} \cap B(0,1) = g(P) \cap B(0,1)$. Let $\pi$ denote the orthogonal projection onto $P$, and define $h = \pi \circ g.$ By definition, $h$ is continuous and is the identity outside $B(0,3/2),$ so, it is surjective. Additionally, for $y \in P,$
	\[ |h(y) -y| = | \pi ( g(y) ) - \pi(y) | \leq |g(y) - y| \lesssim \ve. \]
	Let $z \in P \cap B(0,1/2)$. Since $h$ is surjective, we can find $w \in P$ so that $z = h(w).$ For $\ve$ small enough, it follows from the inequality above that $w \in B(0,3/4)$ and so
	\[ |g(w)| \leq |g(w) - w| + |w| \leq C\ve +\frac{3}{4} \leq 1.  \] 
	In particular, $g(w) \in \tilde{\Sigma} \cap B(0,1)$ and $z = \pi(g(w)).$ We have just show that $P \cap B(0,1/2) \subseteq \pi(\tilde{\Sigma} \cap B(0,1)).$ This gives,
	\[ \mathscr{H}^d(\tilde{\Sigma} \cap B(0,1)) \geq  \mathscr{H}^d(\pi(\tilde{\Sigma} \cap B(0,1))) \geq \mathscr{H}^d(P \cap B(0,\tfrac{1}{2}) )\gtrsim 1. \]
	The lemma follow from this since
	\[ \mathscr{H}^d(\Sigma \cap B(x,r)) = \mathscr{H}^d( A^{-1}(\tilde{\Sigma} \cap B(0,1))) = r^d \mathscr{H}^d(\tilde{\Sigma} \cap B(0,1)) \gtrsim r^d. \]
\end{proof}

As mentioned, we direct the reader to the proofs of the remaining results in Theorem \ref{DT}. For (1),(3),(8) and (9), see \cite[Equation 6.2, Proposition 8.1, Lemma 6.2 and Lemma 6.3]{david2012reifenberg}, respectively. For (12), see \cite[Lemma 7.1]{david2012reifenberg}. This proof relies on \cite[Equation 5.12]{david2012reifenberg} (which we do not prove), but the same estimate can be obtained from \eqref{e:siglip}. 

\section{Constants}\label{a:constants}
\begin{lem}\label{l:reduce-constants1a} 
	Suppose that for any $\tilde{C}_0 > 1$ and $\tilde{A} > \max\{6\tilde{C}_0,10^5\}$ there exists $\ve > 0$ so that \eqref{e:dTSP1'} holds with constant $\tilde{C}_0,\tilde{A}$ and $\tilde{\ve}.$ Then, for any $C_0 > 1$ and $A > 10^5$ there exists $\ve > 0$ so that \eqref{e:dTSP1'} holds with constant $C_0,A$ and $\ve.$
\end{lem}

\begin{proof}
	Let $C_0 > 1$, $A > 10^5$, and $E \subseteq H$ be $(c,d)$-lower content $d$-regular. If $A > 6C_0$ then there is nothing to do, so let us assume $A \leq 6C_0.$ Choose $\kappa \sim_{A,C_0} 1$ so that $A > 6 \kappa C_0 > A/2$ and choose $K = K(A,C_0)$ to be the smallest integer so that $C_0 \rho^K + 1 \leq \kappa C_0.$ Choose $\ve > 0$ so that \eqref{e:dTSP1'} holds with constants $A,\kappa C_0$ and $\ve.$ First, we separate 
	\begin{align}
		\sum_{\substack{Q \in \mathscr{D} \\ Q \subseteq Q_0}} \beta^{d,p}_E(C_0B_Q)^2\ell(Q)^d &= \sum_{k=0}^{K-1} \sum_{\substack{Q \in \mathscr{D}_k \\ Q \subseteq Q_0}} \beta^{d,p}_E(C_0B_Q)^2\ell(Q)^d + \sum_{k=K}^\infty \sum_{\substack{Q \in \mathscr{D}_k \\ Q \subseteq Q_0 \\ Q \in \BWGL(A,\ve)}}  \beta^{d,p}_E(C_0B_Q)^2\ell(Q)^d \\
		&\hspace{4em} + \sum_{k=K}^\infty \sum_{\substack{Q \in \mathscr{D}_k \\ Q \subseteq Q_0 \\ Q \not\in \BWGL(A,\ve)}}  \beta^{d,p}_E(C_0B_Q)^2\ell(Q)^d \\
		&=I_1 + I_2 + I_3. 
	\end{align} 
	The term $I_2$ is bounded by $\BWGL(Q_0,A,\ve)$ be definition. Additionally, since $E$ is lower regular and the balls $c_0B_Q, \ Q \in \mathscr{D}_k$, are disjoint, we have 
	\[ I_1 \lesssim \sum_{k=0}^{K-1} \sum_{\substack{Q \in \mathscr{D}_k \\ Q \subseteq Q_0}} \mathscr{H}^d(E \cap c_0B_Q) \leq K \mathscr{H}^d(Q_0). \] 
	We are left to bound $I_3.$ Let $\mathscr{C}$ to be the collection of cubes $Q \in \bigcup_{k=K}^\infty \mathscr{D}_k$ so that $Q \subseteq Q_0$ and $Q \not\in \BWGL(A,\ve).$ Let $\mathscr{C}_1$ be the set of cubes $Q \in \mathscr{C}$ so that $Q^{(K)} \not\in \BWGL(A,\ve)$.  Also, set $\mathscr{C}_2$ to be the set of cubes $Q \in \mathscr{C}$ so that $Q^{(K)} \in \BWGL(A,\ve)$. Notice that $Q^{(K)} \subseteq Q_0$ in either case.  Then 
	\[ I_3 = \sum_{Q \in \mathscr{C}_1} \beta^{d,p}_E(C_0B_Q)^2\ell(Q)^d + \sum_{Q \in \mathscr{C}_2} \beta^{d,p}_E(C_0B_Q)^2\ell(Q)^d = I_{3,1} + I_{3,2}.  \] 
	We start by estimating $I_{3,1}.$ For $Q \in \mathscr{C}_1,$ set $Q^* = Q^{(K)}.$ We begin by showing that the map $Q \mapsto Q^*$ is at most $C$-to-$1$ for some universal constant $C.$ Indeed, let $R \in \mathscr{D}$ and consider $\mathscr{C}_1(R) = \{Q \in \mathscr{C}_1 : Q^* = R\}.$ Since $R \not\in \BWGL(A,\ve),$ there exists a $d$-plane $P$ so that $d_{AB_R}(E,P) \leq \ve.$ So, for $\ve$ small enough, $\dist(x_Q,P) \leq \ve A \ell(R) \leq \tfrac{c_0}{2} \ell(Q)$ and $c_0B_Q \subseteq B_{R} = B(x_{R},\rho^{-K}\ell(Q))$ for all $Q \in \mathscr{C}_1(R).$ By Lemma \ref{ENV}, this implies $\# \mathscr{C}_1(R) \lesssim 1$ as required. Now, by our choice of $K$, $C_0B_Q \subseteq \kappa C_0 B_{Q^*}$. So, by our assumption in the statement of Lemma \ref{l:reduce-constants1a}, Lemma \ref{lemma:monotonicity}, and the above, we have 
	\begin{align}
		I_{3,1} \lesssim \sum_{Q \in \mathscr{C}_1 } \beta_{E}^{d,p}(\kappa C_0B_{Q*})^2 \ell(Q^*)^d \lesssim \sum_{\substack{Q \in \mathscr{D} \\ Q \subseteq Q_0}}  \beta^{d,p}_E(\kappa C_0B_Q)^2\ell(Q)^d \lesssim \mathscr{H}^d(Q_0) + \BWGL(Q_0,A,\ve).
	\end{align}
	Now for $I_{3,2}$. For $Q \in \mathscr{C}_2$, define $Q^*$ to be the smallest cube so that $Q \subseteq Q^*$ and $Q^* \in \BWGL(A,\ve).$ Such a cube exists since $Q^{(K)} \in \BWGL(A,\ve)$ by definition. Again, we show the map $Q \mapsto Q^*$ is at most $C$-to-$1$ for some universal constant $C.$ Let $R \in \mathscr{D}$ for which the set $\mathscr{C}_2(R) = \{Q \in \mathscr{C}_2 : Q^* = R \}$ is non-empty. Note, there are at most $K$ scales for which $\#\mathscr{C}_2(R) \cap \mathscr{D}_k \not=\emptyset.$ Choose $k$ to be one of these scales and let $Q \in \mathscr{C}_2(R) \cap \mathscr{D}_k.$ Since $R$ is the smallest cube containing $Q$ so that $R \in \BWGL(A,\ve),$ there exists a child $R'$ of $R$ so that $R' \not\in \BWGL(A,\ve).$ As above, since $B_R \subseteq AB_{R'}$ (because $A > 10^5$), we can find a $d$-plane $P$ so that $\dist(x_Q,P) \leq \frac{c_0}{2}\ell(Q)$ and $c_0B_Q \subseteq B_R \subseteq B(x_R, \rho^{-K}\ell(Q)).$ Then, by Lemma \ref{ENV}, $\#\mathscr{C}_2(R) \cap \mathscr{D}_k \lesssim 1$, and so $\mathscr{C}_2(R) \lesssim 1.$ Hence, 
	\begin{align}
		I_{3,2} \lesssim \sum_{Q \in \mathscr{C}_2} \ell(Q^*)^d \lesssim \sum_{\substack{  Q \subseteq Q_0 \\ Q \in \BWGL(A,\ve)}} \ell(Q)^d = \BWGL(Q_0,A,\ve). 
	\end{align} 
	This finishes the proof. 
\end{proof}

\begin{lem}
	Let $\ve > 0$, $A \geq 1$ and $M \geq C_0$, and suppose
	\begin{align}\label{e:Section6a} \mathscr{H}^d(Q_0) + \BWGL(Q_0,A,\ve) \lesssim \ell(Q_0)^d + \sum_{\substack{Q \in \mathscr{D} \\ Q \subseteq Q_0}} \beta_E^{d,1}(MB_Q)^2\ell(Q)^d.
	\end{align}
	Then, \eqref{e:Section6'} holds with constant $\ve,A$ and $C_0$.  
\end{lem}

\begin{proof}
	The proof will be similar to the proof of Lemma \ref{l:reduce-constants1a}. First, by Lemma \ref{c:beta}, 
	\begin{align}\label{e:consts1}
		\sum_{\substack{Q \in \mathscr{D} \\ Q \subseteq Q_0}} \beta_E^{d,1}(MB_Q)^2\ell(Q)^d \lesssim \sum_{\substack{Q \in \mathscr{D} \\ Q \subseteq Q_0}} \beta_E^{d,p}(MB_Q)^2\ell(Q)^d.
	\end{align}
	The remainder of the proof will be focussed controlling the right-hand side of \eqref{e:consts1} in terms of the same quantity but with $C_0$ in place of $M$. Let $K = K(C_0,M)$ be the smallest integer so that $1+M\rho^K \leq C_0$ and let $\delta > 0$ be a constant to be chosen small. Let $\mathscr{C}$ be the set of cubes $Q \in \mathscr{D}$ so that $Q \subseteq Q_0$ and $\beta_E^{d,p}(C_0B_Q) \leq \delta.$ Notice that 
	\begin{align}
		\sum_{\substack{Q \in \mathscr{D} \setminus \mathscr{C}\\ Q \subseteq Q_0}}  \beta_E^{d,p}(MB_Q)^2\ell(Q)^d \lesssim_\delta \sum_{\substack{Q \in \mathscr{D} \\ Q \subseteq Q_0}} \beta_E^{d,p}(C_0B_Q)^2\ell(Q)^d,	
	\end{align}
	so we only need to control the sum over cubes in $\mathscr{C}.$ Let $\mathscr{C}_1$ be the set of cubes $Q \in \mathscr{C} \cap \bigcup_{k=K}^\infty \mathscr{D}_k$ so that $Q^{(K)}$ satisfies $\beta_E^{d,p}(C_0B_{Q^{(K)}}) \leq \delta$ and let $\mathscr{C}_2$ be the set of cubes $Q \in\mathscr{C} \cap  \bigcup_{k=K}^\infty \mathscr{D}_k$ so that $Q^{(K)}$ satisfies $\beta_E^{d,p}(C_0B_{Q^{(K)}}) > \delta$. Note that $Q^{(K)} \subseteq Q_0$ in either case. Then, 
	\begin{align}
		\sum_{Q \in \mathscr{C}} \beta_E^{d,p}(MB_Q)^2\ell(Q)^d  &\lesssim \sum_{k=0}^K \sum_{Q \in \mathscr{C} \cap \mathscr{D}_k} \beta_E^{d,p}(MB_Q)^2\ell(Q)^d + \sum_{Q \in \mathscr{C}_1} \beta_E^{d,p}(MB_Q)^2\ell(Q)^d \\
		&\hspace{4em} + \sum_{Q \in \mathscr{C}_2} \beta_E^{d,p}(MB_Q)^2\ell(Q)^d  \\
		&= I_0 + I_1 + I_2.  
	\end{align}
	Let us estimate $I_1.$ For $Q \in \mathscr{C}_1$ let $Q^* = Q^{(K)}$. Let $R \in \mathscr{D}$ be so that $\mathscr{C}_1(R) = \{Q \in \mathscr{C}_1 : Q^* = R \}$ is non-empty. By Lemma \ref{lemma:betap_betainfty}, $\beta^d_{E,\infty}(\tfrac{C_0}{2}B_R) \lesssim \delta^\frac{1}{d+1}$ so there exists a $d$-plane $P$ so that $\dist(y,P) \lesssim \ve \ell(R)$ for all $y \in E \cap \tfrac{C_0}{2}B_R.$ Then, for $\ve > 0$ small enough, a similar argument to Lemma \ref{l:reduce-constants1a} shows the map $Q \mapsto Q^*$ is at most $C$-to-$1$. Now, since $MB_Q \subseteq C_0B_{Q^*},$ Lemma \ref{lemma:betap_betainfty} implies 
	\begin{align}\label{e:consts2}
		I_1 \lesssim \sum_{Q \in \mathscr{C}_1} \beta^{d,p}_E(C_0B_{Q^*})^2\ell(Q^*)^d \lesssim \sum_{\substack{Q \in \mathscr{D} \\ Q \subseteq Q_0}}  \beta^{d,p}_E(C_0B_{Q})^2\ell(Q)^d
	\end{align} 
	Now for $I_2.$ For $Q \in \mathscr{C}_2$, define $Q^*$ to be the smallest cube so that $Q \subseteq Q^*$ and $\beta_E^{d,p}(C_0B_{Q^*}) > \delta.$ Let $R \in \mathscr{D}$ for which the set $\mathscr{C}_2(R) = \{Q \in \mathscr{C}_2 : Q^* = R \}$ is non-empty. We will show 
	\[ \#\mathscr{C}_2(R) \lesssim 1. \]
	Note, there are at most $K$ scales for which $\#\mathscr{C}_2(R) \cap \mathscr{D}_k \not=\emptyset.$ Choose $k$ to be one of these scales and let $Q \in \mathscr{C}_2(R) \cap \mathscr{D}_k.$ Since $R$ is the smallest cube containing $Q$ with $\beta_E^{d,p}(C_0B_R) > \delta$, there exists a child $R'$ of $R$ so that $\beta_E^{d,p}(C_0B_{R'}) \leq \delta.$ As above, since $B_R \subseteq C_0B_{R'}$ (because we assumed $C_0 \geq 2\rho^{-1}$), we can find a $d$-plane $P$ so that $\dist(x_Q,P) \leq \frac{c_0}{2}\ell(Q)$ and $c_0B_Q \subseteq B_R \subseteq B(x_R, \rho^{-K}\ell(Q)).$ Then, by Lemma \ref{ENV}, $\#\mathscr{C}_2(R) \cap \mathscr{D}_k \lesssim 1$ and so $\#\mathscr{C}_2(R) \lesssim 1.$ Hence, 
	\begin{align}
		I_2 \lesssim_\delta \sum_{Q \in \mathscr{C}_2} \beta_E^{d,p}(C_0B_{Q^*})^2\ell(Q^*)^d \lesssim \sum_{\substack{Q \in \mathscr{D} \\ Q \subseteq Q_0}}  \beta^{d,p}_E(C_0B_{Q})^2\ell(Q)^d. 
	\end{align}
	Finally, we deal with $I_0.$ Let $\delta > 0$ be as above. If $\beta_E^{d,p}(C_0B_{Q_0}) \leq \delta$ then the arguments above \eqref{e:consts2} show that $\#\{Q \in \mathscr{D}_k , 0 \leq k \leq K-1\} \lesssim 1.$ So, in this case,  
	\[ I_0 \lesssim \ell(Q_0)^d. \] 
	Suppose instead that $\beta_E^{d,p}(C_0B_{Q_0}) > \delta.$ For each $Q \in \bigcup_{k = 0}^{K-1} \mathscr{D}_k$ let $Q^*$ be the smallest cube so that $Q \subseteq Q^*$ and $\beta_E^{d,p}(C_0B_{Q^*}) > \delta.$ In this case, by exactly the same arguments as for $I_2,$ we get 
	\[ I_0 \lesssim_\delta \sum_{\substack{Q \in \mathscr{D} \\ Q \subseteq Q_0}} \beta_E^{d,p}(C_0B_Q)^2 \ell(Q)^d. \] 
	Combining the above gives 
	\[\sum_{\substack{Q \in \mathscr{D} \\ Q \subseteq Q_0}} \beta_E^{d,1}(MB_Q)^2\ell(Q)^d \lesssim \sum_{\substack{Q \in \mathscr{D} \\ Q \subseteq Q_0}} \beta_E^{d,p}(C_0B_Q)^2\ell(Q)^d \]
	as required. 
\end{proof}

\section{Reduction to Euclidean space}\label{a:Euclidean}

Here we prove Lemma \ref{l:main-red} and Lemma \ref{l:main-red2}. However, since Lemma \ref{l:main-red} is a special case of Lemma \ref{l:main-red2}, it suffices to only prove Lemma \ref{l:main-red2}, which we restate below for convenience. 

\begin{lem}\label{l:Euclidean}
	Suppose for any $n \geq 2$, $1 \leq d < n,$ $1 \leq p < p(d),$ $C_0 > 1$, and $A > \max\{6C_0,10^5\}$, that there exists $\ve > 0$ so that \eqref{e:dTSP1} holds with constants $d,p,C_0,A,$ and $\ve$, for any lower content $d$-regular set $E \subseteq \R^n$, with constants independent of $n$. Then, there exists $C \geq 1$ so that \eqref{e:dTSP1} holds with constant $d,p,C_0,2A$, and $C\ve,$ for all $E \subseteq H$ lower content $d$-regular sets $E \subseteq H.$ 
\end{lem}

Let $1 \leq d < \infty, \ 1 \leq  p < p(d)$, $C_0 > 1$, $A > \max\{6C_0,10^5\}$, $\ve > 0$, and $E \subseteq H$ be a lower content $d$-regular set. Let $X_m, \ m \in \Z,$ be a sequence of maximally $\rho^m$-separated nets and let $\mathscr{D}$ be the Christ-David cubes for $E$ from Theorem \ref{cubes}. Fix $Q_0 \in \mathscr{D}$ and assume without loss of generality that $x_{Q_0} = 0$ and $\ell(Q_0) = 1.$ We may also assume $\mathscr{H}^d(Q_0) < \infty$ since otherwise \eqref{e:dTSP1} is trivial. For $k \geq 0$ let 
\[ \mathscr{D}^k(Q_0) = \bigcup_{m = 0}^k \{Q \in \mathscr{D}_m : Q \subseteq Q_0 \}. \]
Notice that
\begin{align}\label{e:red1}
	\sum_{\substack{Q \in \mathscr{D} \\ Q \subseteq Q_0}} \beta^{d,p}_E(C_0B_Q)^2\ell(Q)^d = \lim_{k \rightarrow \infty} \sum_{Q \in \mathscr{D}^k(Q_0)}  \beta^{d,p}_E(C_0B_Q)^2\ell(Q)^d,
\end{align} 
so it suffices to show 
\begin{align}\label{e:red1.5}
	\sum_{Q \in \mathscr{D}^k(Q_0)}  \beta^{d,p}_E(C_0B_Q)^2\ell(Q)^d \lesssim \mathscr{H}^d(Q_0) +\BWGL(Q_0,A,\ve) 
\end{align} 
with constant independent $k$. This is our goal for the rest of the section. Let $k \in \Z$ and let $k^*$ be the smallest integer so that $\rho^{k^*} \leq 5\ve \rho^k.$ Since $E$ is lower content $d$-regular, $\# \mathscr{D}^{k^*}(Q_0) < \infty$, because for each $0 \leq m \leq k^*$, 
\begin{align}
	\#\{Q \in \mathscr{D}_m : Q \subseteq Q_0\} \rho^{md} \lesssim \sum_{\substack{Q \in \mathscr{D}_m \\ Q \subseteq Q_0}} \mathscr{H}^d(E \cap c_0B_Q) \leq \mathscr{H}^d(Q_0) < \infty.
\end{align}
Let $V$ be the linear space spanned by the centres of the cubes in $\mathscr{D}^{k^*}(Q_0)$ (recall $x_{Q_0} = 0$). By the above, we can identify $V$ with $\R^n$ for some $n \leq \# \mathscr{D}^{k^*}(Q_0) < \infty.$ We now define a subset $\tilde{E}$ of $\R^n$ which well approximates $E$. Let $T$ be a $d$-dimensional simplex centred at the origin, with vertices in $S^{n-1} \subseteq \R^n$ (the unit sphere). For $Q \in \mathscr{D}_{k^*},$ define $T_Q = x_Q + c_0\ell(Q)T.$ Then each $T_Q$ is Ahlfors $d$-regular and contained in the ball $B_n(x_Q,c_0\ell(Q)) = B_n(x_Q,5c_0\rho^{k^*}).$ Define $\tilde{E}$ by setting
\[ \tilde{E} = \bigcup_{Q \in \mathscr{D}_{k^*}} T_Q \subseteq \R^n. \] 

\begin{lem}
	The set $\tilde{E}$ is lower content $d$-regular 
\end{lem}

\begin{proof}
	Let $x \in \tilde{E}$ and $0 < r < \diam(\tilde{E})$. Assume first of all that $r < 20c_0\rho^{k^*}.$ Let $Q \in \mathscr{D}_{k^*}$ be the cube so that $x \in T_Q.$ Since $T_Q$ is lower content $d$-regular (by virtue of being Ahlfors $d$-regular) and $\diam(T_Q) = 10c_0\rho^{k^*} \gtrsim r$, this implies 
	\begin{align}\label{e:EtildeLR}
		\mathscr{H}^d_\infty(\tilde{E} \cap B(x,r)) \geq \mathscr{H}^d_\infty(T_Q \cap B(x,r)) \gtrsim r^d.
	\end{align}
	Assume instead that $r \geq 20c_0\rho^{k^*}.$ For $A \subseteq \R^n$ and $0 < \delta < \infty$, set  
	\begin{align}
		\mathscr{H}^d_{\delta,\infty}(A) = \inf \left\{ \sum_i \diam(A_i)^d : A \subseteq \bigcup_i A_i \mbox{ and } \diam A_i \geq \delta \right\}.
	\end{align}
	This is like the Hausdorff content except we include a \textit{lower} bound on the diameters of the covering sets. The lemma will be proved once we show the following,  
	\begin{align}\label{e:lowerboundcontent}
		\mathscr{H}^d_\infty(\tilde{E} \cap B(x,r)) \gtrsim \mathscr{H}^d_{5c_0\rho^{k^*},\infty}(\tilde{E} \cap B(x,r/2)) \gtrsim r^d.
	\end{align}
	We start with the left-hand side. First note that $\mathscr{H}^d_\infty(\tilde{E} \cap B(x,r)) > 0$ by \eqref{e:EtildeLR}. Thus, we can find a collection of sets $\{U_i\}_{i \in I}$ so that $\tilde{E} \cap B(x,r) \subseteq \bigcup_{i \in I} U_i$ and 
	\[ \sum_{i \in I } \diam(U_i)^d \leq 2\mathscr{H}^d_\infty(\tilde{E} \cap B(x,r)). \]
	We will modify $\{U_i\}$ to replace those sets with $\diam(U_i) < 5c_0\rho^{k^*}.$ Let $\mathscr{G}$ be the set of cubes in $\mathscr{D}_{k^*}$ so that $T_Q \cap B(x,r/2) \neq\emptyset$. For each $Q \in \mathscr{G},$ let $I_Q$ be the set of $i \in I$ so that $T_Q \cap U_i \not=\emptyset$ and $\diam(U_i) < 5c_0\rho^{k^*}.$ Let $\alpha > 0$ (to be chosen shortly) and further sub-divide $\mathscr{G}$ into 
	\begin{align}
		\mathscr{G}_1 &=  \left\{ Q \in \mathscr{G} : \sum_{i \in I_Q} \diam(U_i)^d > \alpha (10c_0)^d \rho^{dk^*} \right\} ; \\
		\mathscr{G}_2 &=  \left\{Q \in \mathscr{G} : \sum_{i \in I_Q} \diam(U_i)^d \leq \alpha (10c_0)^d \rho^{dk^*} \right\}.
	\end{align}  
	If $Q \in \mathscr{G}_1$ then $T_Q \subseteq B(x_Q,5c_0\rho^{k^*})$ and 
	\[ \diam(B(x_Q,5c_0\rho^{k^*}))^d = (10c_0)^d \rho^{dk^*}  < \frac{1}{\alpha} \sum_{i \in I_Q} \diam(U_i)^d. \] 
	For $Q \in \mathscr{G}_2$ we claim there exists $i_Q \in I$ so that $T_Q \cap U_{i_Q} \not=\emptyset$ and $\diam(U_{i_Q}) \geq 5c_0\rho^{k^*}.$ in particular $i_Q \not\in I_Q.$ Assuming this to be true for the moment, we let $\tilde{I} \subseteq I$ be the set of indices for which there exists $Q \in \mathscr{G}_2$ so that $i_Q = i$. For each $i \in \tilde{I}$ we pick a point $x_i \in U_i$ and set $B_i = B(x_i,2\diam(U_i)).$ In this way
	\[ \diam(B_i) \geq 20c_0\rho^{k^*} \mbox{ and } \bigcup_{\substack{Q \in \mathscr{G}_2 \\ i_Q = i}} T_Q \subseteq B_i. \]
	Putting everything together we have
	\[ \tilde{E} \cap B(x,r/2) \subseteq \bigcup_{Q \in \mathscr{G}} T_Q \subseteq \bigcup_{Q \in \mathscr{G}_1} B(x_Q,5c_0\rho^{k^*}) \cup \bigcup_{i \in \tilde{I}} B_{i}, \]      
	hence, 
	\begin{align}
		\mathscr{H}^d_{5c_0\rho^{k^*},\infty}(\tilde{E} \cap B(x,r/2)) &\leq \sum_{Q \in \mathscr{G}_1} \diam(B(x_Q,5c_0\rho^{k^*}))^d + \sum_{i \in \tilde{I}} \diam(B_i)^d \\
		& \leq \left(\frac{1}{\alpha} + 2 \right) \sum_{i \in I} \diam(U_i)^d \lesssim \mathscr{H}^d(\tilde{E} \cap B(x,r)). 
	\end{align}
	This completes the left-hand side of \eqref{e:lowerboundcontent} modulo our claim. Let us see why the claim is true. Since $T_Q \cap B(x,r/2) \not=\emptyset, \ \diam(T_Q) =10c_0\rho^{k^*} \leq r/2,$ we have $T_Q \subseteq \tilde{E} \cap B(x,r).$ In particular, the collection of set $\{U_i : i \in I, \ T_Q \cap U_i\neq \emptyset\}$ form a cover for $T_Q.$ Then, since $T_Q$ is lower content $d$-regular, we have
	\begin{align}
		\rho^{dk^*} \lesssim \sum_{\substack{i \in I \\ T_Q \cap U_i \not=\emptyset}} \diam(U_i)^d = \sum_{\substack{i \in I\setminus I_Q \\ T_Q \cap U_i \not=\emptyset}} \diam(U_i)^d + \sum_{\substack{i \in I_Q}} \diam(U_i)^d \lesssim \sum_{\substack{i \in I\setminus I_Q \\ T_Q \cap U_i \not=\emptyset}} \diam(U_i)^d + \alpha \rho^{dk^*}.
	\end{align}   
	For $\alpha > 0$ small enough, the penultimate term must be non-zero which implies the existence of the required index $i \in I$ and proves the claim.

	Now for the right-hand side of \eqref{e:lowerboundcontent}. Let $\{V_j\}_{j \in J}$ be a collection of subsets of $\R^n$ so that $5c_0\rho^{k^*} \leq \diam(V_j) < \infty$, $\tilde{E} \cap B(x,r/2) \subseteq \bigcup_{j \in J} V_j$ and 
	\[ \mathscr{H}^d_{5c_0\rho^{k^*},\infty}(\tilde{E} \cap B(x,r/2)) \sim \sum_{j \in J} \diam(V_j)^d. \] 
	For each $j \in J,$ let $y_j \in V_j$ and $\tilde{B}_j = B(y_j,c_0^{-1}\diam(V_j)) \subseteq H.$ For each $Q \in \mathscr{D}_{k^*}$ so that $x_Q \in B(x,r/2)$, there exists $j \in J$ so that $x_Q \in V_j,$ which by the triangle inequality and the lower bound on $\diam(V_j)$, implies $B(x_Q,2\rho^{k^*}) \subseteq \tilde{B}_j.$ In particular, 
	\[ \bigcup_{\substack{Q \in \mathscr{D}_{k^*} \\ x_Q \in B(x,r/2) }} B(x_Q,2\rho^{k^*}) \subseteq \bigcup_{j \in J} \tilde{B}_j. \] 
	By maximality, the collection $\{B(x_Q,2\rho^{k^*}) : Q \in \mathscr{D}_{k^*}, \  x_Q \in B(x,r/2) \}$ forms a cover of $E \cap B(x,r/4)$. By the above, the same is true of the collection of balls $\{\tilde{B}_j\}_{j \in J}.$ Since $E$ is lower content $d$-regular, this implies 
	\begin{align}
		r^d \lesssim \mathscr{H}^d_\infty(E \cap B(x,r/4)) \leq \sum_{j \in J} \diam(\tilde{B}_j)^d \sim \sum_{j \in J} \diam(V_j)^d \sim \mathscr{H}^d_{5c_0\rho^{k^*},\infty}(\tilde{E} \cap B(x,r/2))
	\end{align}
	as required. 
\end{proof}

Let $Y_m, \ m \in \Z$, be a sequence of nested maximally $\rho^m$-separated nets in $\tilde{E}.$ We can construct these so that $Y_m = X_m$ for each $m \leq k^*$ (recall that $X_m$ were the nets we used to construct the cubes for $E$). Let $\tilde{\mathscr{D}}$ be the Christ-David cubes for $\tilde{E}$ constructed from the $Y_m.$ In this way, for each $Q \in \mathscr{D}^{k^*}$ there exists a unique $\tilde{Q} \in \tilde{\mathscr{D}}^{k^*}$ so that $x_Q = x_{\tilde{Q}}$ and $\ell(Q) = \ell(\tilde{Q}),$ and vice-versa. These also satisfy 
\begin{align}\label{e:Y_m}
	Y_m \cap Q = Y_m \cap \tilde{Q}  \mbox{ for all } m \leq k^*.
\end{align}   

\begin{lem}\label{l:measure}
	We have $\mathscr{H}^d(\tilde{Q}_0) \lesssim \mathscr{H}^d(Q_0).$ 
\end{lem} 

\begin{proof}
	First observe that 
	\begin{align}\label{e:tilde{Q}_0}
		\tilde{Q}_0 \subseteq \bigcup_{\substack{Q \in \mathscr{D}_{k^*} \\ Q \subseteq Q_0}} T_Q.  
	\end{align}
	Indeed, suppose the opposite was true. Then we could find a point $z \in \tilde{Q}_0$ so that $z \in T_Q$ for some $Q \in \mathscr{D}_{k^*}$ satisfying $Q \cap Q_0 =\emptyset.$ Let $\tilde{Q}$ be the corresponding cube in $\tilde{\mathscr{D}}.$ Then $x_{\tilde{Q}} \not\in Q_0$ and so $x_{\tilde{Q}} \not\in \tilde{Q}_0$ by \eqref{e:Y_m}. Thus $\tilde{Q} \cap \tilde{Q}_0 = \emptyset.$ This, however, is a contradiction since 
	\[ z \in T_Q \subseteq \tilde{E} \cap c_0B_Q = \tilde{E} \cap c_0B_{\tilde{Q}} \subseteq \tilde{Q}. \] 
	
	Continuing from \eqref{e:tilde{Q}_0}, since each of the $T_Q, \ Q \in \mathscr{D}_{k^*},$ are Ahlfors $d$-regular, and the balls $E \cap c_0B_Q$ are disjoint and contained in $Q_0$, we have 
	\begin{align}
		\mathscr{H}^d(\tilde{Q}_0) \leq \sum_{\substack{Q \in \mathscr{D}_{k^*} \\ Q \subseteq Q_0}} \mathscr{H}^d(T_Q) \lesssim \sum_{\substack{Q \in \mathscr{D}_{k^*} \\ Q \subseteq Q_0}} \ell(Q)^d \lesssim \sum_{\substack{Q \in \mathscr{D}_{k^*} \\ Q \subseteq Q_0}} \mathscr{H}^d(E \cap c_0B_Q) \leq \mathscr{H}^d(Q_0)
	\end{align} 
\end{proof}

\begin{lem}\label{l:BWGL}
	Let $\tilde{Q} \in \mathscr{D}^{k}(\tilde{Q}_0)$ and let $Q \in \mathscr{D}^{k}(Q_0)$ be the corresponding cube from above. If $\tilde{Q} \in \BWGL(\tfrac{1}{2}A,c\ve)$ then $Q \in \BWGL(A,\ve).$
\end{lem}

\begin{proof}
	We will prove the contrapositive. In particular, let us suppose that $Q \not\in \BWGL(A,\ve)$, we will show $\tilde{Q} \not\in \BWGL(\tfrac{1}{2}A,c\ve).$ 
	
	For brevity, we write $B = AB_Q = AB_{\tilde{Q}}.$ Since $Q \not\in \BWGL(A,\ve),$ there exists a $d$-plane $P_Q$ through $x_Q$ so that $d_{B}(E,P_Q) \leq 2\ve.$ Choose a collection of point $z_i, \ i = 1,\dots,d$, in $P_Q \cap (B \setminus \tfrac{1}{2}B)$ so that the vectors $z_i - x_Q$ are mutually orthogonal. Since $d_{B}(E,P_Q) \leq 2\ve,$ there exists $y_i \in E \cap \tfrac{3}{4}B$ so that $|y_i - z_i| \leq 2\ve A \ell(Q)$. By maximality, there exists $Q_i \in \mathscr{D}_{k^*}$ so that $x_{Q_i} \in B$ and $|y_i - x_{Q_i}| \leq \ell(Q_i) = 5\rho^{k^*} \leq \ve \ell(Q).$ Let $P_{\tilde{Q}} \subseteq \R^n$ be the $d$-plane spanned by the vectors $x_{Q_i} - x_Q.$ Since, by the triangle inequality, $|z_i - x_{Q_i}| \lesssim \ve \ell(Q)$, it follows that
	\begin{align}\label{e:nearend}
		d_{B}(P_Q,P_{\tilde{Q}}) \lesssim \ve.
	\end{align} 
	
	Let us now estimate $d_{B/2}(\tilde{E},P_{\tilde{Q}}).$ Let $y \in \tilde{E} \cap \tfrac{1}{2}B$, we will find $z \in P_{\tilde{Q}} \cap  \tfrac{1}{2}B$ so that $|y-z| \lesssim \ve \ell(Q).$ First, since the centres of cubes in $\mathscr{D}_{k^*}$ are contained in $E$ and form a maximally $\rho^{k^*}$ separated net for $\tilde{E},$ we can find $y_1 \in E \cap (\tfrac{1}{2} + \ve)B$ so that $|y-y_1| \leq \rho^{k^*} \leq \ve \ell(Q).$ Since $d_{B}(E,P_Q) < 2\ve$, there exists $y_2 \in P_Q \cap (\tfrac{1}{2} + 3\ve)B$ so that $|y_1 - y_2| \leq 2\ve A \ell(Q).$ By \eqref{e:nearend}, we can find $y_3 \in P_{\tilde{Q}} \cap (\tfrac{1}{2} + C \ve)B$ so that $|y_2 - y_3| \lesssim \ve\ell(Q).$ Finally, by considering the line segment joining $y_3$ and $x_{\tilde{Q}}$, we can find a point $z \in P_{\tilde{Q}} \cap \tfrac{1}{2}B$ so that $|y_3 - y_z| \leq CA\ve \ell(Q).$ Combining the above chain of inequalities shows $z \in P_{\tilde{Q}} \cap \tfrac{1}{2}B$ is the required point. 
	
	A similar argument (essentially working backwards) shows that for each $z \in P_{\tilde{Q}} \cap \tfrac{1}{2}B$ we can find $y \in \tilde{E} \cap \tfrac{1}{2}B$ so that $|y-z| \lesssim \ve \ell(Q).$ This concludes the argument.

\end{proof}

We can now prove Lemma \ref{l:Euclidean}. 

\begin{proof}[Proof of Lemma \ref{l:Euclidean}] 
	Recall that it suffices to show \eqref{e:red1.5}. Let $\mathscr{D}_{\text{Bad}}$ denote the set of cubes $Q$ in $\mathscr{D}^k(Q_0)$ for which there exists a point $y \in E \cap 3C_0B_Q \setminus Q_0,$ and let $\mathscr{G} = \mathscr{D}^k(Q_0) \setminus \mathscr{D}_{\text{Bad}}.$ Then, the left-hand side of \eqref{e:red1.5} can be written as 
	\[ \sum_{Q \in \mathscr{D}_{\text{Bad}}}\beta^{d,p}_E(C_0B_Q)^2\ell(Q)^d + \sum_{Q \in \mathscr{G}} \beta^{d,p}_E(C_0B_Q)^2\ell(Q)^d.\]
	
	Let us estimate the sum over $\mathscr{D}_{\text{Bad}}$. First, we claim that the collection of balls $\tfrac{c_0}{2}B_Q,$ $Q \in \mathscr{D}_{\text{Bad}}$, have bounded overlap. To see this, suppose $Q,R \in \mathscr{D}_\text{Bad}$ are such that $\tfrac{c_0}{2}B_Q \cap \tfrac{c_0}{2}B_R \neq \emptyset.$ The claim will follow once we show $\ell(Q) \sim \ell(R)$ since the balls $c_0B_Q$ are disjoint for cubes in the same generation. Assume without loss of generality that $\ell(Q) \leq \ell(R).$ Assume towards a contradiction that $\ell(Q) \leq \tfrac{c_0}{12C_0} \ell(R).$ It follows that $E \cap 3C_0B_Q \subseteq E \cap c_0B_R$ since for any $y \in 3C_0B_Q,$  
	\begin{align}
		|y - x_R| \leq |y-x_Q| + |x_Q - x_R| \leq 3C_0\ell(Q) + \frac{c_0}{2} \ell(Q) + \frac{c_0}{2}\ell(R) \leq c_0\ell(R).
	\end{align} 
	But this means that $E \cap 3C_0B_Q \subseteq Q_0$ which contradicts the definition of $\mathscr{D}_\text{Bad}.$ Using this, we have 
	\begin{align}
		\sum_{Q \in \mathscr{D}_\text{Bad}} \beta^{d,p}_E(C_0B_Q)^2\ell(Q)^d \lesssim \sum_{Q \in \mathscr{D}_\text{Bad}} \ell(Q)^d \lesssim \sum_{Q \in \mathscr{D}_\text{Bad}} \mathscr{H}^d(E \cap c_0B_Q) \lesssim \mathscr{H}^d(Q_0). 
	\end{align}  
	
	We turn our attention to the sum over $\mathscr{G}.$ It will be useful to note (for later on) that if $Q \in \mathscr{G}$ then $E \cap 3C_0B_Q \subseteq Q_0.$ We begin by applying Lemma \ref{betaest},  
	
	\begin{align}
		\sum_{Q \in \mathscr{G}} \beta^{d,p}_E(C_0B_Q)^2\ell(Q)^d &\lesssim \sum_{Q \in \mathscr{G}} \beta^{d,p}_{\tilde{E}}(2C_0B_Q)^2\ell(Q)^d \\
		&\hspace{4em}+ \sum_{Q \in \mathscr{G}} \left( \frac{1}{\ell(Q)^d} \int_{E \cap 2C_0B_Q} \left(\frac{\dist(y,\tilde{E})}{\ell(Q)} \right)^p \, d\mathscr{H}^d_\infty \right)^\frac{2}{p} \ell(Q)^d \\
		&\eqqcolon I_1 + I_2. 
	\end{align}
	By construction, for each $Q \in \mathscr{G}$ there exists a unique $\tilde{Q} \in\tilde{\mathscr{D}}$ so that $2C_0B_Q = 2C_0B_{\tilde{Q}}.$ Using this along with Theorem \ref{Thm3}, Lemma \ref{l:measure}, and Lemma \ref{l:BWGL}, we have
	\begin{align}
		I_1  \leq \sum_{m=0}^k\sum_{Q \in \tilde{\mathscr{D}}_m}\beta^{d,p}_{\tilde{E}}(2C_0B_Q)^2\ell(Q)^d &\lesssim \mathscr{H}^d(\tilde{Q}_0) + \sum_{\substack{Q \in \BWGL(\frac{1}{2}A,c\ve) \cap \tilde{\mathscr{D}}^k(\tilde{Q}_0)}} \ell(Q)^d \\ 
		&\lesssim \mathscr{H}^d(Q_0) +\sum_{\substack{Q \in \BWGL(A,\ve) \cap \mathscr{D}^k(Q_0)}} \ell(Q)^d  \\
		&\leq \mathscr{H}^d(Q_0) + \BWGL(Q_0,A,\ve).
	\end{align}
	We are left to estimate $I_2.$ We claim 
	\begin{align}\label{e:I_2last}
		I_2 \lesssim \mathscr{H}^d(Q_0) + \BWGL(Q_0,A,\ve).
	\end{align}
	Indeed, notice that
	\begin{align}
		I_2 \leq \sum_{Q \in \mathscr{G}\setminus \BWGL(A,\ve)} \left( \frac{1}{\ell(Q)^d} \int_{E \cap 2C_0B_Q} \left(\frac{\dist(y,\tilde{E})}{\ell(Q)} \right)^p \, d\mathscr{H}^d_\infty \right)^\frac{2}{p} \ell(Q)^d  + \BWGL(Q_0,A,\ve) \\ 
	\end{align}
	Denoting the first term above by $I_3,$ the proof of \eqref{e:I_2last} will be complete once we show
	\[ I_3 \lesssim \mathscr{H}^d(Q_0). \]
	Let $\mathscr{G}' = \mathscr{G} \setminus \BWGL(A,\ve)$ and for $m \in \Z$ set $\mathscr{G}_m' = \mathscr{G}' \cap \mathscr{D}_m.$ First, we may assume $p \geq 2$ by Lemma \ref{c:beta}.  Notice, if $y \in 2C_0B_R$ for some $R \in \mathscr{D}_k$ then $\dist(y,\tilde{E}) \lesssim \ell(R)$ since $x_R \in \tilde{E}.$ Hence, 
	\begin{align}\label{e:red2}
		I_3 &\lesssim \sum_{m=0}^{k} \sum_{Q \in \mathscr{G}_m'} \left( \frac{1}{\ell(Q)^d} \sum_{\substack{R \in \mathscr{D}_k \\ B_R \cap 2C_0B_Q \not=\emptyset}}  \int_{E \cap B_R}  \left(\frac{\dist(y,\tilde{E})}{\ell(Q)} \right)^p \, d\mathscr{H}^d_\infty \right)^\frac{2}{p} \ell(Q)^d \\
		&\lesssim \sum_{m=0}^k \sum_{Q \in \mathscr{G}_m'} \left( \sum_{\substack{R \in \mathscr{D}_k \\ B_R \cap 2C_0B_Q \not=\emptyset}} \frac{\ell(R)^{d+p}}{\ell(Q)^{d+p}}  \right)^\frac{2}{p} \ell(Q)^d \lesssim \sum_{m=0}^k \sum_{Q \in \mathscr{G}_m'} \sum_{\substack{R \in \mathscr{D}_k \\ B_R \cap 2C_0B_Q \not=\emptyset}} \frac{\ell(R)^{d\frac{2}{p} +2}}{\ell(Q)^{d(\frac{2}{p} -1) +2}}
	\end{align}  
	
	Let $R \in \mathscr{D}_k$. For $m \leq k$, let $\mathscr{G}_{R,m}'$ be the set of cubes $Q$ in $\mathscr{G}_m'$ so that $B_R \cap 2C_0B_Q \not=\emptyset.$ We will show 
	\begin{align}\label{e:red4}
		\# \mathscr{G}_{R,m}' \lesssim 1.
	\end{align}
	If $\mathscr{G}_{R,m}'$ is empty then there is nothing to show. Let us assume it is not empty, and fix some $Q^* \in \mathscr{G}_{R,m}'.$ For any other $Q \in \mathscr{G}_{R,m}'$, since $B_R \cap 2C_0B_Q \neq \emptyset$ and $\ell(R) \leq \ell(Q),$ it follows that $c_0B_Q \subseteq B \coloneqq 6C_0B_{Q^*}.$ Also, since $A > 6C_0$ and $Q \not\in \BWGL(A,\ve)$ (by definition of $\mathscr{G}'$) we can find a $d$-plane $P$ through $x_B$ so that $d_B(E,P) \lesssim \ve.$ Thus, for $\ve > 0$ small enough, the balls $c_0B_Q, \ Q \in \mathscr{G}_{R,m}',$ satisfy the conditions of Lemma \ref{ENV}, and we conclude \eqref{e:red4} since 
	\begin{align}
		\# \mathscr{G}_{R,m}' \rho^{md} \lesssim \sum_{Q \in \mathscr{G}_{R,m}'} \ell(Q)^d \lesssim \rho^{md}.
	\end{align}
	
	Let us return to $I_3.$ Notice, if $R \in \mathscr{D}_k$, $m \leq k$, and $Q \in \mathscr{G}'_m$ is such that $B_R \cap 2C_0B_Q \not=\emptyset$, then $E \cap B_R \subseteq E \cap 3C_0B_Q \subseteq Q_0.$ Thus, continuing from \eqref{e:red2} by swapping the order of summation, using \eqref{e:red4}, and summing over a geometric series, we have 
	\begin{align}
		I_3 \lesssim  \sum_{\substack{R \in \mathscr{D}_k \\ R \subseteq Q_0}} \sum_{m=0}^k \sum_{\substack{Q \in \mathscr{G}_m' \\ B_R \cap 2C_0B_Q \not=\emptyset}} \frac{\ell(R)^{d\frac{2}{p} +2}}{\ell(Q)^{d(\frac{2}{p} -1) +2}} \lesssim \sum_{\substack{R \in \mathscr{D}_k \\ R \subseteq Q_0}} \ell(R)^d \lesssim  \sum_{\substack{R \in \mathscr{D}_k \\ R \subseteq Q_0}} \mathscr{H}^d(E \cap c_0B_R) \leq \mathscr{H}^d(Q_0). 
	\end{align}

\end{proof}

\bibliography{Ref}
\bibliographystyle{alpha}

\end{document}